\documentclass[11pt,a4paper]{article}
\pdfoutput=1

\usepackage[a4paper,
            top=2.5cm,
            bottom=2.5cm,
            inner=3.5cm,
            outer=3.5cm]{geometry}

\usepackage[utf8]{inputenc}
\usepackage{amsmath}
\usepackage{amsfonts}
\usepackage{amssymb}
\usepackage{amsthm}
\usepackage{enumerate}
\usepackage{mathtools}
\usepackage[nottoc]{tocbibind}
\usepackage{hyperref} 
\usepackage{graphicx}
\usepackage[dvipsnames]{xcolor}
\usepackage[english]{babel}
\usepackage{microtype}
\usepackage[all,cmtip]{xy}
\usepackage{paralist} 
\usepackage{tabularx}
\usepackage{caption}

\theoremstyle{plain}
\newtheorem{thm}{Theorem}[section]
\newtheorem{lem}[thm]{Lemma}
\newtheorem{prop}[thm]{Proposition}
\newtheorem{cor}[thm]{Corollary}
\newtheorem{question}[thm]{Question}

\newtheorem{claim}[thm]{Claim}

\newtheorem{thmintro}{Theorem}

\newtheorem*{exintro}{Examples}
\newtheorem*{defnintro}{Definition}

\theoremstyle{definition}
\newtheorem*{remintro}{Remark}

\newtheorem{defn-alt}[thm]{Definition}
\newtheorem{example-alt}[thm]{Example}
\newtheorem{examples-alt}[thm]{Examples}
\newtheorem{rem-alt}[thm]{Remark}
\newtheorem*{rem-altstern}{Remark}
\newtheorem{nota-alt}[thm]{Notation}

\newtheorem*{acknow}{Acknowledgements}

\newenvironment{defn}    
{
	\pushQED{\qed}\begin{defn-alt}}
	{\popQED\end{defn-alt}}

\newenvironment{example}    
{
	\pushQED{\qed}\begin{example-alt}}
	{\popQED\end{example-alt}}
	
\newenvironment{examples}    
{
	\pushQED{\qed}\begin{examples-alt}}
	{\popQED\end{examples-alt}}

\newenvironment{rem}    
{
	\pushQED{\qed}\begin{rem-alt}}
	{\popQED\end{rem-alt}}

\newcommand{\N}{\mathbb{N}}
\newcommand{\IN}{\mathbb{N}}
\newcommand{\Z}{\mathbb{Z}}
\newcommand{\IZ}{\mathbb{Z}}

\newcommand{\R}{\mathbb{R}}
\newcommand{\IR}{\mathbb{R}}
\newcommand{\C}{\mathbb{C}}
\newcommand{\IC}{\mathbb{C}}

\newcommand{\cO}{\mathcal{O}}
\newcommand{\cP}{\mathcal{P}}
\newcommand{\cF}{\mathcal{F}}
\newcommand{\cX}{\mathcal{X}}
\newcommand{\cY}{\mathcal{Y}}
\newcommand{\cA}{\mathcal{A}}
\newcommand{\cB}{\mathcal{B}}
\newcommand{\cU}{\mathcal{U}}
\newcommand{\cV}{\mathcal{V}}

\newcommand{\cK}{\mathcal{K}}
\newcommand{\cL}{\mathcal{L}}

\newcommand{\combcompb}[2]{{\overline{#1}\vphantom{#1}^{#2}}}

\newcommand{\EG}{\underline{EG}}

\newcommand{\HX}{H\! X}
\newcommand{\CX}{C\! X}
\newcommand{\HA}{H\! A}

\newcommand{\KX}{K\! X}
\newcommand{\EX}{E\! X}

\newcommand{\LocComp}{\mathbf{LCH}}
\newcommand{\sigLocComp}{\boldsymbol\sigma\mathbf{LCH}}
\newcommand{\Comp}{\mathbf{CH^2}}
\newcommand{\sigComp}{\boldsymbol\sigma\mathbf{CH^2}}

\newcommand{\Cech}{\v{C}ech }

\newcommand{\neighbsub}[1]{\stackrel{#1}{\Subset}}

\newcommand{\proofsubdivisor}[1]{\emph{#1}}

\DeclareMathOperator{\im}{im}
\DeclareMathOperator{\id}{id}
\DeclareMathOperator{\cd}{cd}

\DeclareMathOperator{\diam}{diam}
\DeclareMathOperator{\dist}{dist}
\DeclareMathOperator{\Var}{\operatorname{Var}}
\DeclareMathOperator{\codim}{codim}
\DeclareMathOperator{\asdim}{asdim}
\DeclareMathOperator{\supp}{supp}
\newcommand{\diag}{\mathrm{diag}}
\newcommand{\free}{{\mathrm{free}}}
\newcommand{\ER}{{E\! R}}
\newcommand{\AR}{{A\! R}}
\newcommand{\CAT}{\mathrm{CAT}}

\newcommand{\EZ}{\mathit{EZ}}
\newcommand{\lf}{\mathit{lf}}

\numberwithin{equation}{section} 

\author{
Alexander Engel\thanks{
Fakult{\"a}t f{\"u}r Mathematik, Universit{\"a}t Regensburg, 93040 Regensburg, GERMANY\newline
alexander.engel@mathematik.uni-regensburg.de}
\and
Christopher Wulff\thanks{
Mathematisches Institut, Georg-August-Universität Göttingen, Bunsenstr.~3-5, D-37073 Göttingen, GERMANY
\newline
christopher.wulff@mathematik.uni-goettingen.de}
}
\title{Coronas for properly combable spaces}
\date{}

\begin{document}

\maketitle

\begin{abstract}
This paper is a systematic approach to the construction of coronas (i.\,e.~Higson dominated boundaries at infinity) of combable spaces.
We introduce three additional properties for combings: properness, coherence and expandingness. 
Properness is the condition under which our construction of the corona works.
Under the assumption of coherence and expandingness, attaching our corona to a Rips complex construction yields a contractible $\sigma$-compact space in which the corona sits as a $Z$-set.
This results in bijectivity of transgression maps, injectivity of the coarse assembly map and surjectivity of the coarse co-assembly map. 
For groups we get an estimate on the cohomological dimension of the corona in terms of the asymptotic dimension.
Furthermore, if the group admits a finite model for its classifying space $BG$, then our constructions yield a $Z$-structure for the group.
\end{abstract}

\tableofcontents

\section{Introduction}

Boundaries at infinity became an indispensable tool for the investigation of non-positively curved spaces. Important examples are the visual boundaries of $\CAT(0)$-spaces (Eberlein--O'Neill \cite{eberlein_oneill}, Bridson--Haefliger \cite{bridson_haefliger}) and the Gromov boundary of hyperbolic spaces (Gromov \cite{gromov_hyperbolic_groups}).

In the above mentioned examples the boundary has an important property: in the case of $\CAT(0)$-spaces, the space together with the boundary becomes contractible. And in the case of hyperbolic spaces, if we attach the boundary to suitable Rips complexes of them, then the resulting compact spaces are also contractible. This was used by Bestvina--Mess \cite{bestvina_mess} to start, in the case of hyperbolic groups, the investigation of homological properties of the boundary in relation to the homological properties of the group itself.

Bestvina \cite{bestvina} initiated the systematic study of boundaries of groups, where he introduced $Z$-structures for this purpose (a generalization of the above discussed structures in the $\CAT(0)$ and hyperbolic case). These are contractible spaces $X$ (for technical reasons he demanded them to be Euclidean retracts) on which the group acts properly, cocompactly and freely and such that $X$ admits a compactification so that the resulting compact space is also contractible.\footnote{One can generalize the notion of a $Z$-structure by dropping the freeness assumption for the action and requiring $X$ to be only an absolute retract, see Dranishnikov \cite{dranish_BM}.} For certain reasons one also needs the $Z$-property for the boundary at infinity, i.e., that one can homotope the boundary instantly off into the space $X$.

These $Z$-structures became more important when it was realized that an equivariant version of them (where the group action on $X$ extends continuously to the boundary) enables one to deduce the Novikov conjecture for the group and related injectivity results for other isomorphism conjectures (Roe \cite{roe_index_coarse}, Carlsson--Pedersen \cite{carlsson_pedersen,carlsson_pedersen_2}, Carlsson--Pedersen--Vogell \cite{carlsson_pedersen_vogell} and Farrell--Lafont \cite{farrell_lafont}).

In \cite[Section 3.1]{bestvina} Bestvina writes \emph{``There seems to be no systematic method of constructing boundaries of groups in general [...]''} and this question was phrased again by Farrell--Lafont \cite[Remark 1]{farrell_lafont}. The goal of this paper is to present a construction method for such compactifications based upon proper combings.

\begin{acknow}
Both authors were supported by the DFG Priority Programme \emph{Geometry at Infinity} (SPP 2026, EN 1163/3-1 and WU 869/1-1 \emph{Duality and the coarse assembly map}). The first named author was further supported by the SFB 1085 \emph{Higher Invariants} and the Research Fellowship EN 1163/1-1 \emph{Mapping Analysis to Homology}, both also funded by the DFG. The second author was supported by the program of postdoctoral scholarships of the Universidad Nacional Aut\'onoma de M\'exico (UNAM).

We heartily thank Mladen Bestvina, Martin Bridson, Elia Fioravanti, Thomas Haettel, Ursula Hamenstädt, Robert Howlett, Ralf Meyer, Krishnendu Khan, Narutaka Ozawa, Alessandro Sisto, Jacek \'{S}wi{\c{a}}tkowski, Rufus Willett and Jianchao Wu for answering our questions and giving helpful comments. Furthermore, we thank the participants of the MFO workshop ``Nonpositively curved complexes'' (NPCC) 2021 for the discussions after giving a talk about the results of this paper.

Finally, the first named author is very grateful for the hospitality of the UNAM, in whose inspirational environment the foundations of this paper were laid out.
\end{acknow}

\subsection{Proper combings and combing compactifications}

Let $X$ be a coarse space.
\begin{defnintro}[{cf.~Definition \ref{def_combing}}]
 By a \emph{combing} on $X$ starting at a point $p\in
X$ we mean a map
\[H\colon X\times\N\to X\,,\qquad (x,n)\mapsto H_n(x)\coloneqq H(x,n)\]
such that\begin{enumerate}
\item\label{nbv} $H(x,0)=p=H(p,n)$ for all $x\in X$ and $n\in\N$,
\item\label{nbgfdc} for each bounded subset $K\subset X$ there is
an $N\in\N$ such that for all $n\geq N$ and $x\in K$ we have $H(x,n)=x$, and
\item\label{hztredfg} $H$ is a controlled map.
\end{enumerate}
\end{defnintro}
One should regard $H(x,-)$ as a path from the point $p$ at time $0$ to the point~$x$ which eventually becomes constant (the latter is encoded in the Point~\ref{nbgfdc} of the definition). Point~\ref{hztredfg}, i.e., controlledness of $H$, says that the step-size from $H(x,n)$ to $H(x,n+1)$ is uniformly bounded and that we have the so-called fellow-traveling-property (i.e., paths to nearby points $x$ and $y$ stay uniformly close to each other). If $G$ is a finitely generated group endowed with the usual coarse structure, then this notion of a combing is equivalent to the usual notion used in the geometric group theory literature (they are sometimes called synchronous combings or also bounded combings).

We want to use the combing to define a corona for the space. The usual approach used in the literature (which up to now only works for quasi-geodesic combings with further restrictive properties) is to consider quasi-geodesic rays starting at a fixed point and going to infinity, or alternatively---as an approximation to such rays---sequences of quasi-geodesic segments starting at a fixed point which are uniformly close to each other and whose lenghts go to infinity.
Then one defines an equivalence relation on these rays or sequences of segments (they are equivalent if they stay uniformly close) and its equivalence classes constitute the point-set model for the boundary at infinity. It remains to define the topology and to show that one has created a compact space.

Our approach is, though in spirit essentially the same, in the concrete implementation radically different: if $X$ is a proper topological coarse space, we define in Definition~\ref{def_corona_compactification} the commutative $C^\ast$-algebra $C_H(X)$ to consist of all complex-valued, bounded, continuous functions $f \colon X \to \IC$ satisfying two properties:
\begin{itemize}
\item $f$ has vanishing variation, and \item $H_n^\ast f \xrightarrow{n \to \infty} f$.
\end{itemize}
Since this $C^\ast$-algebra is commutative, it has a Gelfand dual $\combcompb{X}{H}$ (also called the maximal ideal space or the character space of $C_H(X)$). Since $C_H(X)$ is unital, $\combcompb{X}{H}$ is compact. We want this to be our compactification of $X$, which is the case if and only if $C_0(X) \subset C_H(X)$, but this requires an additional property from $H$ (cf.~Definition~\ref{defn_proper_combing}):
\begin{itemize}
\item $H$ is called \emph{proper} if for any bounded subset $K\subset X$ there is a bounded subset
$L\subset X$ and an $N\in\N$ such that $H_n^{-1}(K)\subset L$ for all $n\geq N$.
\end{itemize}

The above procedure automatically provides us also the topology of the corona and that it is a compactification of the original space $X$. In case $X$ was a proper metric space, the combing compactification will be metrizable by a straight-forward and short argument (Lemma~\ref{lem:metrizability}).

\begin{figure}[ht]
\centering
\begin{tabularx}{\textwidth}{|l|X|}
\hline
\textbf{Space} & \textbf{Corona}\\
\hline
Cones &\\
\quad open cone over compact $B$ & the base space $B$ itself\\
\quad foliated cone over $(M,\mathcal{F})$ & Hausdorffization of space of leaves $M/\mathcal{F}$\\
\quad warped cone of $G \curvearrowright M$ & Hausdorffization of orbit space $M/G$\\
\hline
Coarsely convex spaces & Fukaya--Oguni's boundary\\
\quad $\CAT(0)$ space & visual boundary\\
\quad hyperbolic space & Gromov boundary\\
\quad systolic complex & Osajda--Przytycki's boundary\\
\quad injective metric spaces & Descombes--Lang's boundary\\
\quad hierarchically hyperbolic spaces & Durham--Hagen--Sisto's boundary ??\footnotemark \\
\hline
\end{tabularx}
\caption{Examples of properly combable spaces and their coronas}
\label{fig_examples_combable_spaces}
\end{figure}
\footnotetext{See the discussion in Example \ref{examples_coarsely_convex}.\ref{examples_coarsely_convex_iii}.}

\subsection{Coherent and expanding combings}

As explained earlier at the beginning of the introduction, our goal is to get a contractible combing compactification $\combcompb{Y}{H}$ in which the corona sits as a so-called $Z$-set. For this we first need that $Y$ itself is contractible, otherwise there is no reason to expect that $\combcompb{Y}{H}$ will be contractible.

Given any discrete, proper metric space $X$, an easy and functorial way of getting a contractible space out of $X$ is to consider the full Rips complex $\cP(X)$. (For indiscrete spaces one first has to pass to a so-called discretization $Z\subset X$. See Section~\ref{sec_coronas_and_coarse_cohomology_theories} for the details on this procedure.) One first defines $P_R(X)$ for $R \geq 0$ to be the simplicial complex whose simplices are exactly those whose vertices are at most $R$ apart from each other. For $R \leq S$ we obviously have an inclusion $P_R(X) \subset P_S(X)$ and hence we could consider $\cP(X)$ as the corresponding colimit. Due to technical reasons we will actually keep the information of the whole system $(P_R(X))_{R \in \IN}$ and consider $\cP(X)$ as a so-called $\sigma$-locally compact space. If $X$ is combable, then the space $\cP(X)$ is contractible as a $\sigma$-locally compact space (\cite[Theorem 10.6]{Wulff_CoassemblyRingHomomorphism}), which means that in particular for each $R$ there will be an $S$ such that $P_R(X)$ is contractible in $P_S(X)$.

Given a proper combing $H$ on $X$, it is straightforward to construct proper combings on $P_R(X)$ for all $R\geq 0$, which we also denote by $H$, such that each of the combing coronas $\partial_HP_R(X)$ is canonically isomorphic to $\partial_HX$ and such that we have inclusions of subspaces $\combcompb{X}{H}\subset \combcompb{P_R(X)}{H}\subset\combcompb{P_S(X)}{H}$ for all $0\leq R\leq S$. The sequence of all $\combcompb{P_R(X)}{H}$ defines a so-called $\sigma$-compact space $\combcompb{\cP(X)}{H}$ containing $\cP(X)$ and such that $\combcompb{\cP(X)}{H}\setminus\cP(X)=\partial_HX$.

We want that $\combcompb{\cP(X)}{H}$ is contractible. For this we need that the combing $H$ has two additional properties (cf.~Definiton~\ref{defn_proper_combing}, again):
\begin{itemize}
\item $H$ is called \emph{coherent} if 
\[E_{coh}\coloneqq \{(H_m\circ H_n(x),H_m(x))\mid x\in X,\,m,n\in\N,m\leq n\}\]
is an entourage, i.\,e.~$H_m\circ H_n$ and $H_m$ are $E_{coh}$-close for all $m\leq n$.
\item $H$ is called \emph{expanding} if there is an entourage $E_{exp}$ of $X$ such that for every entourage $E$ of $X$ and every $n\in\N$ there exists a bounded subset $K_{E,n}\subset X$ such that $H_n(E[x])\subset E_{exp}[H_n(x)]$ for all $x\in X\setminus K_{E,n}$.
\end{itemize}

It turns out that coherent combings are automatically proper (Lemma~\ref{lem_coherent_is_proper}) and therefore the combing corona and compactification of coherently combable spaces exist. The main technical theorem of the present paper is the following.

\begin{thmintro}[Theorem~\ref{thm:RipsCompactificationContractible}]
Let $(X,d)$ be a proper discrete metric space equipped with an expanding and coherent combing $H$.

Then for every $R>0$ exists $S>R$ such that $\combcompb{P_R(X)}{H}$ is contractible in $\combcompb{P_S(X)}{H}$. Even more, the $\sigma$-compact space $\combcompb{\cP(X)}{H}$ itself is contractible. Further, the corona $\partial_H X$ sits in $\combcompb{\cP(X)}{H}$ as a $Z$-set.
\end{thmintro}

In Section \ref{sec_coronas_and_coarse_cohomology_theories} we will introduce a suitable notion of (co-)homology theories $E_*$ ($E^*$, respectively) for $\sigma$-locally compact spaces which allow us to define the coarse (co-)homology $\EX_*(X)$ ($\EX^*(X)$, respectively) as the $E$-(co-)homology of $\cP(X)$. Important special cases are the locally finite $K$-homology $K_*^{\lf}(-)$, the $K$-theory functor $K^*_c(-,D)$ with compact supports and coefficients in any $C^*$-algebra $D$ and the Alexander--Spanier cohomology $\HA_c^*(-,M)$ with compact supports and coefficients in an abelian group $M$. The coarsification of the latter is isomorphic to Roe's coarse cohomology.
For such theories, the above theorem leads to the following results:

\begin{thmintro}[Sections~\ref{sec_coarse_theories_transgression_maps}, \ref{sec_roes_coarse_cohomology} and \ref{sec_coarse_assembly_coassembly}]\label{thm_intro_transgression}
Let $X$ be a proper metric space equipped with an expanding and coherent combing $H$.

Then the transgression maps
\begin{align*}
\EX_{\ast}(X) & \to \tilde{E}_{\ast-1}(\partial_H X)\\
\tilde{E}^{*-1}(\partial_HX) & \to \EX^*(X)
\end{align*}
are isomorphisms. 

Furthermore, the analytic coarse assembly map
\[\mu\colon \KX_*(X)\to K_*(C^*X)\]
is injective, and the analytic coarse co-assembly map
\[\mu^*\colon \tilde K_{1-*}(\mathfrak{c}(X;D))\to\KX^*(X;D)\]
with an arbitrary coefficient $C^*$-algebra $D$
is surjective.
\end{thmintro}

\begin{remintro}
One of the currently most general results about the analytic coarse assembly map being an isomorphism is due to Yu \cite{yu_embedding_Hilbert_space} who proved that this is the case for any space admitting a coarse embedding into a Hilbert space.
Mendel and Naor \cite{mendel_naor} constructed a proper $\CAT(0)$ space which contains a quasi-isometrically embedded expander and hence does not admit any coarse embedding into a Hilbert space.
This shows that our setup from Theorem~\ref{thm_intro_transgression} encompasses spaces that are not already taken care of by Yu's result.

\end{remintro}

The question now is how general the notions of coherence and ex\-pan\-ding\-ness are. As far as coherence goes, it seems to be a generic property of combings. Though we construct an artificial example of a proper combing which is not coherent, all combings occuring in nature are coherent. Especially, all combings from Figure~\ref{fig_examples_combable_spaces} are coherent. In the case of finitely generated groups, quasi-geodesically combable groups are properly combable but we can not say anything about coherence. What we have is that if the quasi-geodesic combing of a group is induced by an automatic structure, then it will be coherent.

Expandingness is a much more serious issue. As the examples of foliated and warped cones show, expandingness can interestingly fail. Furthermore, it turns out that most of the known automatic structures actually provably do not give expanding combings in general. But nevertheless, there is quite a big plethora of examples of expandingly combably spaces, resp.~groups:

\begin{exintro}[Sections~\ref{sec545342ewrtte} and \ref{secknjb4rermnmnmn}]
The following spaces have expanding and coherent combings:
\begin{itemize}
\item open cones over compact base spaces, and
\item coarsely convex spaces (e.g., hyperbolic and $\CAT(0)$ spaces, systolic complexes, and `nice' (proper and finite-dimensional) injective metric spaces and hierarchically hyperbolic spaces.).
\end{itemize}
\end{exintro}

\subsection{Groups with coherent and expanding combings}

Specializing our results to groups, we get the following:

\begin{thmintro}[Corollarys~\ref{cor23452342}, \ref{cor345243}, \ref{cor_cdRG_asdim}]
Let $G$ be a finitely generated group equipped with a coherent and expanding combing $H$ and let $R$ be a ring.

Then we have the isomorphism
\[H^{\ast+1}(G;R[G]) \cong \tilde{H}^\ast(\partial_H G;R).\]
Furthermore,
\[\cd_R(G) < \infty  \implies  \cd_R(G) - 1 = \max\{n\mid \tilde{H}^n(\partial_H G;R) \not= 0\}\]
and hence, in the case $\cd_R(G) < \infty$, the estimate $\cd_R(G) \le \asdim(G) + 1$.
\end{thmintro}

Section~\ref{sec45342345} is solely devoted to the proof of the following general result on the cohomological dimension of the combing corona:
\begin{thmintro}[Corollary~\ref{corrttrtrrt} and Theorem~\ref{thm_cdleqasdim}]
Let $G$ be a finitely generated group equipped with a coherent, expanding combing $H$ and let $G$ be of finite asymptotic dimension. Let $R$ be  either a unital ring whose underlying additive group structure is finitely generated or a field.

Then we have
\[\cd_R(\partial_H G) + 1 = \max\{k\mid \HX^k(G;R) \not= 0\}.\]
Consequently, we have the estimate\footnote{Our result does not always produce the best possible estimate. Guilbault--Moran \cite{guilbault_moran} prove $\dim(Z) + 1 \le \asdim(G)$, where $Z$ is the boundary of any $Z$-structure for $G$ (see Sections \ref{intro_section_Z_structures} \& \ref{secnm0987678bmvmvmv}).}
\[\cd_R(\partial_H G) \le \asdim(G).\]
\end{thmintro}

To state our final result, we first need one more definition. Let $H$ be a combing on a finitely generated group $G$.
\begin{itemize}
\item We say that $H$ is \emph{coarsely equivariant} if for all $g \in G$ the maps $\lambda_g \circ H$ and $H \circ (\lambda_g \times \id_\IN)$ are close, where $\lambda_g$ is left multiplication on $G$ by $g$.
\end{itemize}

If $H$ is coarsely equivariant, then the action of $\lambda_g$ extends to a continuous action on $\combcompb{\cP(G)}{H}$ by $\sigma$-homeomorphisms.

\begin{thmintro}[Theorem~\ref{thm_SNCF}, Remark~\ref{rem_split_injectivity} and Theorem~\ref{thm_dual_Dirac}]
Let $G$ be a group admitting a finite classifying space $BG$ and an expanding, coherent and coarsely equivariant combing $H$.

Then the analytic assembly map $RK_*(BG) \to K_*(C_r^* G)$ is split-injective, and hence the strong Novikov conjecture holds for $G$.

Further, we also get split injectivity results for the assembly maps for the non-connective versions of algebraic $K$-theory, $L$-theory and $A$-theory.

Moreover, $G$ admits a $\gamma$-element and both the coarse and equivariant co-assembly maps for $G$ are isomorphisms.
\end{thmintro}

Main examples of groups admitting coherent and expanding combings are groups which are coarsely convex spaces. Now in general there is no guarantee that such a structure on the group will be coarsely equivariant. But for the five main examples of coarsely convex structures on groups (the one coming from hyperbolicity or if the group acts geometrically on a $\CAT(0)$-cubical complex, a systolic complex an injective metric space, or a hierarchically hyperbolic space) one can show that the combing will be coarsely equivariant. Therefore we get the following examples (torsion-freeness is to ensure a finite $BG$ in these cases), where we use that Helly groups are an interesting class of groups acting geometrically on an injective metric space \cite{helly_groups}:

\begin{exintro}
Groups satisfying the assumptions of the above theorem are, e.g., torsion-free hyperbolic, $\CAT(0)$ cubical, systolic Helly and hierarchically hyperbolic groups.
\end{exintro}

\subsection{\texorpdfstring{$Z$}{Z}-structures for groups}
\label{intro_section_Z_structures}

Using the full Rips complex to construct a contractible space $\combcompb{\cP(X)}{H}$ in which the corona is a $Z$-set has the problem that $\cP(X)$ is neither finite-dimensional nor metrizable. So results of, e.g., Bestvina \cite{bestvina}, Guilbault--Tirel \cite{guilbault_tirel}, Moran \cite{moran_published} or Guilbault--Moran \cite{guilbault_moran_Z} are not applicable. But in the literature about $Z$-structures one often finds boundary swapping results, and in Section~\ref{secnm0987678bmvmvmv} we will prove such a result also in our setting: it is possible to put the combing corona $\partial_H G$ on an space $X$ with better point-set topological properties (like being a Euclidean retract) than $\cP(G)$. This results in the construction of $Z$-structures for groups using combings.

\begin{thmintro}[Corollary~\ref{cor243terwe}]
Let $G$ be a group equipped with an expanding and coherent combing~$H$ and let $G$ admit a $G$-finite model for $\EG$.

\begin{enumerate}
\item Then $\big(\combcompb{\,\EG\;}{H},\partial_H G \big)$ is a $Z_\ER$-structure for $G$.
\item If $G$ is torsion-free, then $(\combcompb{EG}{H},\partial_H G)$ is a $Z_\ER^\free$-structure for $G$.
\end{enumerate}
\end{thmintro}

Combining this corollary with the above mentioned results of Bestvina, Guilbault, Tirel and Moran, we get the following computation, resp.\ estimate for the covering dimension of the combing corona:
\begin{thmintro}[Corollary~\ref{cor_coronadimension}]
Let $H$ be an expanding and coherent combing on a finitely generated, discrete group $G$.
\begin{enumerate}
\item If $G$ admits a finite model for its classifying space $BG$, then
\[\dim(\partial_H G) + 1 = \cd(G).\]
\item If $G$ admits a $G$-finite model for $\EG$, then
\[\dim(\partial_H G) < \underline{\smash{\mathrm{gd}}}(G),\]
where the $\underline{\smash{\mathrm{gd}}}(G)$ denotes the least possible dimension of a $G$-finite model for the classifying space for proper actions $\EG$.
\end{enumerate}
\end{thmintro}

\section{Properly combable spaces and their compactifications}
\label{sec_properly_combable_spaces}

In Section~\ref{sec_proper_combings} we recall basic notions like coarse spaces and combings, and then define the fundamental notion of a proper combing which enables us to define in Section~\ref{sec_combing_corona} the combing compactification and corona. In the Section~\ref{sec_proper_combings} we also discuss the additional properties of coherence and expandingness whose purpose will not be seen until Section~\ref{sec_contr_compact}.

\subsection{Proper, coherent and expanding combings}
\label{sec_proper_combings}

Coarse spaces were introduced by John Roe. An introduction to this topic is \cite[Chapter 2]{roe_lectures_coarse_geometry}. We will recall only the basic notions which are needed in order to understand this paper.

\begin{defn}[Coarse spaces and maps]\mbox{}
\begin{itemize}
\item Let $X$ be a set.
\begin{itemize}
\item A \emph{coarse structure} on $X$ is a choice of subsets of $X \times X$ containing the diagonal of $X$ and which is closed under taking finite unions, compositions, inverses and subsets.\footnote{If $E$ and $F$ are subsets of $X \times X$, then the inverse of $E$ is $E^{-1} \coloneqq  \{(y,x) \mid (x,y) \in E\}$ and the composition of $E$ and $F$ is $E \circ F \coloneqq  \{(x,y) \mid \exists z \in X \text{ with } (x,z) \in E \ \& \ (z,y) \in F\}$.}
\item Elements of the coarse structure are either called \emph{controlled sets} of $X \times X$ or called \emph{entourages} of $X$.
\item The set $X$ equipped with a coarse structure is called a \emph{coarse space}.
\end{itemize}
\item Let $X$ be a coarse space and let $B$ be a subset of $X$. We call $B$ \emph{bounded} if $B \times B$ is controlled.
\item Let $X$, $Y$ be coarse spaces and $f\colon X \to Y$ a map of the underlying sets. Then $f$ is called
\begin{itemize}
\item \emph{controlled}, if $(f \times f)(U)$ is controlled for every entourage $U$ of $X$,
\item \emph{proper}, if $f^{-1}(B)$ is bounded for every bounded set $B$ of $Y$, and
\item \emph{coarse}, if it is controlled and proper.
\end{itemize}
\item Two maps $f,g\colon X \to Y$ are called \emph{close} (to each other) if $(f \times g)(\diag_X)$ is an entourage of $Y$.\qedhere
\end{itemize}
\end{defn}

\begin{rem}
Two comments are in order to the above definition:
\begin{enumerate}
\item Coarse structure are not always required to contain the diagonal, especially in early publications on this topic. In these days coarse structures containing the diagonal are called unital. See \cite[End of Section~2]{higson_pedersen_roe} for an example of a naturally occuring non-unital coarse structure. We are using here the nowadays common convention to only consider coarse structures containing the diagonal. Note that topological coarse spaces (Definition~\ref{defn2435r23we} below) necessarily must be unital.

Note that usually, when choosing entourages for an argument in a proof, we almost always implicitly assume that the chosen entourage contains the diagonal. Otherwise it happens that, e.g., the set $E[x]$ defined below in \eqref{eq_ball_entourage} does not contain the point $x$ itself, which is counter-intuitive.
\item Our definition of bounded subsets does not necessarily define a bornology, i.e., finite unions of bounded subsets need not be bounded. But in this paper all occuring coarse spaces (e.g., spaces admitting combings) will be so-called coarsely connected which implies that we actually do get a bornology \cite[Proposition 2.19]{roe_lectures_coarse_geometry}. Furthermore note that we are not using the more general bornological coarse spaces, where the bornology is not necessarily induced from the coarse structure, but only required to be compatible with it. See Bunke--Engel \cite{buen} for a thorough treatment of coarse homology theories on bornological coarse spaces.\qedhere
\end{enumerate}
\end{rem}

\begin{example}\label{ex_metric_is_coarse}
Basic examples of coarse spaces are provided by metric spaces. If $(X,d)$ is a metric space, then the natural coarse structure on it is generated by the entourages $E_r \coloneqq  \{(x,y) \in X\times X\mid d(x,y) \le r\}$ for all $r > 0$. The notion of a bounded subset (in the sense of coarse spaces) coincides in this case with the notion of a bounded subset in the usual metric sense.
\end{example}

In this paper we will be concerned with coarse spaces which are properly combable. We will recall the notion of a combing in the following definition. Properness is a property that a combing can have and it will be defined in Definition~\ref{defn_proper_combing}.

\begin{defn}\label{def_combing}
Let $X$ be a coarse space. By a \emph{combing} on $X$ starting at a point $p\in
X$ we mean a map
\[H\colon X\times\N\to X\,,\qquad (x,n)\mapsto H_n(x)\coloneqq H(x,n)\]
such that
\begin{enumerate}
\item\label{4532r} $H(x,0)=p=H(p,n)$ for all $x\in X$ and $n\in\N$,
\item\label{453987ztwert} for each bounded subset $K\subset X$ there is
an $N\in\N$ such that for all $n\geq N$ and $x\in K$ we have $H(x,n)=x$, and
\item\label{654terfd} $H$ is a controlled map.
\end{enumerate}

Given a combing $H$ on a coarse space $X$, we call the pair $(X,H)$ a \emph{combed coarse space}.

If the space can be equipped with a combing then it is called \emph{combable}.
\end{defn}

\begin{rem}
\label{rem_combingdefinition}
\begin{enumerate}[(a)]
\item \label{rem_enum_combedspacemorphisms}From time to time there will be the need to compare two different combed spaces. To this end, note that the combed spaces can be made into a category by defining a morphism $\alpha\colon (X,H_X)\to (Y,H_Y)$ to be a coarse map (or equivalence class of coarse maps, if one wishes) $\alpha\colon X\to Y$ such that $\alpha\circ H_X$ is close to $H_Y\circ(\alpha\times\id_{\N})$.

If $\alpha:X\to Y$ is a coarse equivalence and one of the coarse spaces $X$ or $Y$ is equipped with a combing, then it is straightforward to construct a combing on the other space such that $\alpha$ becomes a morphism between combed coarse spaces. 

\item \label{rem_enum_combingvscoarsecontraction}The second author defined coarse contractions in \cite[Definition~10.1]{Wulff_CoassemblyRingHomomorphism} with the sides $0$ and $\infty$ interchanged, i.\,e.\ with $H(x,0)=x$ for all $x\in X$ and for each  bounded subset $K\subset X$ an $N\in\N$ such that $H(x,n)=p$ for all $x\in K$ and $n\geq N$.

In order to unify the notions of combings and coarse contractions, one can even broaden the definition of coarse contractions to maps $H\colon X\times \Z\to X$ which converge to the identity for $n\to-\infty$ and to the constant map $p$ for $n\to\infty$ in the sense of Property \ref{453987ztwert} of the definition and which satisfies the other properties of the definition up to the appropriate adjustments. We remark that the theory presented in \cite{Wulff_CoassemblyRingHomomorphism} works perfectly even for this more general notion of coarse contractions.\qedhere
\end{enumerate}
\end{rem}

The theory which we are going to present does not work for general combed spaces. Instead, we will need to introduce the following additional properties. The first one will allow us to construct the combing compactification, and the second and third property will be used later to construct contractions of the compactification.

We will need the following notation: if $E$ is an entourage of the coarse space $X$ and $x \in X$ a point, then
\begin{equation}\label{eq_ball_entourage}
E[x] \coloneqq  \{y \in X\mid (y,x) \in E\}\, .
\end{equation}
This is the coarse analogue of the ``ball of radius $E$ around $x$''. Analogously we can define $E[A] \coloneqq  \{y \in X\mid \exists {x \in A} \text{ with } (y,x) \in E\}$ for any subset $A \subset X$.
In the special case of an entourage $E_r$ in a metric space (cf.~Lemma~\ref{ex_metric_is_coarse}) we obtain the closed $r$-ball $B_r(x)=E_r[x]$ and the $r$-neighborhood $B_r(A)=E_r[A]$.

\begin{defn}\label{defn_proper_combing}
Let $H\colon X\times\N\to X$ be a combing on a coarse space $X$.
\begin{itemize}
\item $H$ is called \emph{proper} if for any bounded subset $K\subset X$ there is a bounded subset
$L\subset X$ and an $N\in\N$ such that $H_n^{-1}(K)\subset L$ for all $n\geq N$.
\item $H$ is called \emph{coherent} if 
\[E_{coh}\coloneqq \{(H_m\circ H_n(x),H_m(x))\mid x\in X,\,m,n\in\N,m\leq n\}\]
is an entourage, i.\,e.~$H_m\circ H_n$ and $H_m$ are $E_{coh}$-close for all $m\leq n$.
\item $H$ is called \emph{expanding} if there is an entourage $E_{exp}$ of $X$ such that for every entourage $E$ of $X$ and every $n\in\N$ there exists a bounded subset $K_{E,n}\subset X$ such that $H_n(E[x])\subset E_{exp}[H_n(x)]$ for all $x\in X\setminus K_{E,n}$.
\end{itemize}

The pair $(X,H)$ is called a \emph{properly, coherently, resp.\ expandingly combed space}, if $H$ is a proper, coherent or expanding combing on $X$, respectively.

If $X$ can be equipped with a proper, coherent or expanding combing, then it is called \emph{properly, coherently or expandingly combable space}, respectively.

If $X$ is a metric space and if we have $E_{coh}\subset E_{R_{coh}}$ or $E_{exp}\subset E_{R_{exp}}$ for the entourage with respect to $R_{coh}\geq 0$ or $R_{exp}>0$ (recall Example~\ref{ex_metric_is_coarse}), then we say that  $H$ is \emph{$R_{coh}$-coherent} or \emph{$R_{exp}$-expanding}, respectively.
\end{defn}

We have the following basic relation:

\begin{lem}\label{lem_coherent_is_proper}
A coherent combing $H$ is always proper.
\end{lem}

\begin{proof}
The set $E\coloneqq \{(H(x,n+1),H(x,n))\mid  x\in X,n\in\N\}$ is an entourage by controlledness of $H$.

Given a bounded subset $K\subset X$, we define the subset $L\coloneqq (E\cup\diag_X)\circ (E_{coh}\cup\diag_X)[K]$ which is again bounded. Let $N\in\N$ be such that $H_m$ is the identity on $L$ for all $m\geq N$. We claim that $H_m(X\setminus L)\subset X\setminus K$ for all $m\geq N$.

To this end, let $x\in X\setminus L$. If $H_m(x)\notin L$, then $H_m(x)\notin K$.
If $H_m(x)\in L$, however, then there is $n\geq m$ such that $H_n(x)\in L$ but $H_{n+1}(x)\notin L$. By~our choice of $L$, this implies $H_n(x)\notin E_{coh}[K]$. As $H_m$ is the identity on $L$ we even have $H_m(H_n(x))=H_n(x)\notin E_{coh}[K]$, which implies $H_m(x)\notin K$ by definition of $E_{coh}$.
\end{proof}

Properness, coherence and expandingness are invariant under coarse equivalences, as the following lemma applied to coarse equivalences and their coarse inverses shows.

\begin{lem}\label{lem_AdditionalPropertiesCoarselyInvariant}
Let $(X,H_X)$ and $(Y,H_Y)$ be two combed coarse spaces and let $\alpha\colon X\to Y$ be a coarse equivalence onto its image $Z=\im(\alpha)\subset Y$ such that the map $\alpha\circ H_X$ is close to $H_Y\circ(\alpha\times\id_{\N})$.

Then properness, coherence or expandingness of $H_Y$ implies the respective property of $H_X$.
\end{lem}

\begin{proof}
Let $\beta\colon Z\to X$ be a coarse inverse to $\alpha$ and define the entourages
\begin{align*}
E_{\beta\alpha}&\coloneqq \{(\beta\circ\alpha(x),x)\mid  x\in X\}\subset X\times X \,,
\\E_H&\coloneqq \{(H_Y(\alpha(x,n)),\alpha\circ H_X(x,n))\mid x\in X,n\in\N\}\subset Y\times Y \,.
\end{align*}

\proofsubdivisor{First, assume that $H_Y$ is proper.} Given a bounded subset $K\subset X$, then the subset $K'\coloneqq E_H[\alpha(K)]\subset Y$ is bounded and so there exists $N\in\N$ and a bounded subset $L'\subset Y$ such that $(H_Y)_n^{-1}(K')\subset L'$ for all $n\geq N$. Then
\begin{align*}
(H_Y)_n\circ \alpha((H_X)_n^{-1}(K))&\subset E_H[\alpha\circ (H_X)_n((H_X)_n^{-1}(K))]\subset E_H[\alpha(K)]=K'
\\\implies (H_X)_n^{-1}(K)&\subset \alpha^{-1}((H_Y)_n^{-1}(K'))\subset\alpha^{-1}(L')=:L
\end{align*}
for all $n\geq N$ and $L\subset X$ is bounded. This shows properness of $H_X$.

\proofsubdivisor{Second, let $H_Y$ be coherent.} We denote the coherence-entourage of $H_Y$ by $E_{Y,coh}$ and furthermore we define
\[E'_H\coloneqq \bigcup_{m\in\N}((H_Y)_m\times (H_Y)_m)(E_H)\,,\]
which is an entourage of $Y$ due to the controlledness of $H_Y$. Then we have for all $m,n\in\N$ with $m\leq n$ and all $x\in X$

\begin{align*}
((H_Y)_m\circ (H_Y)_n\circ\alpha(x),(H_Y)_m\circ\alpha(x))&\in E_{Y,coh}
\\
\Rightarrow ((H_Y)_m\circ\alpha\circ (H_X)_n(x),(H_Y)_m\circ\alpha(x))&\in (E'_H)^{-1}\circ E_{Y,coh}
\\
\Rightarrow (\alpha\circ(H_X)_m\circ (H_X)_n(x),\alpha\circ(H_X)_m(x))&\in (E_H)^{-1}\circ(E'_H)^{-1}\circ E_{Y,coh}\circ E_H
\end{align*}
Now $\tilde E\coloneqq ((E_H)^{-1}\circ(E'_H)^{-1}\circ E_{Y,coh}\circ E_H)\cap (Z\times Z)$ is an entourage of $Z$ and we have
\begin{align*}
(\beta\circ\alpha\circ(H_X)_m\circ (H_X)_n(x),\beta\circ\alpha\circ(H_X)_m(x))&\in (\beta\times\beta)(\tilde E)
\\\implies ((H_X)_m\circ (H_X)_n(x),(H_X)_m(x))&\in (E_{\beta\alpha})^{-1}\circ(\beta\times\beta)(\tilde E)\circ E_{\beta\alpha}
\end{align*}
showing that $H_X$ is coherent with coherence-entourage 
\[E_{X,coh}\subset (E_{\beta\alpha})^{-1}\circ(\beta\times\beta)(\tilde E)\circ E_{\beta\alpha}\,.\]

\proofsubdivisor{Last we treat the case that $H_Y$ is expanding.} We denote the expandingness-entourage of $H_Y$ by $E_{Y,exp}$. 
For every entourage $E$ of $X$ and every $n\in\N$ the expandingness condition on $H_Y$ yields a bounded subset $L_{E,n}\subset Y$ such that 
\[(H_Y)_n(((\alpha\times\alpha)(E))[y])\subset E_{Y,exp}[y]\]
for all $y\in Y\setminus L_{E,n}$, and we define $K_{E,n}\coloneqq \alpha^{-1}(L_{E,n})$. Then for all $x\in X\setminus K_{r,n}$ we have
\begin{align*}
\alpha\circ(H_X)_n(E[x])&\subset E_H^{-1}
[(H_Y)_n\circ\alpha(E[x])]
\\&=E_H^{-1}[(H_Y)_n(((\alpha\times\alpha)(E))[\alpha(x)])]
\\&\subset E_H^{-1}\circ E_{Y,exp}[\alpha(x)]
\\\implies (H_X)_n(E[x])&\subset E_{\beta\alpha}^{-1}[\beta\circ\alpha\circ(H_X)_n(E[x])]
\\&\subset E_{\beta\alpha}^{-1}[\beta(E_H^{-1}\circ E_{Y,exp}[\alpha(x)])]
\\&=E_{\beta\alpha}^{-1}\circ((\beta\times\beta)(E_H^{-1}\circ E_{Y,exp}))[\beta\circ\alpha(x)]
\\&\subset E_{\beta\alpha}^{-1}\circ((\beta\times\beta)(E_H^{-1}\circ E_{Y,exp}))\circ E_{\beta\alpha}[x]\,,
\end{align*}
and so $H_X$ is expanding with some expandingness-entourage 
\[E_{X,exp}\subset E_{\beta\alpha}^{-1}\circ((\beta\times\beta)(E_H^{-1}\circ E_{Y,exp}))\circ E_{\beta\alpha}\,.\qedhere\]
\end{proof}

The second part of the following lemma exhibits a strengthening of the properness condition on a proper combing which is always possible.

\begin{lem}\label{lem453ewr43}
Let $H$ be a combing of $X$. Then for every $n \in \IN$ the image of $H_n$ is a bounded subset of $X$.

If $H$ is in addition proper, then for every bounded subset $K \subset X$ exists an $N \in \IN$ such that $H_n^{-1}(K) = K$ for all $n \ge N$.
\end{lem}

\begin{proof}
By Point~\ref{654terfd} of Definition~\ref{def_combing} the maps $H_n$ and $H_0$ are close to each other. But $H_0$ is constant, so $H_n$ must have bounded image.

If a bounded subset $K \subset X$ is given, then if $H$ is proper we know that there is a bounded subset $L \subset X$ and an $N^\prime \in \IN$ such that $H_n^{-1}(K) \subset L$ for all $n \ge N^\prime$. By Point~\ref{453987ztwert} of Definition~\ref{def_combing} there exists an $N^{\prime\prime} \in \IN$ such that $H_n|_L \equiv \id_L$ for all $n \ge N^{\prime\prime}$. Now $N \coloneqq  \max\{N^\prime,N^{\prime\prime}\}$ will do the job.
\end{proof}

\subsection{Combing compactification and corona}
\label{sec_combing_corona}

We will often encounter coarse spaces which are simultaneously topological spaces. The compatibility condition between a coarse structure and a topology is the following:

\begin{defn}\label{defn2435r23we}
Let the set $X$ be equipped both with a coarse structure and a topology which is Hausdorff.
\begin{itemize}
\item We call $X$ a \emph{topological coarse space} if there exists an open neighborhood of the diagonal in $X\times X$ which is also a coarse entourage.
\item We call a topological coarse space $X$ \emph{proper} if every bounded subset of $X$ is precompact (i.e., its closure is compact).
\end{itemize}
Note that a proper topological coarse space is necessarily locally compact and that the bounded subsets are exactly the ones which have compact closure.
\end{defn}

\begin{example}
If $X$ is a metric space, then the coarse structure naturally associated to $X$ (see Example \ref{ex_metric_is_coarse}) turns $X$ into a topological coarse space. Furthermore, $X$ is then a proper topological coarse space if and only if $X$ is proper as a metric space.
\end{example}

If we consider a combing on a topological coarse space, then a priori we do \emph{not} require it to be continuous.

Given a coarse space $X$, we denote by
\begin{itemize}
\item $D_b(X)$ the $C^*$-algebra of all bounded functions $f\colon X\to\C$ and by
\item $D_0(X)$ the sub-$C^*$-algebra of all those functions which vanish at infinity, i.\,e.\ for which for each $\varepsilon>0$ there is a bounded subset $K\subset X$ such that $\|f|_{X\setminus K}\|<\varepsilon$.
\end{itemize}

For each $f\in D_b(X)$ and each entourage $E$, the $E$-variation is the function $\Var_E(f)\in D_b(X)$ defined by
\[\Var_E(f)(x)\coloneqq \sup\big\{|f(y)-f(x)| \,\big|\, (y,x)\in E\big\}\,.\]
For the entourage $E_r$ in a metric space (cf.~Lemma~\ref{ex_metric_is_coarse}) we will simply write $\Var_r$ instead of $\Var_{E_r}$.

\begin{defn}
The function $f$ is said to have \emph{vanishing variation} if we have $\Var_E(f)\in D_0(X)$ for all entourages $E$. We denote the sub-$C^*$-algebra of $D_b(X)$ consisting of the functions of vanishing variation by $D_h(X)$. 
\end{defn}

If $X$ is a proper topological coarse space, then we denote by
\[C_0(X)\subset C_h(X)\subset C_b(X)\]
the respective sub-$C^*$-algebras of $D_0(X)\subset D_h(X)\subset D_b(X)$ consisting of all those functions which are additionally continuous.

The Higson corona of a coarse space $X$ is defined as the maximal ideal space of the unital, commutative $C^\ast$-algebra $D_h(X)/D_0(X)$ and the Higson compactification of a proper topological coarse space is the maximal ideal space of $C_h(X)$. In the latter case, the Higson corona is the boundary of the Higson compactification, because we have $D_h(X)/D_0(X)=C_h(X)/C_0(X)$, see Roe \cite[Top of Page 31]{roe_lectures_coarse_geometry}.

Using these function algebras, we now define the combing compactification and combing corona of properly combable (proper topological) coarse spaces.

\begin{defn}\label{def_corona_compactification}
Let $(X,H)$ be a properly combed coarse space.
\begin{itemize}
\item 
We denote by $D_H(X)$ the largest sub-$C^*$-algebra of $D_h(X)$ satisfying $H_n^*f\xrightarrow{n\to\infty}f$ for all $f\in D_H(X)$. 
\end{itemize}

This is a unital, commutative $C^\ast$-algebra containing $D_0(X)$. We call the maximal ideal space $\partial_HX$ of $D_H(X)/D_0(X)$ the \emph{combing corona} of $(X,H)$.

If $X$ is a proper topological coarse space, then we call the maximal ideal space $\combcompb{X}{H}$ of the unital, commutative $C^\ast$-algebra
\[C_H(X)\coloneqq D_H(X)\cap C_b(X)\]
the \emph{combing compactification} of $(X,H)$.
\end{defn}

By construction, $\partial_HX$ and $\combcompb{X}{H}$ are compact Hausdorff spaces and $\combcompb{X}{H}$ is a compactification of $X$.

Note that in this definition it was important to work with proper combings and not with general combings: if $H$ had not been proper, then the inclusions $D_0(X)\subset D_H(X)$ and $C_0(X)\subset C_H(X)$ would not necessarily hold meaning that $\combcompb{X}{H}$ would not necessarily be a compactification of $X$.

\begin{lem}
The combing corona is the boundary of the combing compactification, i.e., $\partial_HX=\combcompb{X}{H}\setminus X$.
\end{lem}

\begin{proof}
We have to show that $D_H(X)/D_0(X)=C_H(X)/C_0(X)$.

Every function $f\in D_H(X)$ is contained in $D_h(X)=C_h(X)+D_0(X)$, i.e., $f=g+h$ with $g\in D_0(X)$ and $h\in C_h(X)$. But we also know that we have $h=f-g\in D_H(X)$ and therefore $h\in D_H(X)\cap C_h(X)=C_H(X)$. Thus we get $D_H(X)=C_H(X)+D_0(X)$ and together with $C_0(X)=C_H(X)\cap D_0(X)$ we obtain $D_H(X)/D_0(X)=C_H(X)/C_0(X)$.
\end{proof}

The $C^\ast$-algebra $C_H(X)$ describes the topology of $\combcompb{X}{H}$. But if we are only interested in the topology at the corona, the $C^\ast$-algebra $D_H(X)$ is sufficient, as the following lemma shows.

\begin{lem}\label{lem:DHExtensions}
An arbitrary function $f\colon X\to \C$ admits an extension to $\combcompb{X}{H}$ with the extension being continuous at the points of the corona $\partial_HX$ if and only if $f\in D_H(X)$.
\end{lem}

\begin{proof}
For $f$ continuous, this is a consequence of the construction of $\combcompb{X}{H}$. 

If $f\in D_H(X)$, then we have seen in the proof of the previous lemma that $f=g+h$ with $g\in C_H(X)$ and $h\in D_0(X)$. As we have just mentioned, $g$ has a continuous extension to $\combcompb{X}{H}$, and $h$ can be extended to the compactification by setting $h|_{\partial_HX}=0$, and this extension is continuous at all points of the corona. Summing up these two extensions, we obtain the extension of $f$.

Conversely, if $f\colon\combcompb{X}{H}\to\C$ is continuous at all points of the corona, then we can find a continuous extension $g\in C(\combcompb{X}{H})$ of $f|_{\partial_HX}$ by Tietze's extension theorem. Then $h\coloneqq f-g$ restricts to a function in $D_0(X)$ and this implies that $f|_X\in C_H(X)+D_0(X)=D_H(X)$. 
\end{proof}

\begin{lem}\label{lem:metrizability}
Let $H$ be a proper combing on a proper metric space $X$.

Then its combing compactification $\combcompb{X}{H}$ is metrizable.
\end{lem}

\begin{proof}
To prove this, we use the fact that a compact Hausdorff space $Y$ is metrizable if and only if the $C^*$-algebra $C(Y)$ is separable.

As $\overline{\im(H_n)}\subset X$ is a metric space for every $n\in\N$, the corresponding $C^*$-algebras $C(\overline{\im(H_n)})$ are separable and therefore also
\[C(\combcompb{X}{H})=C_H(X)\subset\overline{\bigcup_{n\in\N}H_n^*(C(\overline{\im(H_n)}))}\]
is separable. Hence $\combcompb{X}{H}$ is metrizable.
\end{proof}

\begin{lem}\label{lem24354terwr}
Let $(X,H_X)$ and $(Y,H_Y)$ be two properly combed coarse spaces and let $\alpha\colon X\to Y$ be a coarse map such that 
$\alpha\circ H_X$ is close to $H_Y\circ(\alpha\times\id_{\N})$, i.\,e.~a morphism between combed coarse spaces in the sense of Remark~\ref{rem_combingdefinition}.(\ref{rem_enum_combedspacemorphisms}).

Then $\alpha^*$ maps $D_{H_Y}(Y)$ to $D_{H_X}(X)$. Furthermore, if $\beta\colon X\to Y$ is another coarse map which is close to $\alpha$, then $\alpha^*$ and $\beta^*$ agree up to $D_0(X)$.
\end{lem}

\begin{proof}
Clearly, $\alpha^*D_{H_Y}(Y)\subset \alpha^*D_h(Y)\subset D_h(X)$, because $\alpha$ is a coarse map. 

Let $f\in D_{H_Y}(Y)$ and $\varepsilon>0$. Since $\alpha\circ H_{X,n}$ is $E$-close to $H_{Y,n}\circ\alpha$ for some entourage $E$ of $Y$ and $f$ is of vanishing variation, there is a bounded subset $K\subset Y$ such that $\Var_E(f)$ has norm less than $\varepsilon$ outside of $K$. Because $H_Y$ is proper, we know that there exists a bounded subset $L\subset Y$ and an $N\in\N$ such that $H_Y(Y\setminus L\times\{N,N+1,\dots\})\subset Y\setminus K$.
Now
\begin{align*}
\|H_{X,n}^* & \alpha^*f -\alpha^*f\|\\
& \leq \|H_{X,n}^*\alpha^*f-\alpha^*H_{Y,n}^*f\|+\|\alpha^*H_{Y,n}^*f-\alpha^*f\|
\\&\leq\max\{\|H_{X,n}^*\alpha^*f-\alpha^* H_{Y,n}^* f\|_{\alpha^{-1}(L)}\,,\,\|H_{X,n}^*\alpha^*f-\alpha^* H_{Y,n}^* f\|_{X\setminus\alpha^{-1}(L)}\}
\\&\qquad +\|H_{Y,n}^*f-f\|
\end{align*}
is less than $\varepsilon$ for $n$ large enough, because the second term inside the maximum is less than $\varepsilon$ for $n\geq N$, the first is zero for $n$ big enough by Property~\ref{453987ztwert} of the combings $H_X$ and $H_Y$ as $\alpha^{-1}(L)$ and $\alpha(\alpha^{-1}(L))$ are bounded and the third converges to $0$ because $f\in D_{H_Y}(Y)$.

The second part of the statement is a trivial consequence of vanishing variation of $f$.
\end{proof}

The following corollary summarizes the functoriality properties of the combing compactification and combing corona.

\begin{cor}[Functoriality]\label{cor_coronafunctoriality}
\mbox{}
\begin{itemize}
\item Taking the combing corona is a functor from the category whose objects are properly combed spaces and morphisms are closeness classes of coarse maps compatible with the combings in the sense of Remark~\ref{rem_combingdefinition}.(\ref{rem_enum_combedspacemorphisms}) to the category of compact Hausdorff spaces and continuous maps.
\item Taking the combing compactification defines a functor from the category whose objects are properly combed proper topological coarse spaces and morphisms are continuous coarse maps satisfying the compatibility condition of Remark~\ref{rem_combingdefinition}.(\ref{rem_enum_combedspacemorphisms}) to the category of compact Hausdorff spaces and continuous maps.
\end{itemize}
\end{cor}

\begin{proof}
A map $\alpha$ as in the lemma maps $D_0(Y)$ to $D_0(X)$ by properness of $\alpha$. Hence, using Lemma \ref{lem24354terwr}, we have an induced $*$-homomorphism 
\[\alpha^*\colon D_H(Y)/D_0(Y)\to D_H(X)/D_0(X)\,,\]
which in turn induces a continuous map $\alpha_*\colon \partial_HX\to\partial_HY$.
This map only depends on the closeness class of $\alpha$ by the second part of Lemma \ref{lem24354terwr}.

For proper topological coarse spaces, a continuous coarse map $\alpha$ with the compatibility with the combings induces a $*$-homomorphism
\[\alpha^*\colon C_H(Y)=C_b(Y)\cap D_H(Y)\to C_H(X)=C_b(X)\cap D_H(X)\]
and hence a continuous map $\alpha_*\colon\combcompb{X}{H}\to\combcompb{Y}{H}$.

These two constructions clearly yield functors on the categories specified in the statement of the corollary.
\end{proof}

\begin{rem}
The compatibility condition in Lemma \ref{lem24354terwr} of $\alpha$ with the combings is crucial.

But on the other hand, let us mention \cite[Question 1.19]{bestvina_questions}: if $G$ is a hyperbolic group and $H$ a subgroup of $G$ and $H$ also happens to be hyperbolic, does the inclusion $H \to G$ extend to a continuous map $\partial H \to \partial G$ between the boundaries?\footnote{Note that by now this question was answered negatively by Baker--Riley \cite{baker_riley}.}

If $H$ is a so-called quasi-convexly embedded subgroup of $G$, the answer is yes (this is well-known and basic to prove). Now $H$ being quasi-convexly included in $G$ is exactly the condition that the inclusion $H \to G$ satisfies the compatibility condition from Lemma \ref{lem24354terwr} (here we use that hyperbolic spaces admit natural combings, see Section~\ref{secnb234trz}).
\end{rem}

\section{Examples}\label{sec_examples}
The examples in this section should supply a good understanding of the properties of properness, coherence and expandingness. Keeping these examples in mind should help the reader in understanding the proofs in the subsequent sections.

Properness and coherence are almost always satisfied in naturally appearing examples. Therefore, we start with two somewhat artificial examples to show that properness and coherence can actually fail.

\begin{example}\label{ex344wret}
Let us show that a combing does not necessarily have to be proper.
We consider $X = \IN$ equipped with the usual coarse structure coming from the metric on $\IN$. The map $H$ is defined in the following way:
\[H(x,n) \coloneqq 
\begin{cases}
0 & \text{ for } n \leq x\\
n - x & \text{ for } x\leq n\leq 2x\\
x & \text{ for } 2x \leq n
\end{cases}
\]
So the combing path to a point $x \in \IN$ first stays the time $x$ on $0$ and then moves with unit speed to $x$. One can quickly check that this defines a combing which is not proper.
\end{example}

\begin{example}\label{example_non_coherent}
We consider $X = \IN$ equipped with the usual coarse structure coming from the metric on $\IN$. A combing $H$ is defined by
\[H(x,n)\coloneqq 
\begin{cases}
n & \text{ for } 3n \leq x\\
4n-x & \text{ for } 2n\leq x\leq 3n\\
x & \text{ for } x \leq 2n
\end{cases}
\]
It is easily seen to be a proper combing, as $H_n(x)\geq \min\{x,n\}$ and therefore we have $H_n^{-1}(\{0,\dots,N\})\subset \{0,\dots,N\}$ for all $n\geq N+1$. But $H$ is not coherent, as the distance between $H_{4n}(12n)=4n$ and $H_{4n}(H_{5n}(12n))=H_{4n}(8n)=8n$ becomes arbitrarily large.
\end{example}

If one prefers a group as an example, just perform the combings of these two examples symmetrically on the two halfs of $\Z$.

In nature many occuring combings are quasi-geodesic, meaning that the combing paths are $(\lambda,k)$-quasi-geodesic segments for fixed $\lambda \ge 1$ and $k \ge 0$. Such combings are automatically proper. But let us first recall the notion of quasi-geodesics.
\begin{defn}
Let $(X,d)$ be a metric space and $\lambda \geq 1$, $k \geq 0$ be constants. A \emph{$(\lambda,k)$-quasi-geodesic} is a map $\gamma\colon I\to X$, where $I\subset \R$ is a closed connected subset, such that
\begin{equation}
\label{eq24354t423ret}
\lambda^{-1}\cdot |t-s|-k\leq d(\gamma(t),\gamma(s))\leq \lambda \cdot |t-s|+k
\end{equation}
for all $t,s\in I$.

We call $\gamma$ a \emph{$(\lambda,k)$-quasi-geodesic segment} if $I$ is a closed interval and a  \emph{$(\lambda,k)$-quasi-geodesic ray} if $I=\R_{\geq 0}$.
\end{defn}

\begin{lem}\label{lemretnmv34tr}
Let $X$ be a metric space and $H$ a quasi-geodesic combing on it.

Then $H$ is proper.
\end{lem}

\begin{proof}
The claim follows from the fact that $d(p, H_n(x)) \ge 1/\lambda \cdot n - k$, where $\lambda,k > 0$ are the quasi-geodesicity constants (the estimate holds provided $n$ is not so large that $H_-(x)$ already became contant).

Concretely, given a bounded subset $K \subset X$, we define the number $N \in \IN$ to be $\lceil \lambda\cdot(\max_{x \in K} d(p,x) + k)\rceil + \lambda$ and than choose $L \subset X$ large enough such that all $n \in \IN_0$ with $n \le N$ are still in the domain of definition of the quasi-geodesic segments $H_-(y)$ for all $y \in X \setminus L$ (e.g., we can set $L$ as the ball of radius $\lceil \lambda\cdot N + k \rceil$ around $p$).
\end{proof}

\subsection{Open cones}
\label{sec545342ewrtte}

The most basic examples of coherently and expandingly combable spaces are open cones over compact base spaces. In fact, these are the motivating examples for most of the theory in this paper.

Let $(B,d)$ be a compact metric space and $\phi\colon \R_{>0}\to \R_{>0}$ a monotonously increasing function.
The \emph{open cone} of growth $\phi$ over $B$ is defined as the space $\cO_\phi(B)\coloneqq B\times \R_{\geq 0}/B\times \{0\}$ equipped with the largest metric $d_{\phi}$ such that 
\[d_{\phi}((x,s),(y,t))\leq |t-s|+\phi(\min\{s,t\})\cdot d(x,y)\]
for all $x,y\in B$ and $s,t>0$.

If $B$ is a compact manifold, $d$ comes from a Riemannian metric $g$ and $\phi$ is continuous,\footnote{There is no need for us to assume here that Riemannian metrics are smooth.} then the metric $d_{\phi}$ is coarsely equivalent to the metric induced by the Riemannian metric
\[g_\phi\coloneqq dt^2+(\phi\circ t)^2\cdot g\,,\]
where $t$ denotes the $\R_{\geq 0}$-coordinate of the singular manifold $\cO_\phi(B)$.

For $\phi=\id$ the resulting open is usually called the Euclidean cone, which we denote by just $\cO(B)$.

The open cone $\cO_\phi(B)$ has an obvious canonical coherent combing. In the case $\phi(t)\xrightarrow{t\to\infty}\infty$ it is $R_{exp}$-expanding for all $R_{exp}>0$, and if $\phi$ is bounded, then it is $R_{exp}$-expanding for any $R_{exp} > \diam(M) \cdot \sup \phi$.

\begin{lem}\label{lem2343rew}
Let $B$ be a compact metric space and let $X=\cO_\phi(B)$ the open cone over $B$ equipped with the canonical coherent combing $H$. Then
\begin{itemize}
\item $\combcompb{\cO_\phi(B)}{H}$ is the closed cone compactification $\mathcal{C}(B)=B\times [0,\infty]/B\times\{0\}$ of $\cO_\phi(B)$ if $\phi(t)\xrightarrow{t\to\infty}\infty$, and
\item $\combcompb{\cO_\phi(B)}{H}$ is the one-point compactification $\cO_\phi(B)^+$ if $\phi$ is bounded.
\end{itemize}
\end{lem}

\begin{proof}
Consider first the case $\phi(t)\xrightarrow{t\to\infty}\infty$.
The closed cone $\mathcal{C}(B)$ is a Higson dominated compactification of the open cone $\cO_\phi(B)$. Moreover, the restriction $f|_{\cO_\phi(B)}$ of any $f\in \mathcal{C}(B)$ clearly satisfies $H_n^*f|_{\cO_\phi(B)}\xrightarrow{n\to\infty}f|_{\cO_\phi(B)}$. These two properties show that the restriction map yields an inclusion of $C^*$-algebras $r\colon C(\mathcal{C}(B))\subset C_H(\cO_\phi(B))$.

To show surjectivity of $r$, we use that for every $f\in C_b(\cO_\phi(B))$ the functions $H_n^*f$ obviously extend continuously to the closed cone and hence lie in the image of $r$. So if $f\in C_H(\cO_\phi(B))$, then also $f=\lim_{n\to\infty}H_n^*f$ lies in the image of $r$.

In the case that $\phi$ is bounded, we first note that $C(\cO_\phi(B)^+)\subset C_H(\cO_\phi(B))$, as the latter clearly contains both $C_0(\cO_\phi(B))$ and the constant functions.
For the other direction, let $f\in C_H(\cO_\phi(B))$.
For any $\varepsilon>0$ we can choose $n\in\N$ such that $\|H_n^*f-f\|<\frac{\varepsilon}{3}$ and such that $\Var_{\|\phi\|\cdot\operatorname{diam}(B)}(f)(x,s)<\frac{\varepsilon}{3}$ for all $(x,s)\in B\times [n,\infty)$. Then for all $(x,s),(y,t)\in B\times [n,\infty)$ we have
\begin{align*}
|f(x,&s)-f(y,t)|\leq
\\&\leq |f(x,s)-f(x,n)|+|f(x,n)-f(y,n)|+|f(y,n)-f(y,t)|
\\& <\varepsilon
\end{align*}

Thus, $f$ can be extended to the one-point compactification.
\end{proof}

It is interesting to note that in the case $\phi(t)\xrightarrow{t\to\infty}\infty$ we did not use the vanishing variation condition on functions in $C_H(\cO_\phi(B))$ to prove surjectivity. An explanation for this phenomenon is provided by the following lemma.
\begin{lem}\label{lem:automaticVV}
Assume that $H$ is a combing on a proper metric space which is $R_{exp}$-expanding for every $R_{exp}>0$. Then every function $f\in C_b(X)$ satisfying $H_n^*f\xrightarrow{n\to\infty}f$ has vanishing variation. 
\end{lem}

\begin{proof}
Let $r>0$ and $\varepsilon>0$. Choose $n\in\N$ such that $\|H_n^*f-f\|<\frac{\varepsilon}{3}$. As $f$ is uniformly continuous on the compact subset $\overline{\im(H_n)}$, we can find $\delta>0$ such that $|f(x)-f(y)|<\varepsilon$ whenever $x,y\in \overline{\im(H_n)}$ with $d(x,y)<\delta$.

We now apply the expandingness for $R_{exp}=\delta$ with $r,n$ as above to obtain a bounded subset $K\subset X$ such that for all $x\in X\setminus K$ and $y\in B_{r}(x)$ we have $d(H_n(x),H_n(y))<\delta$ $\implies |f(H_n(x))-f(H_n(y))|<\frac{\varepsilon}{3}$ $\implies |f(x)-f(y)|<\varepsilon$.

This implies $\Var_{r}f(x)<\varepsilon$ for $x\in X\setminus K$, and as $r,\varepsilon>0$ were arbitrary, the claim follows.
\end{proof}

\subsection{Foliated cones and warped cones}\label{subsubsec:examples:foliatedcones}

\begin{defn}[{\cite{Roe_Foliations}}]
Let $(M,\cF)$ be a foliated compact manifold and choose an arbitrary Riemannian metric $g$ on $M$. Denote by $g_N$ the normal component of the metric, i.\,e.~$g_N(v,w)\coloneqq g(Pv,Pw)$, where $P\colon TM\to N\cF$ is the orthogonal projection onto the normal bundle $N\cF$ of the foliation.
 The \emph{foliated cone} $\cO(M,\cF)$ is the singular manifold $M\times \R_{\geq 0}/M\times\{0\}$ equipped with the Riemannian metric
\[g_\mathcal{F}\coloneqq dt^2+g+t^2g_N\,,\]
where $t$ denotes the $\R_{\geq 0}$-coordinate.
\end{defn}
Foliated cones have nice functoriality properties and in particular the coarse structure on $\cO(M,\cF)$ induced by the metric $g_\cF$ does not depend on the choice of $g$ \cite[Theorem 3.3]{Wulff_Foliations}.

The foliated cone $\cO(M,\cF)$ admits a canonical coherent combing $H$, which is--as a point set map--exactly the same as the combing on the Riemannian cone.
Hence by Corollary \ref{cor_coronafunctoriality} the identity map
\[\cO(M)\to\cO(M,\cF)\]
induces continuous surjections
\[P\colon \mathcal{C}(M)=\combcompb{\cO(M)}{H}\to\combcompb{\cO(M,\cF)}{H}\,,\quad p\colon M=\partial_H\cO(M)\to\partial_H\cO(M,\cF)\,.\]

Let $f\in C(\partial_H\cO(M,\cF))$ and $F$ be a continuation of $f$ to $\mathcal{C}(M)$.
Then for any two points $x,y$ on the same leaf of $(M,\cF)$, the distance between $(x,t)$ and $(y,t)$ in $\cO(M,\cF)$ stays bounded  for all $t\in\R_{>0}$ and thus
\begin{align*}
(p^*f)(x)-(p^*f)(y)&=\lim_{t\to\infty}((P^*F)(x,t)-(P^*F)(y,t))
\\&=\lim_{t\to\infty}(F(x,t)-F(y,t))=0
\end{align*}
because of vanishing variation of $F|_{\cO(M,\cF)}$.
This implies that  $p$ is constant along leaves and thus factors through the space of leaves $M/\cF$ of the foliation, considered as a possibly non-Hausdorff quotient of $M$:
\[p\colon M\xrightarrow{p_1} M/\cF\xrightarrow{p_2}\partial_H\cO(M,\cF)\,.\]

\begin{claim}\label{claim:foliatedcone} The combing corona $\partial_H\cO(M,\cF)$ is the Hausdorffization of $M/\cF$, i.\,e.~every continuous map from $M/\cF$ into a compact Hausdorff space factors uniquely through $p_2$.
\end{claim}

\begin{proof}
The claim is equivalent to the map $p_2^*\colon C(\partial_H\cO(M,\cF))\to C(M/\cF)$ being an isomorphism.

By construction, $p_2^*$ is injective since $p_2$ is surjective, so it remains for us to prove surjectivity of $p_2^*$.  Let $f\in C(M/\cF)$ and consider a function 
\[F\colon \cO(M)\to\C\,,\quad F(x,t)=(p_1^*f)(x)\qquad\forall t\geq 1\,.\]
We already know that $F\in C_H(\cO(M))$ and we will have to show that $F$ has vanishing variation with respect to $\cO(M,\cF)$, because then 
\[F\in C_H(\cO(M))\cap D_h(\cO(M,\cF))=C_H(\cO(M,\cF))\cong C\big(\combcompb{\cO(M,\cF)}{H}\big)\]
restricts to a preimage of $f$ under $p_2^*\colon C(\partial_H(\cO(M,\cF)))\to C(M/\cF)$.

The proof of this is quite simple for smooth $f$.
Unfortunately, we cannot simply approximate $p_1^*f$ by smooth functions, because it is unclear whether these smooth approximations will still be constant along the leaves or not. Instead, we shall exploit the uniform continuity of $p_1^*f$: there exists a function $\rho\colon \R_{>0}\to \R_{>0}$ satisfying $\rho(r)\xrightarrow{r\to 0}0$ such that $|(p_1^*f)(x)-(p_1^*f)(y)|<\rho(r)$ whenever $x$ and $y$ have a $g$-distance of at most $r$. 

Let $p\in M$ be a point and choose a foliation chart 
$\varphi\colon U\approx \R^{\dim\cF}\times\R^{\codim\cF}$ 
mapping $x\in U$ to $(0,0)$.
For points $u=\varphi^{-1}(x,z)$, $v=\varphi^{-1}(y,z)$,  this chart gives rise to a holonomy isomorphism $h_{uv}\colon N\cF_u\cong N\cF_v$.

Given a smooth path $\gamma\colon[0,1]\to U$, we would like to construct a smooth path $\tilde\gamma$ in $U$ with the same starting point and the same $\R^{\codim\cF}$-component as $\gamma$, but which satisfies $\frac{\partial}{\partial s}\tilde\gamma(s)\in N\cF_{\tilde\gamma(s)}$, and then estimate the $g$-length of $\tilde\gamma$ by the $(g+t^2g_N)$-length of $\gamma$.

Note that $\tilde\gamma$ is the solution to the ordinary differential equation 
\begin{equation}\label{eq:pathODE}
\frac{\partial}{\partial s}\tilde\gamma(s)=h_{\gamma(s),\tilde\gamma(s)}\circ P_{\gamma(s)}\left(\frac{\partial}{\partial s}\gamma(s)\right)\,.
\end{equation}
Recall that $P\colon TM\to N\cF$ denotes orthogonal projection.
The path $\tilde\gamma$ might not be defined on all of $[0,1]$, but only on $[0,a)$, where at $a$ it would leave $U$.

To avoid the latter case, we restrict our attention to rather short paths starting relatively close to $p$:
Let $r(p)>0$ be such that $B_{3r(p)}(p)\subset U$ and such that the $h_{uv}$ have norm less than $2$ whenever $u,v\in B_{3r(p)}(p)$. If $\gamma$ has starting point within $B_{r(p)}(p)$ and has $g$-length less than $r(p)$, then $\tilde\gamma$ will have $g$-length less than $2r(p)$ and thus stays in $U$ on all of $[0,1]$. 

Assume now that $\gamma$ has $(g+t^2g_N)$-length less than $r<r(p)$, so that in particular also the $g$-length is less than $r$. Then Equation \eqref{eq:pathODE} together with the norm estimate of the $h_{uv}$ implies that $\tilde\gamma$ has $g$-length less than $\frac{2r}{t}$ and we obtain the estimate
\[|(p_1^*f)(\gamma(0))-(p_1^*f)(\gamma(1))|=|(p_1^*f)(\tilde\gamma(0))-(p_1^*f)(\tilde\gamma(1))|<\rho\left(\frac{2r}{t}\right)\,.\]

We can now cover $M$ by finitely many balls $B_{r_{p_1}}(p_1)$, \dots, $B_{r_{p_k}}(p_k)$ (with respect to the metric $g$) and choose a Lebesgue number $r<\min\{r_{p_1},\dots, r_{p_k}\}$ (again, with respect to $g$).
Then for every two points $x,y\in M$ of $(g+t^2g_N)$-distance less than $r$ we have
$|(p_1^*f)(x)-(p_1^*f)(y)|<\rho\left(\frac{2r}{t}\right)$. From this it is straightforward to deduce that the $r$-variation of $F$ vanishes at infinity. As the foliated cone is a path metric space, this also implies that the $R$-variation of $F$ vanishes for all $R>0$. This 
finishes the proof of the claim.
\end{proof}

The canonical combing on a foliated cone is in general not expanding, as the following example shows.

\begin{example}
Let $\cF$ be the Kronecker foliation with irrational slope on the torus $T^2$. Assume that $H$ is $R_{exp}$-expanding for some number $R_{exp}>0$. Pick two points $x$ and $y$ on the same leaf $L\approx\R$ such that their leafwise distance $R$ is greater than $R_{exp}$. This implies that there is $n\in\N$ such that the points $(x,t)$ and $(y,t)$ in the foliated cone $\cO(T^2,\cF)$ have distance within the interval $(R_{exp},R]$ for all $t\geq n$. Now, no matter how big the bounded subset $K_{r,n}$ for $r>R$ is chosen, we will always find a $t\geq n$ such that $(x,t),(y,t)\notin K_{r,n}$. Then the distance between the points $(x,t)$ and $(y,t)$ is less than $r$, but the distance between  $(x,n)=H_n(x,t)$ and $(y,n)=H_n(y,t)$ is greater than $R_{exp}$, contradicting the expansion assumption.
\end{example}

For any compact manifold $M$ the two trivial foliations $\cF_0$ by single points as leaves and $\cF_1$ by the whole of $M$ as one single leaf have canonical coherent combings which are $R_{exp}$-expanding for any $R_{exp}>0$ in the case of $\cF_0$ and for any $R_{exp}>\diam(M)$ in the case of $\cF_1$. This is easily seen, since these two foliated cones are the open cones $\cO_\phi(M)$ of growth $\phi=\id$ and $\phi\equiv 1$, resp.

A very similar concept to foliated cones are warped cones:
\begin{defn}[cf.~\cite{roe_warped}]
Let $(M,d)$ be a compact metric space and $G$ a finitely generated group acting on $M$ by homeomorphisms. Let $S\subset G$ be a finite generating set of the group. Then the \emph{warped cone} $\cO_G(M)$ is as a topological space the same as the Euclidean cone $\cO(M)$, but we equip it with the largest metric $d_G$ which is less or equal to the metric $d_{\id}$ of $\cO(M)$ and satisfies $d_G((x,t),(g\cdot x,t))\leq 1$ for all $(x,t)\in\cO_G(M)$ and $g\in S$.
\end{defn}
For example, the warped cone of the $\Z$-action on $S^1$ which lets $1\in\Z$ act by multiplication with $e^{2\pi i\alpha}$ is coarsely equivalent to the Kronecker foliation with irrational slope $\alpha$.

\begin{lem}
If we equip the warped cone $\cO_G(M)$ with the canonical coherent combing, then its combing corona $\partial_H\cO_G(M)$ is the Hausdorffization of $M/G$.
\end{lem}

\begin{proof}
The proof is completely analogous to the proof of Claim \ref{claim:foliatedcone}, but even simpler. One obtains continuous maps
\[p\colon M\xrightarrow{p_1} M/G\xrightarrow{p_2}\partial_H\cO_G(M)\]
and needs to show that the injective map $p_2^*\colon C(\partial_H\cO_G(M))\to C(M/G)$ is also surjective. Given $f\in C(M/G)$, one considers a function 
\[F\colon\cO(M)\to\C\,,\quad F(x,t)=(p_1^*f)(x)\qquad \forall t\geq 1\,,\]
and the goal is to show that it has vanishing variation with respect to the metric on $\cO_G(M)$.
This is straight-forward: if $R>0$, then a rough estimate shows that the $R$-entourage $E_R^{\cO_G(M)}$ of $\cO_G(M)$ is contained in $E_1\circ\dots\circ E_{2\lceil R\rceil+1}$ where each $E_i$ is either the $R$-entourage $E_R^{\cO(M)}$ of $\cO(M)$ or
\[E_\Gamma\coloneqq \{((x,t),(g\cdot x,t))\mid (x,t)\in\cO(M),g\in S\cup S^{-1}\cup\{e\}\}\,.\]
The $E_\Gamma$-variation of $F$ vanishes constantly whenever $t\geq 1$, since $f$ is invariant under the $G$-action and the $E_R^{\cO(M)}$-variation of $F$ vanishes at infinity. Hence, the $E_R^{\cO_G(M)}$-variation of $F$ vanishes at infinity, too.  The claim now follows as the corresponding claim for foliated cones.
\end{proof}

\subsection{\texorpdfstring{$\CAT(0)$}{CAT(0)}-spaces}

Our reference for $\CAT(0)$-spaces is Bridson--Haeflieger \cite[Part II]{bridson_haefliger}.

\begin{defn}
A geodesic metric space $X$ is called a \emph{$\CAT(0)$-space} if the distance between any two points on a geodesic triangle in $X$ is less or equal than the distance between the corresponding points on a comparison triangle in the space $\R^2$.
\end{defn}

It follows immediately from the definition that for any two points $x,y$ in a $\CAT(0)$-space $X$ there is exactly one geodesic $\gamma_{xy}$ parametrized by path-length connecting them.

For any point $p\in X$, one can define
\[H\colon X\times\N\to X\,,\quad(x,n)\mapsto\begin{cases}\gamma_{px}(n)&n\leq d(p,x)\,,\\x&n\geq d(p,x)\,.\end{cases}\]
The following lemma is clear from looking at the comparison triangles in $\R^2$.

\begin{lem}
\label{lem_cat0combingproperties}
$H$ is a combing which is both $0$-coherent and $R_{exp}$-expanding for all $R_{exp}>0$.\qed
\end{lem}

For each $n\in\N$, the image $X_n\coloneqq \im(H_n)$ is the closed $n$-ball around $p$. We obtain an inverse system of bounded metric spaces
\begin{equation}\label{defeq_cat0invsys}
\ldots\xrightarrow{H_4|_{X_5}} X_4\xrightarrow{H_3|_{X_4}}X_3\xrightarrow{H_2|_{X_3}}X_2\xrightarrow{H_1|_{X_2}}X_1\xrightarrow{H_0|_{X_1}}X_0=\{p\}\,,
\end{equation}
and due to $0$-coherence the compositions $H_m|_{X_{m+1}}\circ H_{m+1}|_{X_{m+1}}\circ\dots\circ H_{n-1}|_{X_n}$ of consecutive arrows in this diagram are equal to $H_m|_{X_n}$ for all $m\leq n+1$.
\begin{defn}[{cf.~\cite[Definitions II.8.1, II.8.5, II.8.6]{bridson_haefliger}}]
The \emph{visual bordification} $\overline{X}$ of a complete $\CAT(0)$-space $X$ is defined as the inverse limit of the system \eqref{defeq_cat0invsys}.
\end{defn}
If $X$ is also proper, then $\overline{X}$ is a compact metric space which is a compactification of $X$, called the \emph{visual compactification}. 

\begin{prop}\label{prop_CATnull_boundary}
The combing compactification $\combcompb{X}{H}$ of a proper $\CAT(0)$-space $X$ coincides with the visual compactification $\overline{X}$.
\end{prop}

\begin{proof}
According to Lemma \ref{lem:automaticVV} in conjunction with the Lemma \ref{lem_cat0combingproperties}, we can characterize the $C^*$-algebra $C_H(X)$ as consisting of exactly all those functions $f\in C_b(X)$ satisfying $H_n^*f\xrightarrow{n\to\infty}f$. 

For any $m\in\N$ and $f'\in C(X_m)$ consider the function $f\coloneqq H_m^*f'\in C_b(X)$.
Because of $0$-coherence, we have $H_n^*f=H_n^*H_m^*f'=(H_m\circ H_n)^*f'=H_m^*f'=f$ for all $n\geq m$, implying that $f\in C_H(X)$. Therefore, pullback with $H_m$ yields a $*$-homomorphism $H_m^*\colon C(X_m)\to C_H(X)$. Again by the characterization of $C_H(X)$ we even have
\[C_H(X)=\overline{\bigcup_{m\in\N}H_m^*C(X_m)}\,.\]

On the other hand, the right hand side of this equation equals the direct limit $\varinjlim_{m\in\N}C(X_m)$ of the dual of \eqref{defeq_cat0invsys}, since each $H_m$ surjects onto $X_m$ and therefore all the maps $H_m^*\colon C(X_m)\to C_H(X)$ are injective.
Thus we have $C(\combcompb{X}{H})=\varinjlim_{m\in\N}C(X_m)$, and by duality $\combcompb{X}{H}=\varprojlim_{m\in\N}X_m=\overline{X}$.
\end{proof}

\begin{rem}
A generalization of $\CAT(0)$-spaces are so-called Busemann spaces, where one demands a convexity condition for geodesics in the space. Fukaya--Oguni \cite{fukaya_oguni_busemann} showed that the compactification of a proper Busemann space by its visual boundary is contractible and Higson dominated.

The natural combing on a Busemann space, defined analogously as the combing above for $\CAT(0)$-spaces, is again $0$-coherent and $R_{exp}$-expanding for all $R_{exp} > 0$. Hence the above proof of Proposition~\ref{prop_CATnull_boundary} goes through in the more general setting of proper Busemann spaces showing that the combing compactification coincides with the visual compactification. The definition of the visual compactification used by Fukaya--Oguni \cite[Paragraph after Definition 2.1]{fukaya_oguni_busemann} is the same as above in the $\CAT(0)$-case: its the inverse limit of balls of bigger and bigger radii.
\end{rem}

\subsection{Hyperbolic spaces}
\label{secnb234trz}

Hyperbolic metric spaces were introduced by Gromov \cite{gromov_hyperbolic_groups}. We refer to \cite[Chapter III.H]{bridson_haefliger} for a thorough exposition about these spaces.

\begin{defn}
Let $(X,d)$ be a complete geodesic metric space. It is called \emph{hyperbolic}, if there exists a $\delta \ge 0$ such that every geodesic triangle is $\delta$-slim, i.\,e.\ if $\triangle$ is a triangle in $X$ consisting of three length miminizing geodesic segments between three vertices, then each of the sides of $\triangle$ lies within a $\delta$-neighborhood of the union of the other two sides.
\end{defn}

We fix a base point $p\in X$.
An important tool for investigating hyperbolic metric spaces is the Gromov product, which is defined by
\[(x|y)\coloneqq \frac12(d(x,p)+d(y,p)-d(x,y))\,,\quad (x,y\in X)\,.\]
In fact, there is even an alternative definition of hyperbolicity in terms of this product, see \cite[Definition III.1.20 and subsequent remarks]{bridson_haefliger}. We will not need this alternative definition, but the Gromov product will become indespensible for defining and investigating the Gromov compactification and the Gromov boundary.

For our purposes, a convenient but non-standard definition of the Gromov compactification and the Gromov boundary is the following one. We refer to \cite[Proposition 2.1]{roe_hyperbolic} for why this definition is equivalent to others.
\begin{defn}\label{def:Gromovcompactification}
The \emph{Gromov compactification} $\overline{X}$ of $X$ is the maximal ideal space of the commutative $C^*$-algebra
consisting of all bounded continuous functions $f\colon X\to\C$ with the property that for all $\varepsilon>0$ there is $K>0$ such that for all $x,y\in X$ we have
\begin{equation}
\label{eq_gromovboundary}
(x|y)>K\implies|f(x)-f(y)|<\varepsilon\,.
\end{equation}
By the following lemma, this $C^*$-algebra contains $C_0(X)$ as an essential ideal, so $\overline{X}$ is indeed a compactification of $X$ and we define the \emph{Gromov boundary} as $\partial X\coloneqq \overline{X}\setminus X$.
\end{defn}
\begin{lem}
For any minimizing geodesic $\overline{xy}$ between $x,y\in X$ we have $d(\overline{xy},p)\geq (x|y)$. In particular $d(x,p)\geq (x|y)$ and $d(y,p)\geq (x|y)$.
\end{lem}
\begin{proof}
For any $z\in\overline{xy}$ the triangle inequality implies \[d(z,p)\geq \frac12(d(x,p)-d(x,z)+d(y,p)-d(y,z))=(x|y)\,.\qedhere\]
\end{proof}

The next lemma defines the natural combing that we will use on hyperbolic spaces. Note that there are choices involved in the definition, but for any two choices the identity map of the space will be compatible with the combings in the sense of Remark~\ref{rem_combingdefinition}.(\ref{rem_enum_combedspacemorphisms}).

\begin{lem}\label{lem:hyperboliccontraction}
For any hyperbolic metric space $(X,d)$ and $p\in X$ there is a proper combing $H$ of $X$ starting at $p$. It is constructed by choosing for each $x\in X$ a minimizing geodesic $\gamma_x:[0,d(x,p)]\to X$ parametrized by path length between $p$ and $x$ and defining
\[H(x,n)\coloneqq \begin{cases}\gamma_x(n)&n\leq d(x,p)\,,\\x&n\geq d(x,p)\,.\end{cases}\]
\end{lem}

\begin{proof}
The map $H$ satisfies the first two of the properties in Definition \ref{def_combing}. Since the combing paths are geodesics, properness follows from Lemma~\ref{lemretnmv34tr}.

To show that $H$ is a controlled map, let $x,y\in X$ and denote by $\triangle$ the geodesic triangle consisting of the sides $\gamma_x$, $\gamma_y$ and some length minimizing geodesic $\overline{xy}$.

First, fix some $n\in\N$ and we shall abbreviate $x_n=H(x,n)$, $y_n=H(y,n)$. If $d(x_n,\overline{xy})<\delta$ and $d(y_n,\overline{xy})<\delta$, then obviously $d(x_n,y_n)< 2\delta+d(x,y)$.

If, on the other hand, $d(x_n,\overline{xy})\geq\delta$ (or analogously $d(y_n,\overline{xy})\geq\delta$), then there must be a point $z_n=\gamma_y(t_n)$ on $\gamma_y$ such that $d(x_n,z_n)<\delta$.
Now,
\begin{align*}
d(z_n,y_n)&=|d(y_n,p)-d(z_n,p)|\leq|n-d(z_n,p)|=|d(x_n,p)-d(z_n,p)|
\\&\leq d(x_n,z_n)<\delta\,,
\end{align*}
implying $d(x_n,y_n)< 2\delta$.

In any case, $d(x_n,y_n)< 2\delta+d(x,y)$ and consequently we have the estimate
$d(x_m,y_n)< 2\delta+d(x,y)+|m-n|$. This proves the controlledness condition.
\end{proof}

We note the following useful consequence of the previous two lemmas and their proofs.
\begin{cor}\label{cor:hyperbolic2delta}
For $n\leq(x|y)-\delta$ we have $d(H(x,n),H(y,n))< 2\delta$.\qed
\end{cor}

\begin{lem}
The combings of hyperbolic metric spaces constructed in the above Lemma \ref{lem:hyperboliccontraction} are $2\delta$-coherent and $2\delta$-expanding.
\end{lem}

\begin{proof}
To show coherence, let $m\leq n$ and $x\in X$. If $d(x,p)\leq n$, then we have $H_n(x)=x$ and hence $H_m(H_n(x))=H_m(x)$.
If $n< m+\delta$, then 
\begin{align*}
d(H_m(H_n(x)),H_m(x))&\leq d(H_m(H_n(x)),H_n(x))+d(H_n(x),H_m(x))
\\&\leq 2(n-m)<2\delta\,,
\end{align*}
where the second inequality is because the combing is by geodesics.
What remains is the case $d(x,p)>n\geq m+\delta$. But then
$(x|H_n(x))=n>m+\delta$ and Corollary \ref{cor:hyperbolic2delta} implies
\[d(H_m(H_n(x)),H_m(x))< 2\delta\,.\]
Thus, $H$ is $2\delta$-coherent.

To show expandingness, define the bounded subset $K_{r,n}=B_{r+n+\delta}(p)$ for $r>0$, $n\in\N$. Then for all $x\in X\setminus K_{r,n}$ and $y\in B_r(x)$ we have 
\begin{align*}
(x|y)&=\frac12(d(x,p)+d(y,p)-d(x,y))
\\&>\frac12((r+n+\delta)+(r+n+\delta-r)-r)=n+\delta
\end{align*}
and again Corollary \ref{cor:hyperbolic2delta} implies that $d(H_n(x),H_n(y))<2\delta$.
This proves the $2\delta$-expandingness.
\end{proof}

\begin{lem}\label{lem:Gromovequalscombing}
For any hyperbolic metric space $X$, the Gromov compactification coincides with the combing compactification, i.e., $\overline{X}=\combcompb{X}{H}$.
\end{lem}

\begin{proof}
Roe showed that the Gromov compactification $\overline{X}$ is a Higson dominated compactification of $X$.

Now, consider $f\in C(\overline{X})$. For $\varepsilon>0$ we choose $R>0$ as in the definition of $\overline{X}$. 
Then for all $n>R$ and all $x\in X$ we have either $x=H(x,n)$ (in the case $d(x,p)\leq n$) or
\[(x|H(x,n))=d(p,H(x,n))=n> R\]
(if $d(p,x)>n$). In both cases $|f(x)-f(H(x,n))|<\varepsilon$, so $\|f|_X-H_n^*f|_X\|\leq\varepsilon$ for all $n>R$.
In other words, $H_n^*f|_X\xrightarrow{n\to\infty}f|_X$.

We have thus shown that restriction yields an inclusion
$C(\overline{X})\subset C_H(X)$. It remains to prove that this restriction map is also surjective.

To this end, let $f\in C_H(X)$. 
Given $\varepsilon>0$ we choose $R>0$ such that we have $|f(x)-f(y)|<\frac{\varepsilon}{3}$ whenever $d(x,p),d(y,p)\geq R$ and $d(x,y)\leq 2\delta$. This can be achieved because of the vanishing variation condition. Next we choose $n\geq R$ big enough such that $\|f|_X-H_n^*f|_X\|<\frac{\varepsilon}{3}$, which can be done by the other condition on $C_H(X)$. Then for all $x,y\in X$ with $(x|y)>n+\delta$ we have $d(H(x,n),H(y,n))\leq 2\delta$ and $d(H(x,n),p)=n\geq R$, $d(H(y,n),p)=n\geq R$. Thus we have
\begin{align*}
|f & (x) - f(y)|\\
& \leq |f(x)-f(H(x,n))|+|f(H(x,n))-f(H(y,n))|+|f(H(y,n))-f(y)|\\
& < \varepsilon\,.
\end{align*}
Therefore $f$ extends to a continuous function on $\overline{X}$.
\end{proof}

\subsection{Coarsely convex spaces}
\label{secknjb4rermnmnmn}

A common generalization of both the hyperbolic and the $\CAT(0)$ case was given by Fukaya--Oguni \cite{fukaya_oguni} in the form of coarsely convex spaces, and they also constructed a boundary for such spaces. These coarsely convex spaces come by definition with a combing. The goal of this section is to prove that this combing is coherent and expanding, and hence coarsely convex spaces fit naturally into the setting of the present paper.

Note that there are also other generalizations of the notions of hyperbolic and/or $\CAT(0)$ spaces (Descombes--Lang \cite{descombes_lang}, Hotchkiss \cite{hotchkiss} and Buckley--Falk \cite{buckley_falk}). But they are all subsumed in the notion of coarsely convex spaces.

\begin{defn}[Fukaya--Oguni~{\cite[Definition~3.1]{fukaya_oguni}}]
\label{defn23435erwqwnb}
A metric space $X$ is called coarsely convex, if there exist constants $\lambda \ge 1$, $k \ge 0$, $E \ge 1$ and $C \ge 0$, a non-decreasing function $\theta\colon \IR_{\ge 0} \to \IR_{\ge 0}$, and a family $\mathcal{L}$ consisting of $(\lambda,k)$-quasi-geodesic segments, such that
\begin{enumerate}
\item for any two points $v,w \in X$ there is a quasi-geodesic segment $\gamma \in \mathcal{L}$ with $\gamma\colon [0,a] \to X$, $\gamma(0) = v$ and $\gamma(a) = w$,
\item\label{43rew} for any two quasi-geodesic segments $\gamma, \eta \in \mathcal{L}$ with $\gamma\colon [0,a] \to X$ and $\eta\colon [0,b] \to X$ and for all $t \in [0,a]$, $s \in [0,b]$ and $0 \le c \le 1$ we have
\[d(\gamma(ct),\eta(cs)) \le c E d(\gamma(t),\eta(s)) + (1-c) E d(\gamma(0),\eta(0)) + C,\]
\item \label{def:coarslyconvex:enum:tminuss} and for all quasi-geodesic segments $\gamma, \eta \in \mathcal{L}$ with $\gamma\colon [0,a] \to X$ and $\eta\colon [0,b] \to X$ and all $t \in [0,a]$ and $s \in [0,b]$ we have
\[|t-s| \le \theta\big( d(\gamma(0), \eta(0)) + d(\gamma(t),\eta(s)) \big).\qedhere\]
\end{enumerate}
\end{defn}

\begin{rem}
\label{rem2345rte3q}
Note that without loss of generality we can assume that for any two points of $X$ there is only one quasi-geodesic in $\mathcal{L}$ connecting these two points, and that for any point $x \in X$ the quasi-geodesic in $\mathcal{L}$ connecting $x$ to itself is the constant path.
\end{rem}

\begin{lem}\label{lem:coarslyconvexcombing}
Let $X$ be coarsely convex and $\cL$ a choice of quasi-geodesics as in Remark \ref{rem2345rte3q}. Let $p\in X$ be an arbitrary basepoint and denote by $\gamma_x\in\cL$ the quasi-geodesic $\gamma_x\colon [0,t_x]\to X$ from $p$ to $x$. Then the formula
\begin{equation*}
H(x,n)\coloneqq \begin{cases}
\gamma_x(n) & \text{for } n\leq  t_x\,,\\
x & \text{for } n \geq t_x\,.
\end{cases}
\end{equation*}
defines a coherent and expanding combing on $X$.
\end{lem}

\begin{proof}
Point \ref{4532r} of Definition~\ref{def_combing} is satisfied by construction of $H$, resp.~by assumption (see Remark~\ref{rem2345rte3q}).

Point~\ref{453987ztwert} is satisfied, because $\cL$ consists of quasi-geodesics with uniform constants $(\lambda,k)$. Concretely, for all $x\in X$ we have
\[d(x,p)=d(\gamma_x(t_x),\gamma_x(0))\geq \lambda^{-1}\cdot t_x-k\implies t_x\leq \lambda\cdot(k+d(x,p))\,,\]
so if $K\subset X$ is bounded, then $H(x,n)=x$ for all $x\in K$ and 
\[n\geq N\coloneqq  \left\lceil\sup_{x\in K}\lambda\cdot(k+d(x,p))\right\rceil\,.\]

A similar argument also shows that $H$ is proper: For $K\subset X$ bounded and $N\in\N$ as above we claim that $H_n^{-1}(K)\subset K$ for all $n\geq N+1$: If $x\notin K$ and $n\geq t_x$, then $\gamma_x(n)=x\notin K$, and if $x\notin K$ and $t_x\geq n\geq N+1$, then 
\[d(\gamma_x(n),p)\geq \lambda^{-1}\cdot n-k> \sup_{x\in K}d(x,p)\implies \gamma_x(n)\notin K\,.\]

It remains to show Point~\ref{654terfd}, i.\,e., that $H$ is a controlled map, to prove that $H$ is indeed a combing. 
For all $x,y\in X$ Point \ref{def:coarslyconvex:enum:tminuss} of Definition \ref{defn23435erwqwnb} yields
\[|t_x-t_y|\leq \theta(d(\gamma_x(t_x),\gamma_y(t_y)))=\theta(d(x,y))\,.\]
Assuming without loss of generality that $t_x\geq t_y$ and making use of the second inequality in~\eqref{eq24354t423ret}, we have for all $n\in [t_y,t_x]\cap\N$
\begin{align*}
d(H_n(x),H_n(y))&=d(\gamma_x(n),y)\leq d(\gamma_x(n),x)+d(x,y)
\\&\leq \lambda\cdot(t_x-n)+k+d(x,y)
\\&\leq \lambda \theta(d(x,y))+k+d(x,y)\,.
\end{align*}
For $n\in[0,t_y)\cap\N$ we use Property~\ref{43rew} of Definition~\ref{defn23435erwqwnb} to calculate
\begin{align*}
d(H_n(x),H_n(y))&=d(\gamma_x(n),\gamma_y(n))
\\&\leq \frac{n}{t_y}Ed(\gamma_x(t_y),\gamma_y(t_y)))+C
\\&\leq Ed(\gamma_x(t_y),y)+C
\\&\leq E\theta(d(x,y))+Ek+Ed(x,y)+C\,,
\end{align*}
where the last inequality sign is derived just as in the previous calculation.
Finally, we use the second inequality in~\eqref{eq24354t423ret} again and piece all together, obtaining the inequality
\[d(H_m(x),H_n(y))\leq E\theta(d(x,y))+Ek+Ed(x,y)+C+\lambda\cdot|m-n|+k\,.\]
This shows that $H$ is a controlled map.

Let us prove now coherence of $H$. Let $m\leq n$ and $x\in X$. If $t_x\geq n$, then $H_m(H_n(x))=H_m(x)$. So assume $n<t_x$ and denote $y\coloneqq H_n(x)$ out of convenience. From  Point \ref{def:coarslyconvex:enum:tminuss} of Definition \ref{defn23435erwqwnb} we obtain 
\[|n-t_y|\leq \theta(d(\gamma_x(n),\gamma_y(t_y)))=\theta(d(y,y))=\theta(0)\,.\]
Denote $s\coloneqq \min\{n,t_y\}$, such that either $\gamma_x(s)=y$ and $d(\gamma_y(s),y)\leq \lambda\theta(0)+k$ or $d(\gamma_x(s),y)\leq \lambda\theta(0)+k$ and $\gamma_y(s)=y$. Then in the case $m<s$ we have
\begin{align*}
d(H_m(H_n(x)),H_m(x))&=d(\gamma_y(m),\gamma_x(m))\\
& \leq \frac{m}{s}Ed(\gamma_y(s),\gamma_x(s))+C\\
&\leq E(\lambda\theta(0)+k)+C\,.
\end{align*}
The other case, $m\geq s$, implies $t_y=s\leq m\leq n$. Here we have
\begin{align*}
d(H_m(H_n(x)),H_m(x))&=d(y,\gamma_x(m))\\
&=d(\gamma_x(n),\gamma_x(m))\\
&\leq\lambda |n-m|+k\\
&\leq\lambda |n-t_y|+k\\
&\leq \lambda\theta(0)+k\,.
\end{align*}
In total, the maps $H_m\circ H_n$ and $H_m$ are $R_{coh}\coloneqq (E(\lambda\theta(0)+k)+C)$-close for all $m\leq n$.

It remains to prove expandingness. For $r>0$ and $n\in\N$ we define
\[K_{r,n}\coloneqq B_{\lambda n(r+\lambda\theta(r)+k+1)+\theta(r)+k}(p)\,.\]
Then for $x\in X\setminus K_{r,n}$ and $y\in B_r(x)$ we denote $t\coloneqq \min\{t_x,t_y\}$. Note that $|t_x-t_y|\leq \theta(r)$ and 
$t\geq n(r+\lambda\theta(r)+k+1)>n$
by choice of $K_{r,n}$ and the $(\lambda,k)$-quasi-geodicity of $\gamma_x$ and $\gamma_y$. Similar to previous calculations we have
\begin{align*}
d(H_n(x),H_n(y))&=d(\gamma_x(n),\gamma_y(n))\\
& \leq \frac{n}{t}Ed(\gamma_x(t),\gamma_y(t))+C\\
&\leq \frac{n}{t}E(r+\lambda\theta(r)+k)+C\\
&< E+C=:R_{exp}.
\end{align*}
This shows expandingness.
\end{proof}

For a coarsely convex, proper metric space $X$, Fukaya--Oguni constructed a compactification $\overline{X}$ and showed that $X$ is coarsely homotopy equivalent to the Euclidean open cone $\mathcal{O}(\partial X)$ over the boundary $\partial X \coloneqq  \overline{X} \setminus X$. After this article appeared on the arXiv, Fukaya--Oguni \cite[Corollary~8.9]{fukaya_oguni} proved that the combing compactification $\combcompb{X}{H}$ of a coarsely convex, proper metric space $X$ is homeomorphic to $\overline{X}$ by providing a functional analytic characterization of $\overline{X}$ analogous to \eqref{eq_gromovboundary} for the Gromov boundary. Note that in the special cases of hyperbolic and of $\CAT(0)$ spaces we have done this identification (i.e., that the boundaries one usually defines for these spaces are homeomorphic to the combing corona) in the corresponding sections above.

\begin{examples}\label{examples_coarsely_convex}
We have already mentioned that hyperbolic and $\CAT(0)$ spaces are coarsely convex. Let us collect some more examples.
\begin{enumerate}
\item Osajda--Przytycki \cite{osajda_przytycki} defined a boundary for systolic complexes by first constructing a system of (quasi-)geodesics satisfying coarsely a weak form of the $\CAT(0)$ condition and being coarsely closed under taking subsegments. This implies that systolic complexes are coarsely convex. Let us further mention that the system of (quasi-)geodesics is coarsely equivariant for the whole isometry group of the systolic complex.

Comparing definitions, we see that the boundary of Osajda--Przytycki coincides with the one constructed by Fukaya--Oguni.

\item Lang proved that every injective metric space $X$ admits a conical and reversible geodesic bicombing \cite[Prop.\ 3.8]{lang}. If the space is additionally proper and of finite topological dimension, then we even have a convex, consistent and reversible bicombing \cite[Thm.\ 1.1 \& 1.2]{descombes_lang}. This implies that proper, finite-dimensional injective metric spaces are coarsely convex. Further, this bicombing is equivariant for the whole isometry group of $X$.

Descombes--Lang construct a boundary for any complete metric space equipped with a convex and consistent geodesic bicombing \cite[Sec.\ 5]{descombes_lang}. Comparing definitions, we see that this boundary coincides with the one of Fukaya--Oguni.

\item\label{examples_coarsely_convex_iii} Haettel--Hoda--Petyt proved that hierarchically hyperbolic spaces are quasi-isometric to injective metric spaces \cite{HHP}. So, if they are additionally proper and finite-dimensional, then they admit coarsely convex structures which are coarsely equivariant for the respective automorphism groups.\footnote{Note that Durham--Minsky--Sisto also constructed a bicombing on hierarchically hyperbolic spaces \cite{DMS}, but which is in general not expanding.}

The construction of Fukaya--Oguni provides us a boundary for these spaces. But Durham--Hagen--Sisto \cite{HHS_boundaries} also constructed a boundary for hierarchically hyperbolic spaces and it is not clear whether these boundaries are homeomorphic to each other.\qedhere
\end{enumerate}
\end{examples}

\subsection{Automatic groups}
\label{secjkjbjnbbbm}

For simplicity we restrict ourselves to finitely generated groups only. We turn them into metric spaces by choosing a finite, symmetric generating set and using the induced word metric. Different choices of generating sets lead to quasi-isometric word metrics (in fact, the identity map of the underlying sets will be a quasi-isometry) and hence our results in this section are independent from the chosen finite generating set.

We choose the identity element of the group as the base point for combings on it. Note that a combing in our sense (i.e., in the sense of Definition~\ref{def_combing}) is what Gersten \cite{gersten} calls a synchronous combing and what Alonso \cite{alonso} calls a bounded combing.

We use \cite{epstein_et_al} as our reference for automatic groups. If $G$ is finitely generated and automatic, then by \cite[Theorem 2.5.1]{epstein_et_al} we can (and will) assume that the automatic structure of it has the uniqueness property. We then conclude from \cite[Lemma 2.3.2]{epstein_et_al} that $G$ is combable by a combing which is quasi-geodesic by \cite[Theorem 3.3.4]{epstein_et_al}. From Lemma~\ref{lemretnmv34tr} we then conclude that this combing is proper.

If the automatic structure of $G$ is prefix-closed (in addition to having the uniqueness property), then the combing will be $0$-coherent (i.\,e.~$H_m \circ H_n = H_m$ for all $m\leq n$). But it is still an open problem whether an automatic group can always be equipped with an automatic structure which simultaneously has the uniqueness property and is prefix-closed. Nevertheless we have the following:

\begin{lem}\label{lem234erw324}
Let $G$ be a finitely generated, automatic group.

Then the automatic structure induces a coherent combing on $G$.
\end{lem}

\begin{proof}
Let $L$ be the language of the automatic structure of $G$ and let $L^\prime$ be its prefix-closure (which in general does not have the uniqueness property). By \cite[Theorem 2.5.9]{epstein_et_al}, $L^\prime$ is also an automatic structure of $G$. Therefore $L^\prime$ has the $k$-fellow-traveling property for some $k$. This result is also valid for automatic structure which do not have the uniqueness property. This implies that the combing defined by $L$ is coherent.

To provide more details: we want a global bound on the distance between $H_m(H_n(g))$ and $H_m(g)$. Now $H_\bullet(H_n(g))$ is a word for $H_n(g)$ which is accepted by $L^\prime$, and so is $H_\bullet(g)$ (we use here that the combing line of $H_\bullet(g)$ passes through $H_n(g)$). Because $L^\prime$ has the $k$-fellow-traveling property the distance between $H_m(H_n(g))$ and $H_m(g)$ will be at most $k$.
\end{proof}

Let us turn our attention to expandingness. As it turns out, the automatic structures that we currently know are unfortunately in general not expanding. Let us start with the product automatic structure on a product of automatic groups.

\begin{example}\label{ex46rwertf}
The product automatic structure \cite[Theorem 4.1.1]{epstein_et_al} on $\IZ \times \IZ$, where we use on both $\IZ$-factors the usual automatic structure, is not expanding.

This even holds in full generality: the combing paths in the product $A \times B$ are given by first walking inside $A$ to its identity element $e_A$ and then walking inside $B$ to its identity $e_B$. So if $A$ and $B$ both contain a quasi-isometrically embedded copy of $\IZ$, then this is never expanding.
\end{example}

The following is a (non-exhaustive) list of groups known to be automatic (we give in parentheses references for the proofs of automaticity):
\begin{enumerate}
\item hyperbolic groups (see, e.g., \cite[Theorem 3.4.5]{epstein_et_al}),
\item central extensions of hyperbolic groups (Neumann--Reeves \cite{neumann_reeves}),
\item $\CAT(0)$ cube groups (Niblo--Reeves \cite[Theorem 5.3]{niblo_reeves} for torsion-free groups and {\'{S}}wi{\c{a}}tkowski \cite[Corollary 8.1]{swiatkowski} in the general case),
\item Coxeter groups (Brink--Howlett \cite{brink_howlett}),
\item systolic groups (Januszkiewicz--\'{S}wi{\c{a}}tkowski \cite[Theorem E]{janus_swia}),
\item Artin groups of finite type (Charney \cite{charney}),
\item Artin groups of almost large type (Huang--Osajda \cite{huang_osajda}),
\item groups acting geometrically and in an order preserving way on Euclidean buildings of type $\tilde A_n$, $\tilde B_n$ or $\tilde C_n$ (Noskov \cite{noskov} for the case of groups acting freely and {\'{S}}wi{\c{a}}tkowski \cite{swiatkowski} for the general case), and
\item mapping class groups (Mosher \cite{mosher}).
\end{enumerate}

The automatic structure on hyperbolic groups induces the combing discussed in Section~\ref{secnb234trz} and is hence expanding. But unfortunately, in many of the other cases this is not the case anymore:
\begin{enumerate}
\item The automatic structure on $\IZ \times \IZ$ viewed as a $\CAT(0)$ cube group in the usual way, induces the following combing paths: ``walk diagonally until you hit an axis and then walk along the axis towards the identity element.'' This is not expanding.
\item In Coxeter groups the automatic structure uses the shortlex word for each element of the group relative to a fixed ordering of the generating reflections. This is also in general not expanding. To give an example, let us consider the Coxeter group generated by reflections $a,b,c,d$ such that $ab$ and $cd$ have infinite order and $ac,ad,bc,bd$ have order two. This Coxeter group contains a free abelian subgroup generated by $ab$ and $cd$. The shortlex word for $(ab)^n (cd)^m$ for the ordering $a<b<c<d$ is $abab \ldots abcdcd \ldots cd$. So the combing paths we see are the same as in the product automatic structure of $\IZ^2$, and this is not expanding.
\item The combing paths induced by the automatic structure on $\IZ \times \IZ$ viewed as a systolic group in the usual way, is also not expanding.

Artin groups of almost large type were proven by Huang--Osajda \cite{huang_osajda} to be automatic by establishing that they are systolic. Hence their combing paths in this case are in general also not expanding.

Note that different kinds of Artin groups were shown to be automatic by different methods (e.g.,~\cite{peifer,brady_mccammond,holt_rees_1,holt_rees_biautomatic}). We have not investigated all of those and therefore can not say whether any of them have expanding combing paths.
\item If we consider the trivial central extension of a hyperbolic group by~$\IZ$, then the automatic structure constructed on it by Neumann and Reeves \cite{neumann_reeves} will be the product automatic structure, see \cite[Cor.~2.3]{neumann_reeves_2} where it is explicitly described. By Example~\ref{ex46rwertf} we know that this is not expanding.

But see also Question~\ref{question_central_extensions} whether one can equip central extensions of hyperbolic groups with coherent and expanding combings.
\end{enumerate}

\begin{rem}
Note that $\CAT(0)$ cube groups and systolic groups are examples of coarsely convex space in the sense of Fukaya--Oguni, see Section~\ref{secknjb4rermnmnmn}. Hence there exists on such groups an expanding and coherent combing. But this combing is not the one induced by the above discussed automatic structures on these groups.
\end{rem}

We have not investigated what happens in the case of Artin groups of finite type, in the case of mapping class groups, or in the case of groups acting nicely on buildings of certain types, see Question~\ref{quesnrt23er}.

We will see in later sections that expandingness is a crucial property of a combing as it implies that the combing compactification contains a lot of information about the space it compactifies.
If expandingness fails, one can still ask if there is some information stored in the combing compactification.
The following example shows that this is in general not the case.

\begin{example}\label{ex_onepointcompactification}
We have discussed above two different automatic structures on $\IZ \times \IZ$: the product automatic structure and the one arising from viewing it as a $\CAT(0)$ cube group in the usual way.
In both cases one can check that the combing compactification of the induced combings is the one-point compactification.
\end{example}

\section{Coronas and coarse (co-)homology theories}
\label{sec_coronas_and_coarse_cohomology_theories}

For the purpose of analyzing the coarse geometry, resp.~coarse (co-)homology theories of a space it is often advantageous to pass over to a coarsely equivalent topological coarse space whose topology is simple (Roe \cite[Pages~13--15]{roe_index_coarse} or Bunke--Engel \cite[Section~6.8]{buen}). A functorial way to do so is the Rips complex construction.

\begin{defn}\label{defn_Rips_complex_finite scale}
The \emph{Rips complex} $P_R(X)$ of a metric space $X$ at scale $R\geq0$ is defined as the simplicial complex where any finite collection of points of $X$ with mutual distance less than or equal to $R$ defines a simplex.
\end{defn}
It should be said that we do not distinguish between simplicial complexes and their geometric realization in this context.
We have allowed $R=0$ in the definition to include the case $P_0(X)=X$. Note that this is in general only an equality of sets and the topologies agree only if $X$ is discrete.

For every $R\geq0$ the Rips complex $P_R(X)$ carries a canonical coarse structure for which the inclusion $X\subset P_R(X)$ is a coarse equivalence. The coarse inverse to this inclusion is simply any map which maps a point of $P_R(X)$ to the vertex of a simplex containing it. With these coarse structures, the canonical inclusions $P_R(X) \subset P_S(X)$ as  closed subcomplexes are coarse equivalences for all $S \geq R \geq 0$.

If $X$ is a discrete, proper metric space, then for every $R > 0$ the simplicial complex $P_R(X)$ is locally finite and hence a locally compact Hausdorff space. If, however, $X$ is not discrete and proper, then $P_R(X)$ is locally too big for our purposes and it is more advisable to consider only the Rips complex $P_R(Z)$ of some subset $Z$ which is a so-called discretization in the following sense.

\begin{defn}
Let $X$ be a metric space. We say a subspace $Z \subset X$ is a \emph{discretization} of $X$ if it is proper and discrete, and the inclusion $Z \to X$ is a coarse equivalence.
\end{defn}

If $X$ is a proper metric space, then it admits a discretization: choose an $\varepsilon > 0$ and then any maximal, $\varepsilon$-separated subset $Z \subset X$ will do the job.

The Rips complex $P_R(X)$ is always locally contractible. If $X$ is a hyperbolic metric space and $Z\subset X$ is a discretization, then it is known that $P_R(Z)$ is even contractible for sufficiently large $R>0$ (see \cite[III.3.23]{bridson_haefliger}).

We cannot expect this strong form of contractibility in the set-up, but there are weaker results in this direction: We know from \cite[Theorem~10.6]{Wulff_CoassemblyRingHomomorphism} that if $(X,H)$ is a combed discrete proper metric space, then for each $R>0$ there is $S>0$ such that $P_R(X)$ is a contractible subspace of $P_S(X)$. Further, in Theorem~\ref{thm:RipsCompactificationContractible} we will show that if $H$ is coherent and expanding, then for each $R>0$ there is $S>0$ such that $\combcompb{P_R(X)}{H}\subset \combcompb{P_S(X)}{H}$ is contractible.

For this reason, we cannot investigate the $P_R(X)$ separately, but we have to consider the whole family $(P_R(X))_{R\geq 0}$ at once. This is best formalized by the notion of $\sigma$-locally compact spaces. This notion also provides a nice framework for introducing coarse (co-)homology theories and transgression maps. All of this is done in this Section.

\subsection{\texorpdfstring{$\sigma$}{sigma}-locally compact spaces and Rips complexes}

\begin{defn}[{cf.~\cite[Section 2]{EmeMey}, \cite[Definition 3.1]{Wulff_CoassemblyRingHomomorphism}}]
A \mbox{$\sigma$-locally} compact space $\cX$ is an increasing sequence $X_0\subset X_1\subset X_2\subset X_3\subset\ldots$ of locally compact Hausdorff spaces such that for all $m \le n$ the space $X_m$ is closed in $X_n$ and carries the subspace topology.

We call $\cX$ a \emph{$\sigma$-compact space}, if all $X_n$ are compact Hausdorff spaces.
\end{defn}

By an abuse of notation we will use the symbol $\cX$ for the set $\cX = \bigcup_{n \in \IN} X_n$ endowed with the final topology, i.\,e.~a subset $\cA \subset \cX$ is open/closed if and only if every intersection $A_n \coloneqq  \cA \cap X_n$ is open/closed.
The final topology has the property that a map from $\cX$ into another topological space is continuous if and only if it is continuous on every $X_n$.

Our main examples of $\sigma$-locally compact spaces originate from the Rips complex construction.
\begin{example}\label{ex_Rips_complex}
Let $X$ be a proper discrete metric space. Then the sequence of Rips complexes $P_0(X)\subset P_1(X)\subset P_2(X)\subset\dots$ is a $\sigma$-locally compact space denoted by $\cP(X)$. We call this space the \emph{Rips complex} of $X$.
\end{example}

Let us define what it means to $\sigma$-compactify a $\sigma$-locally compact space.

\begin{defn}
Let $\cX=\bigcup_{n\in\N}X_n$ be a $\sigma$-locally compact space. We will say that a $\sigma$-compact space $\overline{\cX}=\bigcup_{n\in\N}\overline{X_n}$ is a \emph{$\sigma$-compactification} of $\cX$ if each $\overline{X_n}$ is a compactification of $X_n$.
\end{defn}

It is straightforward to check that in this case $\cX$ is an open subspace of $\overline{\cX}$ in the sense of the following definition.

\begin{defn}
We say that $\cX=\bigcup_{n\in\N}X_n$ is a \emph{subspace} of $\cY=\bigcup_{n\in\N}Y_n$ if $\cX\subset \cY$, $\cX\cap Y_n = X_n$ and each $X_n$ carries the subspace topology of $Y_n$.
\end{defn}

\begin{example}\label{ex:Ripscomplexcoarsecompactifications}
Let $\overline{X}$ be a Higson dominated compactification of a proper metric space $X$ and denote its corona by $\partial X\coloneqq \overline{X}\setminus X$.
If $Z$ is a discretization of $X$, then $\overline{Z}\coloneqq Z\cup \partial X$ with the subspace topology of $\overline{X}$ is a Higson dominated compactification of $Z$ with the same corona $\partial X$.

Passing to the Rips complexes of $Z$, there is a canonical Higson dominated compactification $\overline{P_R(Z)}$ of $P_R(Z)$ with the same corona $\partial X$: One way to construct it is as the maximal ideal space of the sub-$C^*$-algebra of $C_h(P_R(Z))$ consisting of those functions whose restriction to $Z\subset P_R(Z)$ has a continuous extension to $\overline{Z}$.

These compactifications have the property that all the inclusions
\[X\supset Z\subset P_R(Z)\subset P_S(Z)\]
for all $0<R\leq S$ have continous extensions to the compactifications and that these extensions are the identity on the coronas $\partial X$.
In particular,
\[\overline{\cP(Z)}\coloneqq \bigcup_{R\in\N}\overline{P_R(Z)}\]
is a $\sigma$-compactification of $\cP(Z)$ with corona $\partial X\subset\overline{\cP(Z)}$.
\end{example}

\begin{example}\label{ex_sigmacombingcompactification}
As a special case of the previous example we can consider combing compactifications.

Let $(X,H)$ be a properly combed proper metric space and let $Z\subset X$ be a discretization. Then we obtain (by Remark~\ref{rem_combingdefinition}.(\ref{rem_enum_combedspacemorphisms}) and Lemma~\ref{lem_AdditionalPropertiesCoarselyInvariant}) proper combings on $Z$ and $P_R(Z)$, which we denote by the same letter $H$, and the combing coronas $\partial_HX$, $\partial_HZ$ and $\partial_H(P_R(Z))$ agree (by Corollary \ref{cor_coronafunctoriality}).

For every $0 < R \le S$ the inclusion $P_R(Z)\subset P_S(Z)$ induces a closed embedding $\combcompb{P_R(Z)}{H} \subset \combcompb{P_S(Z)}{H}$ of compact metrizable spaces. Hence we get a $\sigma$-compactification $\combcompb{\mathcal{P}(Z)}{H} \coloneqq  \bigcup_{R \in \IN}\combcompb{P_R(Z)}{H}$ with corona $\partial_HX$.
\end{example}

In order to study $\sigma$-compact and $\sigma$-locally compact spaces further, we first need to introduce adequate notions of continuous maps and homotopies between them.

\begin{defn}
Let $\cX$, $\cY$ be $\sigma$-locally compact spaces given by the filtrations $(X_m)_{m\in\N}$ and $(Y_n)_{n\in\N}$, respectively, and let $f\colon\cX\to\cY$ be a map.
\begin{itemize}
\item We call $f$ a \emph{$\sigma$-map} if for every $m\in\N$ there exists $n\in\N$ such that $f(X_m)\subset Y_n$.

Note that such an $f$ is continuous in the final topologies if and only if all the restricted maps $f|_{X_m}\colon X_m\to Y_n$ are continuous.
\item We call a continuous $\sigma$-map $f$ \emph{proper} if the preimages of all $\sigma$-compact subspaces of $\cY$ under $f$ are $\sigma$-compact subspaces of $\cX$, or equivalently, if the restricted maps $f|_{X_m}\colon X_m\to Y_n$ are proper continous maps.\qedhere
\end{itemize}
\end{defn}

\begin{example}
Let $X$, $Y$ be proper discrete metric spaces and $f\colon X\to Y$ a coarse map between them. Then affine linear extension over all simplices yields a proper continuous $\sigma$-map $\cP(f)\colon\cP(X)\to\cP(Y)$. If $f$ is compatible with proper combings $H_X$, $H_Y$ on $X$, $Y$, respectively, then $\cP(f)$ even extends to a continuous $\sigma$-map $\combcompb{\cP(X)}{H_X}\to\combcompb{\cP(Y)}{H_Y}$. These constructions are obviously functorial.
\end{example}

Given two $\sigma$-locally compact spaces $\cX$, $\cY$ with filtrations $(X_n)_{n\in\N}$ and $(Y_n)_{n\in\N}$, respectively, then $\cX\times\cY$ is also a $\sigma$-locally compact space defined by the filtration $(X_n\times Y_n)_{n\in\N}$. If both $\cX$ and $\cY$ are $\sigma$-compact, then so is $\cX\times\cY$.

The inclusions $\cX\to\cX\times\{y\}\subset \cX\times\cY$ and $\cY\to\{x\}\times \cY\subset\cX\times\cY$ as slices are $\sigma$-proper continuous $\sigma$-maps. 
If one of the two spaces, say $\cY$, is even $\sigma$-compact, then the corresponding projection map
$\cX\times\cY\to \cX$ is also a proper continuous $\sigma$-map. In particular, the latter applies when $\cY$ is simply the unit intervall $I\coloneqq [0,1]$, and hence we have a meaningful notion of homotopy of (proper) continous $\sigma$-maps.

This notion of homotopy is important in the context of the Rips complex, as the following example shows.

\begin{example}\label{ex:homotopyRipscomplex}
Let $f,g\colon X\to Y$ be two coarse maps between proper discrete metric spaces which are close to each other. Then the two proper continuous $\sigma$-maps $\cP(f),\cP(g)\colon \cP(X)\to \cP(Y)$ are canonically homotopic by affine linear interpolation.

Consequently, coarse equivalences induce homotopy equivalences. And in particular, if $Z_1$, $Z_2$ are two discretizations of the same proper (non-discrete) metric space $X$, then there is a canonical homotopy class of proper homotopy equivalences between $\cP(Z_1)$ and $\cP(Z_2)$.

If the two coarse maps $f$ and $g$ have continous extensions $\overline{f},\overline{g}\colon\overline{X}\to\overline{Y}$ between Higson dominated compactifications $\overline{X}$ of $X$ with corona $\partial X$ and $\overline{Y}$ of $Y$ with corona $\partial Y$, then $\cP(f)$ and $\cP(g)$ can be extended to continuous $\sigma$-maps between pairs of $\sigma$-compact spaces
\[\overline{\cP(f)},\overline{\cP(g)}\colon (\overline{\cP(X)},\partial X)\to(\overline{\cP(Y)},\partial Y)\,,\]
where $\overline{\cP(X)}$ and $\overline{\cP(Y)}$ are the $\sigma$-compactifications of Example \ref{ex:Ripscomplexcoarsecompactifications}.
\end{example}

\subsection{Steenrod (co-)homology theories for \texorpdfstring{$\sigma$}{sigma}-spaces}\label{sec:Steenrod}

Roe constructed in \cite{roe_coarse_cohomology} his \emph{transgression map}
\[H^p(\partial X)\to \HX^{p+1}(X)\]
for any proper metric space $X$ with Higson dominated compactification $\overline{X}$ whose corona $\partial X\coloneqq \overline{X}\setminus X$ is homeomorphic to a finite polyhedron.

In this section, we present more general transgression maps 
\[E^p(\partial X)\to \EX^{p+1}(X)\quad\text{and}\quad \EX_p(X)\to E_{p-1}(\partial X)\]
defined whenever $E_*$ (or $E^*$, resp.) is a \emph{generalized Steenrod (co-)homology theory for $\sigma$-locally compact spaces}. This notion will be introduced below in an axiomatic manner.

Also, our definition does not require the corona to be homeomorphic to a finite polyhedron. Any Higson dominated compactification will work.

After introducing our transgression maps, we will show that they are equivalent to Roe's in the case of Alexander--Spanier cohomology.

To get started, we need to introduce the following  categories of spaces.

\begin{defn}
\begin{itemize}
\item Let $\sigComp$ be the category whose objects are pairs $(\cX,\cA)$ of $\sigma$-compact spaces with $\cA\subset\cX$ a closed subspace and whose morphisms between $(\cX,\cA)$ and $(\cY,\cB)$ are continuous $\sigma$-maps $\cX\to\cY$ which restrict to maps $\cA\to\cB$.
\item The category $\sigLocComp$ has objects the $\sigma$-locally compact spaces and its morphisms between $\cX$ and $\cY$ are proper continuous $\sigma$-maps $\cU\to\cY$ where $\cU\subset\cX$ is an open subset.
\item The categories $\Comp$ and $\LocComp$ are the restrictions of the above categories to pairs of compact Hausdorff spaces and locally compact Hausdorff spaces, respectively. Equivalently, one can define them by removing all $\sigma$'s in the above two parts of the definition.\qedhere
\end{itemize}
\end{defn}

The definition of the morphisms in the categories $\LocComp$ and $\sigLocComp$ may appear a bit strange at first sight, but it turns out to be natural in the context of Steenrod (co-)homology theories, which we shall discuss in a minute.

For now, note that these morphisms of $\LocComp$ are exactly those which make $\LocComp$ equivalent to the opposite of the category of commutative $C^*$-algebras and $*$-homomorphisms via the Gelfand--Naimark theorem. So in particular, $K$-theory and $K$-homology of locally compact spaces are functors on $\LocComp$.

Similarily, the category $\sigLocComp$ is equivalent to the opposite of the category of commutative $\sigma$-$C^*$-algebras and one can define the $K$-theory of $\sigma$-locally compact spaces via Phillips' $K$-theory of $\sigma$-$C^*$-algebras (cf.~\cite{Phillips_KtheorysigmaCstar,phillips_Frechet}). This yields a contravariant $K$-theory functor defined on $\sigLocComp$.

Further, we note that there are canonical functors $\sigComp\to\sigLocComp$ and $\Comp\to\LocComp$ which map pairs $(\cX,\cA)$ of spaces to their difference $\cX\setminus\cA$ and morphisms $(\cX,\cA)\to(\cY,\cB)$ to their restrictions $\cX\setminus\cA\supset f^{-1}(\cY\setminus\cB)\to\cY\setminus\cB$.

\begin{defn}
A \emph{generalized Steenrod (co-)homology theory} $E_*$ (or $E^*$, respectively) for $\sigma$-compact spaces is a (co-)homology theory on the category $\sigComp$ which satisfies the homotopy, exactness and \emph{strong excision axiom}.
\end{defn}

Here, the strong excision axiom says that all morphisms $(\cX,\cA)\to(\cY,\cB)$ in $\sigComp$ which restrict to homeomorphisms (i.\,e.~bijective and continuous $\sigma$-maps whose inverses are also continuous $\sigma$-maps) $\cX\setminus\cA\to\cY\setminus\cB$, induce isomorphisms on the $E$-(co-)homology groups. In particular this applies for the canonical maps $(\cX,\cA)\to ((\cX\setminus\cA)^+,\{\infty\})$, where
\[(\cX\setminus\cA)^+=\bigcup_{n\in\N}(X_n\setminus A_n)^+=\bigcup_{n\in\N}(X_n\setminus A_n)\cup\{\infty\}\]
is the one-point-$\sigma$-compactification of $\cX\setminus\cA$. This construction shows that the groups $E_*(\cX,\cA)$ or $E^*(\cX,\cA)$, respectively, depend only on the difference $\cX\setminus\cA$. Even more, the functors $E_*$, $E^*$ factor through the category $\sigLocComp$ and it is easy to see that a (co-)homology theory on $\sigComp$ which satisfies the homotopy and exactness axioms is a generalized  Steenrod (co-)homology theory if and only if it factors through $\sigLocComp$.

\begin{defn}
Given a generalized Steenrod (co-)homology theory $E_*$ ($E^*$, respectively) for $\sigma$-compact spaces, we call the induced functor defined on the category $\sigLocComp$ a \emph{generalized Steenrod (co-)homology theory for $\sigma$-locally compact spaces} and denote it by $E_*^{\lf}$ ($E_c^*$, respectively).
\end{defn}

Of course, we understand a generalized Steenrod (co-)homology theory for compact Hausdorff spaces or for locally compact Hausdorff spaces to be defined in exactly the same way, just without all the $\sigma$'s.

There are some well known examples of generalized Steenrod (co-)homology theories for (locally) compact Hausdorff spaces, e.\,g.~$K$-theory, $K$-homology and Alexander--Spanier cohomology, but there exists also very general methods for constructing such theories from spectra (cf.~Carlsson--Pedersen \cite[Sections 5--7]{carlsson_pedersen}, or \cite[Theorem~A]{KahnKaminkerSchochet} although they construct such theories only on the category of compact metrizable spaces).
One can then go on and generalize them to $\sigma$-(locally) compact spaces in one of the following ways:
\begin{itemize}
\item If $E_*$ is a generalized Steenrod homology theory for compact Hausdorff spaces (e.\,g.~$K$-homology), then one can simply define
\[E_*(\cX,\cA)\coloneqq \varinjlim_{n\in\N}E_*(X_n,A_n)\,,\qquad E^{\lf}_*(\cX)\coloneqq \varinjlim_{n\in\N}E^{\lf}_*(X_n)\,,\]
respectively, where $\cX=\bigcup_{n\in\N}X_n$ and $\cA=\bigcup_{n\in\N}A_n$ are the filtrations. The axioms are readily verified. In the case of $K$-homology one usually omits the supersript $-^\lf$.
\item In the special case of $K$-theory, one can choose any $C^*$-algebra $D$ as coefficients and define $K^*(\cX;D)\coloneqq K_{-*}(C_0(\cX)\otimes D)$, where  $C_0(\cX)$ is a $\sigma$-$C^*$-algebra associated to $\cX$ and $K_*$ is $K$-theory of $\sigma$-$C^*$-algebras (cf.~\cite{Phillips_KtheorysigmaCstar,phillips_Frechet}). In the special case $D=\C$ we usually only write $K^*(\cX)\coloneqq K^*(\cX;\C)$.

These $K$-theory groups constitute a Steenrod cohomology theory for $\sigma$-locally compact spaces. In the case of $K$-theory, we omit the subscript $-_c$ just like we omit the superscript $-^\lf$ for $K$-homology.
\item If a generalized Steenrod cohomology theory $E^*$ for $\sigma$-compact spaces is given by spectra or cochain complexes, then one can simply take the derived inverse limit over the spectra or cochain complexes. In this case, the theory will satisfy a $\varprojlim^1$-sequence and the strong excision axiom can be verified using the five-lemma.
\end{itemize}

\subsection{Coarse theories and transgression maps}
\label{sec_coarse_theories_transgression_maps}

\begin{defn}
Let $X$ be a proper metric space and $E_*$ (or $E^*$, respectively) be a generalized Steenrod (co-)homology theory.

Then the \emph{coarse $E$-(co-)homology} of $X$ is defined as $\EX_*(X)\coloneqq E^{\lf}_*(\cP(Z))$ (or $\EX^*(X)\coloneqq E_c^*(\cP(Z))$, respectively), for $Z\subset X$ a discretization of $X$.
\end{defn}
By Example~\ref{ex:homotopyRipscomplex}, this definition is independent of the choice of discretization $Z$, and coarse maps $f\colon X\to Y$ induce corresponding homomorphisms $f_*\colon \EX_*(X)\to \EX_*(Y)$ and $f^*\colon \EX^*(Y)\to \EX^*(X)$, which depend only on the closeness class of $f$.

\begin{defn}\label{def:transgression}
Let $\overline{X}$ be a Higson dominated compactification of a proper metric space $X$ with corona $\partial X$ and consider a generalized Steenrod homology theory $E_*$ or a generalized Steenrod cohomology theory $E^*$. Then the associated \emph{transgression map} 
\[\EX_{p+1}(X)\to E_p(\partial X)\qquad\text{or}\qquad E^p(\partial X)\to \EX^{p+1}(X)\,,\]
respectively, is the connecting homomorphism in the long exact sequence associated to the pair $(\overline{\cP(Z)},\partial X)$, where $Z\subset X$ is a discretization and $\overline{\cP(Z)}$ is the $\sigma$-compactification constructed in Example \ref{ex:Ripscomplexcoarsecompactifications}.
\end{defn}

Exploiting functoriality under the map of pairs $(\overline{\cP(Z)},\partial X)\to (\{*\},\{*\})$, one sees that the homological version of transgression has image contained in 
\[\widetilde{E}_p(\partial X)\coloneqq \ker(E_p(\partial X)\to E_p(\{*\}))\]
and the cohomological version factors through
\[\widetilde{E}^p(\partial X)\coloneqq \operatorname{coker}(E^p(\{*\}\to E^p(\partial X))\,.\]
This leads to reduced versions of the transgression maps, which are much more relevant in the context of isomorphism conjectures. The following property is obvious:
\begin{lem}\label{lem_transgressionisoforcontractible}
If the $\sigma$-compact space $\overline{\cP(Z)}$ is contractible, then the \emph{reduced transgression maps}
\[\EX_{p+1}(X)\to \widetilde{E}_p(\partial X)\qquad\text{and}\qquad \widetilde{E}^p(\partial X)\to \EX^{p+1}(X)\]
are isomorphisms.\qed
\end{lem}

An important tool of calculating coarse (co-)homology theories (although not in this paper, but we mention it anyway for the sake of completeness) are coarse Mayer--Vietoris sequences as introduced in \cite{higson_roe_yu}. Recall that a pair $(X,Y)$ of closed subspaces $X_1,X_2\subset X$ of a proper metric space $X$ is called \emph{coarsely excisive}, if there is a function $\rho\colon \N\to\N$ such that $\dist(x,X_1)\leq n$ and $\dist(x,X_2)\leq n$ implies $\dist(x,X_1\cap X_2)\leq \rho(n)$ for all $x\in X$ and $n\in\N$.

\begin{prop}
If $(X_1,X_2)$ is coarsely excisive, then there are coarse Mayer--Vietoris sequences
\begin{align*}
\dots &\to \EX_p(X_1\cap X_2)\to \EX_p(X_1)\oplus \EX_p(X_2)\to \EX_p(X_1\cup X_2)\to
\\&\to \EX_{p-1}(X_1\cap X_2)\to \EX_{p-1}(X_1)\oplus \EX_{p-1}(X_2) \to\dots
\\\text{and}&
\\\dots &\to \EX^p(X_1\cup X_2)\to \EX^p(X_1)\oplus \EX^p(X_2)\to \EX^p(X_1\cap X_2)\to
\\&\to \EX^{p+1}(X_1\cup X_2)\to \EX^{p+1}(X_1)\oplus \EX^{p+1}(X_2) \to\dots
\end{align*}
for any coarse homology theory $\EX_*$ and any coarse cohomology theory $\EX^*$, respectively.
\end{prop}
\begin{proof}
If $\cA$ and $\cB$ are closed subsets of a $\sigma$-locally compact space, then a standard argument from algebraic topology shows that the strong excision property and the long exact sequences of a generalized Steenrod homology theory yield a Mayer-Vietoris sequence
\begin{align*}
\dots &\to E^{\lf}_p(\cA\cap\cB)\to E^{\lf}_p(\cA)\oplus E^{\lf}_p(\cB)\to E^{\lf}_p(\cA\cup \cB)\to
\\&\to E^{\lf}_{p-1}(\cA\cap\cB)\to E^{\lf}_{p-1}(\cA)\oplus E^{\lf}_{p-1}(\cB) \to\dots
\end{align*}
and similarily for cohomology.

We may assume without loss of generality that $X_1$ and $X_2$ are discrete and we apply the Mayer--Vietoris sequence for $\cA=\cP(X_1)$ and $\cB=\cP(X_2)$. We have the equality $\cP(X_1)\cap \cP(X_2)=\cP(X_1\cap X_2)$, but only an includion $\iota\colon \cP(X_1)\cup \cP(X_2)\subset \cP(X_1\cup X_2)$. It remains to show that the latter inclusion is a homotopy equivalence.

The coarse excisiveness allows us to define a map $\pi\colon X_1\cup X_2\to X_1\cap X_2$ with the properties that $\pi$ is the identity on $X_1\cap X_2$ and $d(\pi(x),x)\leq\rho(n)$ whenever $x\in X_{1,2}\setminus X_{2,1}$ satisfies $\dist(x,X_{2,1})\leq n$. 

We can now define a homotopy inverse $\tau\colon \cP(X_1\cup X_2)\to \cP(X_1)\cup \cP(X_2)$ as follows: An arbitrary point $x\in\cP(X_1\cup X_2)$ can be written as an affine linear combination
\[x=\sum_{i=1}^Ir_ix_i+\sum_{j=1}^Js_jy_j+\sum_{k=1}^Kt_kz_k\]
with $r_i,s_j,t_k\in [0,1]$ and $x_i\in X_1\setminus X_2$, $y_j\in X_1\cap X_2$, $z_k\in X_2\setminus X_1$ for all $i=1,\dots,I$, $j=1,\dots,J$ and $k=1,\dots,K$  and 
\[1=\sum_{i=1}^Ir_i+\sum_{j=1}^Js_j+\sum_{k=1}^Kt_k\,.\]
With this notation, we can define the two continuous functions
\[R,T\colon \cP(X_1\cup X_2)\to[0,1]\,,\quad R(x)\coloneqq \sum_{i=1}^Ir_i\,,\quad T(x)\coloneqq \sum_{k=1}^Kt_i\]
and the homotopy inverse $\tau$ by
\[\tau(x)\coloneqq \begin{cases}
\sum_{i=1}^I\frac{r_i}{R(x)}((R(x)-T(x))x_i+T(x)\pi(x_i)) &
\\\qquad+\sum_{j=1}^Js_jy_j+\sum_{k=1}^Kt_k\pi(z_k) & \text{if } R(x)\geq T(x),
\\\sum_{i=1}^Ir_i\pi(x_i)+\sum_{j=1}^Js_jy_j &
\\\qquad+\sum_{k=1}^K\frac{t_k}{T(x)}((T(x)-R(x))z_k+R(x)\pi(z_k)) & \text{if } T(x)\geq R(x).
\end{cases}\]
Note that these maps $R,T,\tau$ are well defined and continuous, the image of $\tau$ is contained in $\cP(X_1)\cup\cP(X_2)$ and we trivially have $\rho\circ\iota=\id$. Furthermore, $x\in P_n(X_1\cup X_2)$ implies $\tau(x)\in P_{n+2\rho(n)}(X_1\cup X_2)$, so $\tau$ is a $\sigma$-map. Finally, $\iota\circ\tau$ is homotopic to the identity via the affine linear homotopy.
\end{proof}

\subsection{Roe's coarse cohomology}
\label{sec_roes_coarse_cohomology}

In this section we compare our transgression map with Roe's. Here is the definition of Roe's coarse cohomology, which we have reformulated for arbitrary coarse spaces instead of just metric spaces and arbitrary coefficients $M$ instead of just $\R$:

\begin{defn}
Let $X$ be a coarse space and $M$ an abelian group.
\begin{itemize}
\item A subset $Q\subset X^{q+1}$ is called \emph{multicontrolled} if every two of the coordinate projections $Q\to X$ are close to each other
\item Denote by $C\!X^*(X;M)$ the cochain complex whose groups $C\!X^q(X;M)$ consist of all those functions
\[\phi\colon X^{q+1}\to M\]
whose support intersects each multicontrolled subset $Q\subset X^{q+1}$ in a bounded subset, and with coboundary maps defined by 
\begin{equation}\label{eq:coboundaryformula}
\delta\phi(x_0,\dots,x_{q+1})\coloneqq \sum_{i=0}^{q+1}(-1)^i\phi(x_0,\dots,\widehat{x_i},\dots,x_{q+1})\,.
\end{equation}
\item The \emph{coarse cohomology} $\HX^*(X;M)$ of $X$ with coefficients in $M$ is the cohomology of the cochain complex $C\!X^*(X;M)$.\qedhere
\end{itemize}
\end{defn}

This definition of coarse cohomology is modeled after the definition of Alexander--Spanier cohomology with compact support, which we review next.

\begin{defn}
Let $X$ be a locally compact Hausdorff space and let $M$ be an abelian group.
\begin{itemize}
\item A function $\phi\colon X^{q+1}\to M$ is called \emph{locally zero} if it vanishes on a neighborhood of the multidiagonal. It is called locally zero outside of a compact subset $K$ if the restriction $\phi|_{(X\setminus K)^{q+1}}$ is locally zero.
\item Denote by $C\!A_c^*(X;M)$ the cochain complex whose $q$-th group $C\!A_c^q(X;M)$ is the quotient of the group of all functions $\phi\colon X^{q+1}\to M$ which are locally zero outside of a compact subset, modulo the group of all locally zero functions. The coboundary map of this complex is given by the same formula \eqref{eq:coboundaryformula}.
\item The \emph{Alexander--Spanier cohomology with compact support} of $X$ with coefficients in $M$ is the cohomology $\HA_c^*(X;M)$ of the cochain complex $C\!A_c^*(X;M)$.
\item If $X$ is compact and $A\subset X$ is closed, we define the \emph{Alexander--Spanier cohomology of the pair $(X,A)$} with coefficients in $M$ as the cohomology of the cochain complex
\[C\!A^*(X,A;M)\coloneqq \ker\left(C\!A_c^*(X;M)\xrightarrow{\text{restriction}} C\!A_c^*(A;M)\right)\,.\qedhere\]
\end{itemize}
\end{defn}

It is known that $\HA^*(-,-;M)$ is a generalized Steenrod cohomology theory for compact Hausdorff spaces and $\HA_c^*(-;M)$ is a generalized Steenrod cohomology theory for locally compact Hausdorff spaces. Indeed, for any pair of compact Hausdorff spaces $(X,A)$ there is a natural cochain map $\operatorname{exc}\colon C\!A_c^*(X\setminus A;M)\to C\!A^*(X,A;M)$ which induces a natural isomorphism $\HA_c^*(X\setminus A;M)\xrightarrow{\cong}\HA^*(X,A;M)$. 

This cochain map can be defined as follows: any cochain in $C\!A_c^q(X\setminus A;M)$ can be represented by functions $\phi\colon (X\setminus A)^{q+1}\to M$ which are not only locally zero outside of a compact subset $K\subset X\setminus A$, but whose extensions $\Phi\colon X^{q+1}\to M$ by zero are even locally zero outside of the compact subset $X\setminus K$. 
Note that this property does not hold for any representative, as the neighborhood of the multidiagonal in $(X\setminus (A\cup K))^{q+1}$, on which $\phi$ vanishes, may become arbitrary thin when approaching $A$. The cochain map is then defined by mapping $[\phi]$ to $[\Phi]$ for these special type of representatives $\phi$, and the additional assumption on the representatives ensures that we obtain a well-defined cochain map.

We refer to \cite[Section 6.6]{spanier} for details on the proof of the property of strong excision.

Now, for any topological coarse space $X$ the support condition on the cochains in $C\!X^q(X)$ are such that there is an obvious cochain map
\[C\!X^*(X;M)\to C\!A_c^*(X;M)\,.\]
\begin{defn}
For any topological coarse space $X$ the \emph{character map}
\[\HX^*(X;M)\to \HA_c^*(X;M)\]
is the homomorphism induced by the above cochain map.
\end{defn}

Roe showed that the character map is an isomorphism for some spaces, for example for uniformly contractible spaces. In the following we are going to show that it is always an isomorphism for the full Rips complex $\cP(Z)$ of a proper, discrete metric space $Z$. Of course, we have to generalize the above notions to $\sigma$-spaces first.

\begin{defn}
Let $Z$ be a proper discrete metric space and $M$ an abelian group. Then we define $\HX^*(\cP(Z);M)$ to be the cohomology of the cochain complex
\[C\!X^*(\cP(Z);M)\coloneqq \varprojlim_{R\in\N}C\!X^*(P_R(Z);M)\,.\qedhere\]
\end{defn}
\begin{lem}\label{lem_HXHPX}
Let $X$ be a proper metric space and $Z\subset X$ a discretization. Then $\HX^*(X;M)$ is canonically isomorphic to $\HX^*(\cP(Z);M)$.
\end{lem}
\begin{proof}
This follows from the $\varprojlim^1$-sequence
\[0\to {\varprojlim_{R\in\N}}^1\HX^{p-1}(P_R(Z);M)\to \HX^p(\cP(Z);M)\to \varprojlim_{R\in\N}\HX^p(P_R(Z);M)\to 0\]
and the fact that all the inclusions
$X\supset Z\subset P_R(Z)\subset P_S(Z)$ for $0\leq R\leq S$ induce isomorphsism in coarse cohomology.
\end{proof}

\begin{defn}
For a $\sigma$-locally compact space $\cX=\bigcup_{n\in\N}X_n$ and an abelian group $M$ we define $\HA_c^*(\cX;M)$ to be the cohomology of the cochain complex
\[C\!A_c^*(\cX;M)\coloneqq \varprojlim_{n\in\N}C\!A_c^*(X_n;M)\,.\]
If $\cX$ is even $\sigma$-compact and $\cA=\bigcup_{n\in\N}A_n\subset\cX$ is closed, then we define $\HA^*(\cX,\cA;M)$ to be the cohomology of the cochain complex
\[C\!A^*(\cX,\cA;M)\coloneqq \varprojlim_{n\in\N}C\!A^*(X_n,A_n;M)\,.\qedhere\]
\end{defn}

\begin{lem}
The cohomology groups $\HA^*(\cX,\cA;M)$ and $\HA_c^*(\cX;M)$ defined above constitute generalized Steenrod cohomology theories for $\sigma$-compact spaces and $\sigma$-locally compact spaces, respectively.
\end{lem}

\begin{proof}
Functoriality, the homotopy axiom and exactness of $\HA^*(-,-;M)$ are straightforward. For the strong excision property we note that there is a natural diagram with exact rows
\begin{align*}
\mathclap{\xymatrix{
0\ar[r]&{\varprojlim\limits_{n\in\N}}^1\HA_c^{p-1}(X_n\setminus A_n;M)\ar[r]\ar[d]^-{\cong}& \HA_c^p(\cX\setminus\cA;M)\ar[r]\ar[d]& \varprojlim\limits_{n\in\N}\HA_c^p(X_n\setminus A_n;M)\ar[r]\ar[d]^-{\cong}& 0
\\0\ar[r]&{\varprojlim\limits_{n\in\N}}^1\HA^{p-1}(X_n,A_n;M)\ar[r]& \HA^p(\cX,\cA;M)\ar[r]& \varprojlim\limits_{n\in\N}\HA^p(X_n,A_n;M)\ar[r]& 0
}}
\end{align*}
and the claim follows from the five-lemma.
\end{proof}

\begin{lem}\label{lem:charactermapisomorphism}
Let $Z$ be a proper discrete metric space and $M$ any abelian group.
Then the obvious character map $\HX^*(\cP(Z);M)\to \HA_c^*(\cP(Z);M)$ is an isomorphism.

Consequently, for any proper metric space $X$ there is a canonical isomorphism $\HX^*(X;M)\cong H\!A\!X^*(X;M)$.
\end{lem}

\begin{proof}
The cochain groups $C\!X^q(\cP(Z);M)$ can be described equivalently as the group of all functions $\phi\colon (\cP(Z))^{q+1}\to M$ such that for all $R\in\N$ and for all penumbras $Q\subset (P_R(Z))^{q+1}$ the intersection of the support of $\phi$ with $Q$ is bounded in $(P_R(Z))^{q+1}$. 

In order to give a similar alternative description of the cochain groups $C\!A_c^q(\cX;M)$ for $\cX=\bigcup_{n\in\N}X_n$ a $\sigma$-locally compact space, we introduce the following notion: A subset $\cU=\bigcup_{n\in\N}U_n\subset \cX$ is a \emph{$\sigma$-neighborhood} of a subset $\cA=\bigcup_{n\in\N}A_n\subset\cX$ if  each $U_n$ is a neighborhood of $A_n$ in $X_n$.
Note that any neighborhood in the final topology is a $\sigma$-neighborhood, but the converse is not true.

We now say that a function $\phi\colon \cX^{q+1}\to M$ is $\sigma$-locally zero if it vanishes on a $\sigma$-neighborhood of the multidiagonal, and we call it $\sigma$-locally zero outside of a $\sigma$-compact subset $\cK\subset\cX$ if the restriction $\phi|_{(\cX\setminus\cK)^{q+1}}$ is locally zero. One now checks that $C\!A_c^q(\cX;M)$ can be described  as the group of all functions $\phi\colon \cX^{q+1}\to M$ which are $\sigma$-locally zero outside of a $\sigma$-compact subset modulo the subgroup of all $\sigma$-locally zero functions.

The character map is again induced by the canonical quotient map 
\[c\colon C\!X^*(\cP(Z);M)\to C\!A_c^*(\cP(Z);M)\,.\]
It is straightforward to see that this cochain map is surjective, so it suffices to prove that its kernel $C_{\ker}^*\coloneqq \ker(c)$ is acyclic.
The groups $C_{\ker}^q$ consist exactly of all those functions $\phi\colon (\cP(Z))^{q+1}\to M$ which vanish on some $\sigma$-neighborhood of the multidiagonal and whose support intersects each penumbra $Q\subset (P_R(Z))^{q+1}$ in a bounded subset of $(P_R(Z))^{q+1}$.

The proof of the acyclicity will make use of iterated barycentric subdivisions. To this end, we introduce the following notation: given a $(q+1)$-tuple $x=(x_0,\dots,x_q)\in (\cP(Z))^{q+1}$, let $\sigma_x\subset \cP(Z)$ be the affine linear $q$-simplex spanned by $x_0,\dots, x_q$. Conversely, we denote by $x_\sigma$ the $(q+1)$-tuple of vertices of such a simplex $\sigma$.

For each affine linear $q$-simplex $\sigma$ we denote by $\Sigma(\sigma)$ the set of $q$-simplices in the barycentric subdivision of $\sigma$. Furthermore, let $C_{\operatorname{tot}}^*$ denote the cochain complex whose $q$-th group consists of \emph{all} functions $\phi\colon (\cP(Z))^{q+1}\to M$. Then the barycentric subdivision operator $S$ is a cochain map from $C_{\operatorname{tot}}^*$ to itself defined by the formula
\[S\phi(x)\coloneqq \sum_{\sigma\in\Sigma(\sigma_x)}\phi(x_\sigma)\,.\]

We first claim that $S$ maps the subcomplex $C_{\ker}^*\subset C_{\operatorname{tot}}^*$ to itself. Let $\phi\in C_{\ker}^q$.
The essential property which we need, and which is readily verified, is the following: for each $R\in\N$ and each multicontrolled subset $Q\subset(P_R(Z))^{q+1}$ there is an $R'\in\N$ such that
\[Q'\coloneqq \bigcup_{x\in Q}\sigma_x\]
is contained in $(P_{R'}(Z))^{q+1}$ and even is multicontrolled in it. Therefore the intersection $B'=\supp(\phi)\cap Q'$ is bounded in $(P_{R'}(Z))^{q+1}$. From this one obtains that the intersection $\supp(S\phi)\cap Q$ is contained in the bounded subset
\[B\coloneqq \{x\in (P_R(Z))^{q+1}\mid  \sigma_x\cap B'\not=\emptyset\}\]
of $(P_R(Z))^{q+1}$.
Further, if $\phi\colon (\cP(Z))^{q+1}\to M$ vanishes on a $\sigma$-neighborhood $\cU\subset(\cP(Z))^{q+1}$ of the multidiagonal, then
\[\cV\coloneqq \{x_\sigma\mid \sigma\subset \cU\}\subset(\cP(Z))^{q+1}\]
is a $\sigma$-neighborhood of the multidiagonal on which $S\phi$ vanishes. This finishes the proof that $S$ maps $C_{\ker}^*\subset C_{\operatorname{tot}}^*$ to itself.

There is also a well-known cochain homotopy between $S$ and the identity, which we denote by 
\[T\colon C_{\operatorname{tot}}^*\to C_{\operatorname{tot}}^{*-1}\,,\qquad \delta T+T\delta =\id-S\,.\]
It is no surprise that $T$ also restricts to a cochain homotopy $T\colon C_{\ker}^*\to C_{\ker}^{*-1}$ between the restricted operator $S$ and the identity. The proof of this fact is basically exactly the same proof as the one above, just with slightly more complicated formulas. Therefore, we omit it.

By iteration we obtain that
\[T^k\coloneqq T+ST+S^2T+\dots+S^{k-1}T\colon C_{\ker}^*\to C_{\ker}^{*-1}\]
is a cochain homotopy equivalence between $S^k$ and the identity.

Now, let $\phi\in C_{\ker}^q$ and $\cU\subset(\cP(Z))^{q+1}$ be a $\sigma$-neighborhood of the multidiagonal on which $\phi$ vanishes. For each $(q+1)$-tuple let $R\in\N$ be such that $\sigma_x\subset P_R(Z)$. Then there is $N\in\N$ such that each $q$-simplex of the iterated barycentric subdivision $\Sigma^k(\sigma_x)$ is contained in the open neighborhood $U_R=\cU\cap (P_R(Z))^{q+1}$ of the multidiagonal in $(P_R(Z))^{q+1}$ for all $k\geq N$. Hence, $(S^k\phi)(\sigma)=0$ for all $k\geq N$ and we can define $T^\infty\phi\in C_{\ker}^{q-1}$ simplexwise by the formula
\[T^\infty\phi(\sigma)\coloneqq \lim_{k\to\infty}T^k\phi(\sigma)\,.\]
This obviously defines a cochain contraction of $C_{\ker}^*$ and we are done.
\end{proof}

Finally we have the following theorem.
\begin{thm}
Assume that $\partial X$ is homeomorphic to a finite polyhedron, such that Roe's transgression map $\HA^p(\partial X;M)\to \HX^{p+1}(X;M)$ is defined (see \cite[Section 5.3]{roe_coarse_cohomology}). Then the latter can be identified with the transgression map $\HA^p(\partial X;M)\to H\!A\!X^{p+1}(X;M)$ of Definition \ref{def:transgression} under the isomorphism of Lemma \ref{lem:charactermapisomorphism}.
\end{thm}
We omit the proof, since it is essentially \cite[Proposition 5.25 (iii)]{roe_coarse_cohomology}, which says that the transgression map, character map, strong excision and coboundary map in Alexander--Spanier cohomology fit into the commutative diagram
\[\xymatrix{
\HA^p(\partial X;M)\ar[r]\ar[d]&\HX^{p+1}(X;M)\ar[d]
\\
\HA^{p+1}(\overline{X},\partial X;M)&\HA_c^{p+1}(X;M) \ar[l]_-{\cong}
}\]
Roe noted that this is true already on the level of cochains and that its proof is a straightforward consequence of the definitions. The same is true for the proof of our theorem, since the isomorphism of Lemma \ref{lem:charactermapisomorphism} is nothing else but a character map, just with $\cP(Z)$ instead of the proper metric space $X$.

\section{Contractible compactifications}
\label{sec_contr_compact}

We have already seen in Lemma~\ref{lem_transgressionisoforcontractible} the importance of finding $\sigma$-contractible $\sigma$-compactifications of the full Rips complex. For properly combable spaces, the obvious approach is of course to try to construct a $\sigma$-contraction of the combing $\sigma$-compactification of the Rips complex by following approximately the combing lines. For the un-$\sigma$-compactified Rips complex, this is already known to work in full generality:

\begin{thm}[{See \cite[Theorem 10.6]{Wulff_CoassemblyRingHomomorphism} and Remark \ref{rem_combingdefinition}.(\ref{rem_enum_combingvscoarsecontraction}}]
Let $X$ be a combable, discrete, proper metric space.
Then $\cP(X)$ is contractible by a non-proper but continuous $\sigma$-homotopy.
\end{thm}

In the case $X$ is a finitely generated group $G$, the above theorem is well-known since in this case the Rips complex $\cP(G)$ is a model for the classifying space $\underline{EG}$ of proper $G$-actions and hence contractible.

Unfortunately, it is in general not possible to extend the $\sigma$-contraction of $\cP(X)$ continuously to the combing-$\sigma$-compactification. We will need the assumptions of expandingness and coherence of the combing to construct such a $\sigma$-contraction directly. This construction is carried out in the first part of this section. Afterwards, we take a look at implications of $\sigma$-contractibility to isomorphism conjectures.

\subsection{Construction of the contractions}
\label{sec_constructionofcontractions}

Analogously to the approach of Roe \cite{roe_hyperbolic} we will first construct a pseudo-continuous extension of $H$ to the compactification.

\begin{defn}
A map $f\colon Y\to X$ from a topological space $Y$ into a  metric space $X$ is defined to be \emph{$R_{pc}$-pseudo-continuous}, if for every $y\in Y$ the set $f^{-1}(B_{R_{pc}}(f(y)))$ is a neighborhood of $y$. 

We will call it \emph{pseudo-continuous} if it is $R_{pc}$-pseudo-continuous for an $R_{pc}>0$.
\end{defn}
Note that this is actually a property for maps from topological into coarse spaces, but we will only use it for maps into proper metric spaces: since our proofs below are already very technical we did not try to carry them out in the more general setting of proper topological coarse spaces.

Recall that for proper metric spaces $X$ the combing compactification $\combcompb{X}{H}$ is metrizable by Lemma \ref{lem:metrizability} and hence second-countable.

\begin{defn}
Let $(X,d)$ be a proper metric space equipped with a proper combing $H$.
Then a \emph{standard extension}
\[\overline{H}\colon\combcompb{X}{H} \times\N\to X\]
of $H$ is an extension which is constructed in the following way:

For $x\in\partial_HX$ we choose a sequence $(x_k)_{k\in\N}$ in $X$ converging to $x$. Then we define recursively $\overline{H}_n(x)$ and sub-sequences $(x_k^n)_{k\in\N}$ for all $n\in\N$ as follows:
\begin{itemize}
\item We define $\overline{H}_0(x)\coloneqq p$ and $x_k^0\coloneqq x_k$ for all $k$.
\item For $n>0$, the sequence $(H_n(x_k^{n-1}))_{k\in\N}$ is bounded by Lemma~\ref{lem453ewr43} and thus contains an accumulation point. We let $\overline{H}_n(x)$ be this accumulation point and let $(x_k^n)_{k\in\N}$ be a subsequence of $(x_k^{n-1})_{k\in\N}$ such that $H_n(x_k^n)\xrightarrow{k\to\infty}\overline{H}_n(x)$.\qedhere
\end{itemize}
\end{defn}

Here are some first properties of standard extensions.
\begin{lem}\label{lem:StandardExtensionBasicProperties}
With the notation of the preceding definition (in particular this means $x\in\partial_HX$) we have:
\begin{enumerate}
\item \label{lem:StandardExtension:Hnsubseqconv} For all pairs of integers $m$ and $n$ with $m\geq n$ we even have 
\[\overline{H}_n(x)=\lim_{k\to\infty}H_n(x_k^n)=\lim_{k\to\infty}H_n(x_k^{m})\,.\]
\item \label{lem:StandardExtension:Hnxconv} $\overline{H}_n(x)\xrightarrow{n\to\infty}x$.
\item \label{lem:StandardExtension:SliceCloseness}
Let $R_1>0$ be a constant such that $H$ maps points of distance of at most $1$ to points of distance of at most $R_1$. Then the maps $\overline{H}_{n+1}$ and $\overline{H}_n$ are $R_1$-close for all $n\in\N$.
\end{enumerate}
\end{lem}

We cannot expect more continuity properties of $\overline{H}$ than in the second part of the lemma right now, for example since $X$ might even be discrete.

\begin{proof}
The first part holds trivially, because $(x_k^{m})_{k\in\N}$ is a subsequence of $(x_k^n)_{k\in\N}$.
For the second part, note that for all $f\in C(\combcompb{X}{H})\cong C_H(X)$ we have
\begin{align*}
\limsup_{n\to\infty}|f(\overline{H}_n(x))-f(x)|&=\limsup_{n\to\infty}\lim_{k\to\infty}|f(H_n(x_k^n))-f(x_k^n)|
\\&\leq\limsup_{n\to\infty}\|H_n^*f-f\|=0\,.
\end{align*}
Since $C(\combcompb{X}{H})$ separates points, this implies the claim.

Finally, for the third part
we use the first statement to show
\begin{align*}
d\left(\overline{H}_{n+1}(x), \overline{H}_n(x)\right)&=d\left(\lim_{k \to \infty} H_{n+1}(x_k^{n+1}),\lim_{k \to \infty} H_{n}(x_k^{n+1})\right)
\\&=\lim_{k \to \infty}d\left(H_{n+1}(x_k^{n+1}),H_{n}(x_k^{n+1})\right)\leq R_1\,.\qedhere
\end{align*}
\end{proof}

\begin{lem}\label{lem:ExtensionCoherence}
Assume that the proper combing $H$ on the proper metric space $(X,d)$ is $R_{coh}$-coherent for an $R_{coh}\ge 0$. Then the coherence passes over to the standard extension $\overline{H}$ in the sense that there is a constant $\overline{R}_{coh}\geq R_{coh}$ such that
\[d(H_m(\overline{H}_n(x)),\overline{H}_m(x))\leq \overline{R}_{coh}\]
for all $x\in \combcompb{X}{H}$ and natural numbers $m\leq n$.
\end{lem}

\begin{proof}
The inequality is of course true for $x\in X$ and arbitrary $\overline{R}_{coh}\geq R_{coh}$. 
For $x\in\partial_HX$ we use the notations from the previous definition and the lemma.
For each pair of integers $m$ and $n$ with $m \leq n$ we choose $k\in\N$ such that $d(\overline{H}_n(x),H_n(x_k^n))<1$ and $d(\overline{H}_m(x),H_m(x_k^n))<1$. This is possible by Part \ref{lem:StandardExtension:Hnsubseqconv} of the preceding lemma.
Then we have
\begin{align*}
d(H_m(\overline{H}_n(x)),\overline{H}_m(x))
&\leq d(H_m(\overline{H}_n(x)),H_m(H_n(x_{k}^n)))
\\&\qquad +d(H_m(H_n(x_{k}^n)),H_m(x_{k}^n)))
\\&\qquad +d(H_m(x_{k}^n)),\overline{H}_m(x))
\\&\leq R_1+R_{coh}+1=:\overline{R}_{coh}\,.\qedhere
\end{align*}
\end{proof}

\begin{lem}\label{lem:ExtensionPseudocontinuity}
Let $(X,d)$ be a proper metric space which is equipped with an expanding and coherent combing $H$.

Then the standard extension $\overline{H}$ is pseudo-continuous.
\end{lem}

\begin{proof}
We let $R_{exp}>0$ and $R_{coh}>0$ be the expandingness- and coherence-constants. Furthermore, we again use all the notation introduced so far in this section. Clearly, $H$ itself is $R_{pc}$-pseudo-continuous for all $R_{pc}\geq R_1$ for the $R_1$ defined in Lemma~\ref{lem:StandardExtensionBasicProperties}.\ref{lem:StandardExtension:SliceCloseness}.
Thus it remains to verify pseudo-continuity at the boundary.

Let $x\in\partial_HX$ and $n\in\N$. In the following we are going to show that 
\begin{equation}
\label{eq12343rwe32er}
U\coloneqq \overline{H}_n^{-1}(B_{R_{pc}}(\overline{H}_n(x)))
\end{equation}
is a neighbourhood of $x$ for all $R_{pc}\geq R_{exp}+2R_{coh}+2\overline{R}_{coh}$. To this end, we need a function $f\in D_H(X)$ whose extension to $\combcompb{X}{H}$ as in Lemma \ref{lem:DHExtensions} vanishes outside of $U$ and takes the value $1$ at $x$. 
The function $f$ will  be constructed in Urysohn-style by a sequence of nested subsets of $U\cap X$.

\proofsubdivisor{Preliminary step to the construction of the function $f$.}
To simplify our exposition, let us introduce the following notation. For two subsets $A,B\subset X$ and an $R>0$ we will write $A\neighbsub{R}B$ if the $R$-neighborhood $B_R(A)$ of $A$ is contained in~$B$. We remark that one should be cautious when using this notion, since we always have
\[A\neighbsub{R}B\neighbsub{S}C\implies A\neighbsub{\max\{R,S\}}C\,,\]
but the implication $A\neighbsub{R}B\neighbsub{S}C\implies A\neighbsub{R+S}C$ is in general only true in path metric spaces. 

Furthermore, 
let $S_0\coloneqq R_{exp}+2R_{coh}\leq S_1\leq S_2\leq\ldots$ be a sequence of real numbers such that 
for each $k\in\N$ and $(y,m)\in X\times\N$ the map
$H$ maps the ball of radius $k$ around $(y,m)$ into the ball of radius $S_k$ around $H_{m}(y)$.
We also define $I_k\coloneqq (2^{-k}\cdot \Z)\cap[0,1]$ for all $k\in\N$.

We are now going to construct an increasing sequence of natural numbers $n=n_0\leq n_1\leq n_2\leq\ldots$, subsets $A^+_{k,\lambda} \subset X$ for all $k\in\N$, $\lambda\in I_k\setminus\{1\}$ and subsets $A^-_{k,\lambda} \subset X$ for all $k\in\N$, $\lambda\in I_k\setminus\{0\}$ with the  properties listed below, where we have set $C^\pm_{k,\lambda}\coloneqq H_{n_k}^{-1}(A^\pm_{k,\lambda})$:
\begin{enumerate}[(i)]
\item\label{enum:nestA} For each fixed $k\in\N$, the sets $A^\pm_{k,\lambda}$ are nested by the inclusions
\begin{itemize}
\item $A^+_{k,\lambda}\subset A^-_{k,\lambda}$ for $\lambda\in I_k\setminus\{0,1\}$, and
\item $A^-_{k,\lambda+2^{-k}}\neighbsub{S_k} A^+_{k,\lambda}$ for $\lambda\in I_k\setminus\{1\}$.
\end{itemize}
\item\label{enum:nestCa} A direct consequence of (\ref{enum:nestA}) is that we have the inclusions $C^+_{k,\lambda}\subset C^-_{k,\lambda}$  and $C^-_{k,\lambda+2^{-k}}\neighbsub{k} C^+_{k,\lambda}$ for all $k,\lambda$ as above.
\item\label{enum:nestCb}
\begin{itemize}
\item For all $k\in\N\setminus\{0\}$, $\lambda\in I_{k-1}\setminus\{1\}$ we have $C^+_{k,\lambda}\subset C^+_{k-1,\lambda}$.
\item For each $k\in\N\setminus\{0\}$ there is a bounded subset $L_k\subset X$ such that $C^-_{k-1,\lambda}\setminus L_k\subset C^-_{k,\lambda}$ for all $\lambda\in I_{k-1}\setminus\{0\}$.
\end{itemize}
\item Finally, $C_{0,0}^+\subset U\cap X$
and $C_{0,1}^-$ contains all $\overline{H}_m(x)$ with $m\geq n$. Note that $n\in\N$ and $U\subset\combcompb{X}{H}$ have been fixed at the beginning of the proof.\label{enum:nest0}
\end{enumerate}
For $k=0$ we can set $A^-_{0,1}\coloneqq B_{\overline{R}_{coh}}(\overline{H}_n(x))$ and $A^+_{0,0}\coloneqq B_{\overline{R}_{coh}+S_0}(\overline{H}_n(x))$. We obtain $C^-_{0,1}=H_n^{-1}(A^-_{0,1})$ and $C^+_{0,0}=H_n^{-1}(A^+_{0,0})$. 
Properties (\ref{enum:nestA}), (\ref{enum:nestCa}) and (\ref{enum:nest0}) are then clear and (\ref{enum:nestCb}) not yet applicable.

Assume now that the sets $A^\pm_{k',\lambda}$ and $ C^\pm_{k',\lambda}$ have already been constructed for all $k'<k$.
The $R_{exp}$-expandingness condition on $H$ then gives us a bounded subset $K_k\coloneqq K_{2S_k,n_{k-1}}$ such that $H_{n_{k-1}}(B_{2S_k}(y))\subset B_{R_{exp}}(H_{n_{k-1}}(y))$ for all $y\in X\setminus K_k$. We then define for each $l=1,\dots,2^{k-1}$
\begin{align}
A^-_{k,2l\cdot 2^{-k}}&\coloneqq H_{n_{k-1}}^{-1}(B_{R_{coh}}(A^-_{k-1,l\cdot 2^{1-k}}))\setminus B_{2l\cdot S_k}(K_k),\label{eq43trzretz}\\
A^+_{k,(2l-1)\cdot 2^{-k}}&\coloneqq A^-_{k,(2l-1)\cdot 2^{-k}}\coloneqq  B_{S_k}(A^-_{k,2l\cdot 2^{-k}}),\notag\\
A^+_{k,2(l-1)\cdot 2^{-k}}&\coloneqq B_{2S_k}(A^-_{k,2l\cdot 2^{-k}}).\notag
\end{align}
The number $n_k$ will be chosen below in the proof of \eqref{enum:nestCa}. We now have to show the Points~(\ref{enum:nestA})--(\ref{enum:nestCb}) from the above list of properties for these subsets. Point~(\ref{enum:nest0}) holds by construction, see above the case $k=0$.
\begin{enumerate}[(i):]
\item We clearly have
\[A^-_{k,2l\cdot 2^{-k}}\neighbsub{S_k}A^+_{k,(2l-1)\cdot 2^{-k}}=A^-_{k,(2l-1)\cdot 2^{-k}}\neighbsub{S_k}A^+_{k,2(l-1)\cdot 2^{-k}}\,.\]
These shows three quarter of the inclusions in (\ref{enum:nestA}). Now, if
\[y\in A^+_{k,2l\cdot 2^{-k}}=B_{2S_{k}}(  H_{n_{k-1}}^{-1}(B_{R_{coh}}(A^-_{k-1,(l+1)\cdot 2^{1-k}}))\setminus B_{2(l+1)\cdot S_k}(K_k)  )\,,\] then we clearly have $y\notin B_{2l\cdot S_k}(K_k)$, so in particular $y\notin K_k$ and hence the expansion condition on $H$ together with (\ref{enum:nestA}) for $k-1$ instead of $k$ implies
\[H_{n_{k-1}}(y)\in B_{R_{exp}+R_{coh}}(A^-_{k-1,(l+1)\cdot 2^{1-k}})\subset A^-_{k-1,l\cdot 2^{1-k}}\,.\]
Thus $y\in  A^-_{k,2l\cdot 2^{-k}}$, and this shows the remaining part of (\ref{enum:nestA}).
\item\label{345tr234} Define the bounded subset $L_k\coloneqq B_{2^k\cdot S_k}(K_k)$.
According to Lemma \ref{lem453ewr43} there is an $n_k\geq n_{k-1}$ such that $H_{n_k}(X\setminus L_k)\subset X\setminus L_k$.
With this $n_k$ we define $C^\pm_{k,\lambda}\coloneqq H_{n_k}^{-1}(A^\pm_{k,\lambda})$ and then (\ref{enum:nestCa}) is automatic.
\item It remains to show (\ref{enum:nestCb}). For $\lambda\in I_{k-1}\setminus\{0\}$ we have
\begin{align*}
C^-_{k-1,\lambda}\setminus L_k&\subset H_{n_{k-1}}^{-1}(A^-_{k-1,\lambda})\setminus H_{n_k}^{-1}(L_k)
\\&\subset H_{n_k}^{-1}(H_{n_{k-1}}^{-1}(B_{R_{coh}}( A^-_{k-1,\lambda} ))\setminus L_k )
\\&\subset H_{n_k}^{-1}(A^-_{k,\lambda})=C^-_{k,\lambda}\,.
\end{align*}
For $\lambda\in I_{k-1}\setminus\{1\}$ we exploit $A^-_{k,\lambda+2^{1-k}}\subset X\setminus K_k$, which follows from \eqref{eq43trzretz}, and \eqref{eq43trzretz} itself to get the fourth and fifth subset relation in the following chain
\begin{align*}
H_{n_{k-1}}(C^+_{k,\lambda}) & \subset B_{R_{coh}}(H_{n_{k-1}}(H_{n_{k}}(C^+_{k,\lambda})))
\\& \subset B_{R_{coh}}(H_{n_{k-1}}(A^+_{k,\lambda}))
\\& \subset B_{R_{coh}}(H_{n_{k-1}}(B_{2S_k}(A^-_{k,\lambda+2^{1-k}})))
\\& \subset B_{R_{coh}}(B_{R_{exp}} (H_{n_{k-1}} (A^-_{k,\lambda+2^{1-k}})))
\\& \subset B_{R_{coh}}( B_{R_{exp}} (B_{R_{coh}}( A^-_{k-1,\lambda+2^{1-k}} )))
\\& \subset B_{S_{k-1}}( A^-_{k-1,\lambda+2^{1-k}} )
\\& \subset A^+_{k-1,\lambda}
\end{align*}
and this implies $C^+_{k,\lambda}\subset C^+_{k-1,\lambda}$.
\end{enumerate}

\proofsubdivisor{The function $f$.}
Now that we know that a suitable collection of subsets exists, we can finally define the function 
\[f\colon X\to [0,1]\,,\quad y\mapsto \sup\{\lambda\mid y\in C^-_{k,\lambda},\,k\in\N,\,\lambda\in I_k\setminus\{0\}\}\]
(here, we understand the supremum of the empty set to be zero)
and we want to show that it is an element of $D_H(X)$. 

To this end, we also define the functions
\[f_k\colon X\to [0,1]\,,\quad y\mapsto \sup\{\lambda\mid y\in C^-_{k,\lambda},\,\lambda\in I_k\setminus\{0\}\}\]
and note that (\ref{enum:nestCa}) and (\ref{enum:nestCb}) together imply
\begin{equation}
\label{eq3243etw24}
f_k\leq f \qquad \text{and} \qquad f|_{X\setminus \tilde L_k}\leq f_k|_{X\setminus \tilde L_k}+2^{-k}\, ,
\end{equation}
where $\tilde L_k\coloneqq L_1\cup\dots\cup L_k$. Indeed, the first inequality is clear and for the second we use
\[C^-_{j,\lambda}\setminus\tilde L_k\subset C^-_{\max\{j,k\},\lambda}\subset C^+_{\max\{j,k\},\lambda-2^{-k}}\subset C^+_{k,\lambda-2^{-k}}\subset C^-_{k,\lambda-2^{-k}}\,.\]

Next, we note that the $R$-variation of $f_k$ is norm bounded by $2^{-k}$ for all $R<k$ by property (\ref{enum:nestCa}). Thus, given $R>0$ and $\varepsilon>0$, we can choose $k>R$ with $2^{-k}<\frac{\varepsilon}3$ to see that 
\[\Var_Rf(y)\leq \Var_Rf_k(y)+\frac{2\varepsilon}{3}<\varepsilon\]
for all $y$ outside of the bounded subset $B_R(\tilde L_k)$.

In order to show the other property, we first take a look at the functions $H_m^*f_k$ for $m\geq n_k$. For all $y\in X$ the points $H_{n_k}(H_m(y))$ and $H_{n_k}(y)$ are at most $R_{coh}$ appart, so if one of them is contained in a set $A^-_{k,\lambda}$, then the other is contained in $A^-_{k,\lambda-2^{-k}}$by property (\ref{enum:nestA}). This implies
\begin{equation}
\label{eq1243retwq23r}
|f_k(H_m(y))-f_k(y)|\leq 2^{-k}
\end{equation}
and so $\|H_m^*f_k-f_k\|\leq 2^{-k}$.

To transfer this property to $f$, choose $k\in\N$ with $2^{-k}<\frac{\varepsilon}3$.
By Lemma~\ref{lem453ewr43} we can find $N\geq n_k$ such that $H_m(X\setminus \tilde L_k)\subset X\setminus \tilde L_k$ and $H_m|_{\tilde L_k}=\id_{\tilde L_k}$ for all $m\geq N$.
Then, for all $y\in \tilde L_k$ and $m\geq N$ we have $f(H_m(y))=f(y)$. If, on the other hand, $y\in X\setminus \tilde L_k$, then also $H_m(y)\in X\setminus \tilde L_k$ and we can conclude
\begin{align*}
|f(H_m(y & ))-f(y)|\\
& \leq |f(H_m(y))-f_k(H_m(y))|+|f_k(H_m(y))-f_k(y)|+|f_k(y)-f(y)|\\
& \leq 3\cdot 2^{-k}<\varepsilon
\end{align*}
by \eqref{eq3243etw24}, \eqref{eq1243retwq23r} and again \eqref{eq3243etw24}. Therefore we conclude $\|H_m^*f-f\|<\varepsilon$ for all $m\geq N$, hence $\|H_m^*f-f\|\xrightarrow{m\to\infty}0$ and so $f\in D_H(X)$.

We denote the extension of $f$ to $\combcompb{X}{H}$, which is continuous at all points of the corona, by the same letter. If $f$ does not vanish on $y\in\combcompb{X}{H}$, then $f(\overline{H}_m(y))\not=0$ for some $m\geq n$, because $\overline{H}_m(y)\xrightarrow{m\to\infty}y$. This implies that $\overline{H}_m(y)\in C^-_{k,\lambda}$ for some $k \in \IN$ and $\lambda \in I_k$ with $\lambda \not= 0$ by our construction of~$f$. By the properties of the sets $C^\pm_{-,-}$ therefore $\overline{H}_m(y)\in C^+_{0,0}$ and hence
\begin{align*}
& \ H_n(\overline{H}_m(y))\in A^+_{0,0}=B_{R_{exp}+2R_{coh}+\overline{R}_{coh}}(\overline{H}_n(x))\\
\implies & \ \overline{H}_n(y)\in B_{R_{exp}+2R_{coh}+2\overline{R}_{coh}}(\overline{H}_n(x))\subset B_{R_{pc}}(\overline{H}_n(x))\\
\implies & \ y\in U\, .
\end{align*}
By construction $f(x)=\lim_{m\to\infty}f(\overline{H}_m(x))=1$, hence $x\in f^{-1}(\C\setminus 0)\subset U$ and $U$ is indeed a neighborhood of $x$.
\end{proof}

\begin{thm}\label{thm:RipsCompactificationContractible}
Let $(X,d)$ be a proper discrete metric space equipped with an expanding and coherent combing $H$.

Then for every $R>0$ exists $S>R$ such that $\combcompb{P_R(X)}{H}$ is contractible in $\combcompb{P_S(X)}{H}$. Even more, the $\sigma$-compact space $\combcompb{\cP(X)}{H}$ itself is contractible.
Further, contractions can be chosen which fix the starting point $p\in X$ of the combing.
\end{thm}

\begin{proof}
Let $\overline{H}$ be a standard extension of $H$ and choose the constants $R_{exp}$, $R_{coh}$, $\overline{R}_{coh}$, $R_1$, $R_{pc}$  as in the previous lemmas of this section. Exploiting the pseudocontinuity of $\overline{H}$ and discreteness of $X$ we obtain an open cover $\mathcal{U}_n=\{U_{y,n}\}_{y\in\combcompb{X}{H}}$ of $\combcompb{X}{H}$ for each $n\in\N$ such that for all $y\in \combcompb{X}{H}$
\begin{enumerate}[(i)]
\item $U_{y,n}$ is an open neighborhood of $(y,n)$, \label{enum:neighb}
\item $U_{y,n}=\{y\}$ if $y\in X$, \label{enum:singleton}
\item $\overline{H}_n(U_{y,n})\subset B_{R_{pc}}(\overline{H}_n(y))$, \label{enum:pseudocont} \item the starting point $p$ of the combing is contained in no $U_{y,n}$ with $y\not=p$.\label{enum:pfixpoint}
\end{enumerate}
For each $n\in\N$ let $\{\varphi_{y,n}\}_{y\in\combcompb{X}{H}}$ be a locally finite partition of unity which is subordinate to the cover $\mathcal{U}_n$.
We can extend it to a locally finite partition of unity on $\combcompb{P_R(X)}{H}$ for arbitrary $R>0$ by the formula
\[\varphi_{y,n}\left(\sum_{i=0}^k\lambda_ix_i\right)\coloneqq \sum_{i=0}^k\lambda_i \varphi_{y,n}(x_i)\,.\]

Let $R>0$ be arbitrary, let $S_R>0$ a constant such that $H$ maps points of distance at most $R$ to points of distance at most $S_R$.
The properties on $\overline{H}$ imply that if $x_0,\dots, x_k\in X$ span a $k$-simplex of $P_R(X)$,
i.\,e.~have mutual distance at most $R$, then for each $n\in\N$ all of the points of the form $\overline{H}_n(y)$ or $\overline{H}_{n+1}(y)$ with $x_i\in U_{y,n}$ for some $i\in\{0,\dots,k\}$ have mutual distance at most $S\coloneqq 2R_{pc}+S_R+R_1$. 

Furthermore, if $x\in\partial_HX$, then for each $n\in\N$ all of the points of the form $\overline{H}_n(y)$ or $\overline{H}_{n+1}(y)$ with $x\in U_{y,n}$ have mutual distance at most $2R_{pc}+R_1\leq S$.

Hence we can define a continuous map
$H^R\colon \combcompb{P_R(X)}{H} \times[0,\infty)\to P_S(X)$ by the formula
\begin{align*}
H^R(x,t) \coloneqq  (n+1-t) & \cdot \sum_{y\in\combcompb{X}{H}}\varphi_{y,n}(x)\overline{H}_n(y)\\
+(t-n) & \cdot \sum_{y\in\combcompb{X}{H}}\varphi_{y,n+1}(x)\overline{H}_{n+1}(y)
\end{align*}
whenever $t\in[n,n+1]$. We note the following three properties of $H^R$:
\begin{itemize}
\item $H^R_0$ is constantly equal to the starting point $p$ of the combing $H$, because $\overline{H}_0$ has this property. Furthermore, Property (\ref{enum:pfixpoint}) implies that $H^R(p,t)=p$ for all $t\in[0,\infty)$. \item Properties (\ref{enum:singleton}) and (\ref{enum:pseudocont}) imply that for each $x\in X$ there is an $N\in\N$ such that for every $n\geq N$ the point $x$ is contained in no  other set $U_{y,n}$ but $U_{x,n}=\{x\}$. Consequently, for every bounded subset $K\subset P_R(X)$ there is a $T>0$ such that $H^R(x,t)=x$ for all $x\in K$ and $t\geq T$.
\item Denote the inclusions of $X$ into $P_R(X)$ and $P_S(X)$ by $i_R$ and $i_S$, respectively, and the inclusion of $\N$ into $[0,\infty)$ by $i_\N$. Then we see immediately from the construction of $H^R$ that $H^R|_{P_R(X)\times[0,\infty)}$ is a controlled map and $H^R\circ(i_R\times i_\N)$ is close to $i_S\circ H$.
\end{itemize}

We have to extend $H^R$ continuously to a map $\combcompb{P_R(X)}{H}\times[0,\infty]\to \combcompb{P_S(X)}{H}$ with  $H^R|_{\combcompb{P_R(X)}{H}\times\{\infty\}}$ being the inclusion $\combcompb{P_R(X)}{H} \subset \combcompb{P_S(X)}{H}$. Using the above three properties we do it as described as follows.

We can extend $H^R|_{P_R(X)\times\N}$ to a proper combing $\tilde H$ of $P_S(X)$. For this combing, $\tilde H\circ(i_S\times \id_\N)$ is then close to $i_S\times H$, i.\,e.~$\tilde H$ is a proper combing which induces the compactification $\combcompb{P_S(X)}{H}$, as discussed in Example~\ref{ex_sigmacombingcompactification}. 
In particular, for every $f\in C(\combcompb{P_S(X)}{H})$ we have that $\tilde H_n^*(f|_{P_S(X)})\xrightarrow{n\to\infty} f|_{P_S(X)}$ and we immediately obtain $(H_t^R)^*f\xrightarrow{t\to\infty}f|_{\combcompb{P_R(X)}{H}}\in C(\combcompb{P_R(X)}{H})$.
This tells us that $H^R$ can be extended continuously to $\combcompb{P_R(X)}{H} \times[0,\infty]\to \combcompb{P_S(X)}{H}$ as needed, finishing the proof of the first part.

A priori, the contractibility of $\combcompb{\cP(X)}{H}$ as a $\sigma$-compact space is a stronger property: The task here is to construct \emph{one single} continuous contraction $H^\infty\colon \combcompb{\cP(X)}{H}\times [0,\infty]\to \combcompb{\cP(X)}{H}$ which respects the filtration in the sense that for all $R>0$ there is $S\geq R$ such that $H^\infty$ maps the subspace $\combcompb{P_R(X)}{H}\times[0,\infty]$ to $\combcompb{P_S(X)}{H}$.

But this contraction of $\combcompb{\cP(X)}{H}$ is obtained by simply combining all the $H^R$ that we constructed above: Both the open cover $\mathcal{U}_n=\{U_{y,n}\}_{y\in\combcompb{X}{H}}$ and the to it subordinate partition of unity $\{\varphi_{y,n}\}_{y\in\combcompb{X}{H}}$ can be chosen independently of $R$. If we perform the above construction for this fixed partition of unity, then the contractions will have the property that $H^{R'}$ restricts to $H^R$ whenever $R'\geq R$. Thus, all the $H^R$ combine to one single contraction $H^\infty$ of the $\sigma$-compact space $\combcompb{\cP(X)}{H}$ with the desired property.
\end{proof}

\subsection{Injectivity of the coarse assembly map and surjectivity of the coarse co-assembly map}
\label{sec_coarse_assembly_coassembly}

From Lemma~\ref{lem_transgressionisoforcontractible} and Theorem~\ref{thm:RipsCompactificationContractible} we get the following:
\begin{cor}\label{corjhgfdnbv}
Let $X$ be a proper metric space equipped with an expanding and coherent combing $H$.

Then the transgression maps
\[\KX_{\ast}(X) \to \tilde{K}_{\ast-1}(\partial_H X)\quad\text{and}\quad \tilde{K}^{*-1}(\partial_HX;D)\to\KX^*(X;D)\]
are isomorphisms, where $D$ is any coefficient $C^*$-algebra.
\end{cor}
Recall from the constructions of transgression maps in Section \ref{sec:Steenrod} that in the above corollary $\tilde{K}$ denotes versions of reduced $K$-homology and $K$-theory which satisfy the strong excision axiom, and that $\KX$ are the corresponding coarse versions.

\begin{thm}\label{thm243reds}
Let $X$ be an expandingly and coherently combable proper metric space. Then the analytic coarse  assembly map
\[\mu\colon \KX_*(X)\to K_*(C^*X)\]
is injective.
\end{thm}

\begin{proof}
From \cite[Remark 12.3.8]{higson_roe} we know that there is a homomorphism $t\colon K_*(C^*X)\to \widetilde{K}_{*-1}(\partial_HX)$ such that the composition
\[t\circ\mu\colon \KX_*(X)\to \tilde{K}_{*-1}(\partial_HX)\]
is the transgression map. Since the transgression map is an isomorphism by the above Corollary~\ref{corjhgfdnbv}, we conclude that $\mu$ must be injective.
\end{proof}

In \cite{EmeMey}, a coarse co-assembly map  
\[\mu^*\colon \tilde{K}_{1-*}(\mathfrak{c}(X;D))\to \KX^{*}(X;D)\]
with coefficients in a $C^*$-algebra $D$ was introduced. In the case $X$ has bounded geometry and $D=\C$, the co-assembly map is dual to the coarse assembly map in the sense that there are pairings compatible with the assembly and co-assembly maps. Similar to the above theorem we can prove:

\begin{thm}\label{thm_coasssemblysurjective}
Let $X$ be an expandingly and coherently combable proper metric space. Then the coarse co-assembly map
\[\mu^*\colon \tilde K_{1-*}(\mathfrak{c}(X;D))\to\KX^*(X;D)\]
with an arbitrary coefficient $C^*$-algebra $D$
is surjective.
\end{thm}

\begin{proof}
Consider the commutative diagram
\[\xymatrix{
0\ar[r]&C_0(\cP(Z))\otimes D\ar[d]\ar[r]& C(\combcompb{\cP(Z)}{H})\otimes D\ar[d]\ar[r]& C(\partial_HX)\otimes D\ar[d]\ar[r]&0
\\0\ar[r]& C_0(\cP(Z))\otimes D\otimes\mathfrak{K}\ar[r]&\overline{\mathfrak{c}}(\cP(Z);D)\ar[r]& \mathfrak{c}(\cP(Z);D)\ar[r]&0
}\]
with exact rows, whose left two vertical arrows are induced by the inclusion $\C\to\mathfrak{K}$ as some rank $1$-projection and the right vertical arrow is the unique $*$-homomorphism which makes this diagram commute.
By naturality of the connecting homomorphism in $K$-theory and passing to reduced $K$-theory we obtain the commutative diagram
\[\mathclap{\xymatrix{
\tilde{K}^{*-1}(\partial_HX; D)\ar[r]^-{\cong}\ar[d]
&\KX^*(X;D)\ar@{=}[d]
\\\tilde{K}_{1-*}(\mathfrak{c}(\cP(Z);D))\ar[r]
&\KX^*(X;D)
}}\]
where the upper horizontal arrow is an isomorphism by the corollary.
The claim now follows by recalling that the co-assembly map is by definition the composition of the lower horizontal map with a canonical isomorphism $\tilde{K}_{1-*}(\mathfrak{c}(X;D))\cong \tilde{K}_{1-*}(\mathfrak{c}(\cP(Z);D))$.
\end{proof}

\subsection{Implications for groups}

Let $G$ be a finitely generated group. Choosing a finite, symmetric generating set, we can equip $G$ with the resulting word metric to turn it into a discrete, proper metric space which we will also denote by $G$.

In Lemma~\ref{lem:charactermapisomorphism} we have shown that $\HX^\ast(G;M) \cong \HA_c^\ast(\cP(G);M)$ for any abelian group $M$. Now Roe \cite[Example 5.21]{roe_lectures_coarse_geometry} has shown the isomorphism
\begin{equation}
\label{eq_HX_Gruppenring}
\HX^\ast(G;R) \cong H^\ast(G;R[G])\, ,
\end{equation}
where $R[G]$ is the group ring for a ring~$R$. Combining these two isomorphisms we get
\begin{equation}\label{eq243rwewqe}
H^\ast(G;R[G]) \cong \HA_c^\ast(\cP(G);R)\, .
\end{equation}
In the case that $G$ is such that a Rips complex $P_n(G)$ for a fixed scale $n > 0$ is contractible, the isomorphism $H^\ast(G;R[G]) \cong H_c^\ast(P_n(G);R)$ is well-known; see, e.g., Bestvina--Mess \cite[Page 470]{bestvina_mess}. Note that we have an isomorphism $H_c^\ast(P_n(G);R) = \HA_c^\ast(P_n(G);R)$ and hence Formula \eqref{eq243rwewqe} is the corresponding generalization to all groups.

Let $G$ be equipped with a coherent and expanding combing $H$. Then by Lemma~\ref{lem_transgressionisoforcontractible} and Theorem~\ref{thm:RipsCompactificationContractible} we know that the transgression map for coarse cohomology is an isomorphism. Combining this with the above discussion we get the following:

\begin{cor}\label{cor23452342}
Let $(G,H)$ be a finitely generated group equipped with a coherent and expanding combing.

Then we have for any ring $R$ the isomorphism
\[H^{\ast+1}(G;R[G]) \cong \widetilde{\HA}^\ast(\partial_H G;R).\qedhere\]
\end{cor}
On the right hand side we could have equally well used \Cech cohomology as usually used in the literature for the boundary, because it coincides with Alexander--Spanier cohomology for compact Hausdorff spaces like $\partial_HG$.

For hyperbolic groups the above formula was obtained by Bestvina--Mess \cite[Corollary 1.3(b)]{bestvina_mess}, and more generally for groups admitting a so-called $Z$-structure (cf.~Section \ref{secnm0987678bmvmvmv}) by Bestvina \cite{bestvina}. 

Recall that the cohomological dimension $\cd_R(G)$ of $G$ over $R$ is defined as
\[\cd_R(G) \coloneqq  \sup\{n \mid H^n(G;M) \not= 0 \text{ for some } R[G]\text{-module }M\}.\]
Recall that a group $G$ is called of type $F\! P_\infty$ if there exists a (possibly infinite) projective resolution of $\IZ$ over $\IZ G$ by finitely generated modules. If $\cd_R(G)$ is finite and the group $G$ is of type $F\! P_\infty$, then we can show that we have $\cd_R(G) \coloneqq  \max\{n\mid H^n(G;R[G]) \not= 0\}$ (combine Propositions VIII.2.3 and VIII.5.2 in Brown \cite{brown}). Note that this fails in general if $\cd_R(G)$ is infinite (as can be quickly seen by taking $R \coloneqq  \IZ$ and $G$ a finite group). 
Now if $G$ is combable, it is of type $F\! P_\infty$ (Alonso \cite[Theorem 2]{alonso}), and if the combing is moreover expanding and coherent, we can use the above Corollary~\ref{cor23452342} to show the following.

\begin{cor}\label{cor345243}
Let $G$ be a finitely generated group equipped with a coherent and expanding combing $H$, and let $R$ be a ring.

Then we have
\[\cd_R(G) < \infty  \implies  \cd_R(G) - 1 = \max\{n\mid \widetilde{\HA}^n(\partial_H G;R) \not= 0\}.\]
\end{cor}

The asymptotic dimension $\asdim(X)$ of a metric space $X$ will be defined in Definition~\ref{def_dimensions}.\ref{defn_asdim}.
Because $\HX^k(G;R)$ vanishes for $k > \asdim(G) + 1$, we get, by using \eqref{eq_HX_Gruppenring}, the following result:
\begin{cor}\label{cor_cdRG_asdim}
Let $G$ be a finitely generated group admitting a coherent and expanding combing, and let $R$ be a ring.

If $\cd_R(G)$ is finite, then $\cd_R(G) \le \asdim(G) + 1$.
\end{cor}
Similar results were obtained by Dranishnikov \cite[Prop.\ 5.10 \& Cor.\ 5.11]{dranish_approach} for groups admitting a finite classifying space.\footnote{He actually gets the stronger estimate $\cd_R(G) \le \asdim(G)$.}

Let us now quickly discuss an important implication if we additionally assume a certain equivariance of the combing.

\begin{defn}
Let $H$ be a combing on the group $G$. We say that $H$ is \emph{coarsely equivariant} if for all $g \in G$ the maps $\lambda_g \circ H$ and $H \circ (\lambda_g \times \id_\IN)$ are close, where $\lambda_g$ denotes left multiplication on $G$ by $g$.
\end{defn}

If $H$ is coarsely equivariant, then by Corollary~\ref{cor_coronafunctoriality} the action of $\lambda_g$ on $G$ extends to a continuous action on $\combcompb{\cP(G)}{H}$ by $\sigma$-homeomorphisms.

\begin{thm}\label{thm_SNCF}
Let $G$ be a group which admits a finite classifying space $BG$ and an expanding, coherent and coarsely equivariant combing $H$.

Then the integral strong Novikov conjecture holds for $G$, i.e., the analytic assembly map $RK_*(BG) \to K_*(C_r^* G)$ is injective.
\end{thm}

\begin{proof}
Using Corollary~\ref{cor243terwe} and Remark~\ref{rem24356545435459} we get a contractible, metrizable compactification $\combcompb{EG}{H}$ of $EG$ such that $\partial_H G$ is a $Z$-set in it and such that the $G$-action on $EG$ is cocompact and extends continuously to $\combcompb{EG}{H}$. The claim now follows from \cite[Theorem 10.8]{roe_index_coarse}.
\end{proof}

\begin{rem}\label{rem_split_injectivity}
Under the same assumptions as in the above theorem we get also split injectivity results for the assembly maps for the non-connective versions of algebraic $K$-theory, $L$-theory and $A$-theory by using the results of Carlsson--Pedersen \cite{carlsson_pedersen,carlsson_pedersen_2} and Carlsson--Pedersen--Vogell \cite{carlsson_pedersen_vogell}.

That under these assumptions the group satisfies the classical Novikov conjecture was also proven by Farrell--Lafont \cite{farrell_lafont}.
\end{rem}

\begin{example}
Groups satisfying the assumptions of the above Theorem~\ref{thm_SNCF} are torsion-free hyperbolic, $\CAT(0)$ cubical and systolic groups.
\end{example}

Using results of Emerson--Meyer we can actually strengthen Theorem~\ref{thm_SNCF}:
\begin{thm}\label{thm_dual_Dirac}
Let $G$ be a group which admits a finite classifying space $BG$ and an expanding, coherent and coarsely equivariant combing $H$.

Then $G$ admits a $\gamma$-element. Consequently, the analytic assembly map for $G$ is split-injective, and the co-assembly and coarse co-assembly maps for $G$ are isomorphisms.
\end{thm}

\begin{proof}
By \cite[Thm.~14]{EM_gamma} the group $G$ admits a $\gamma$-element. It is well-known that this implies that the analytic assembly map for $G$ is split-injective.

By \cite[Thm.~13(d)]{EM_gamma} this also implies that the co-assembly map is an isomorphism, and moreover, by \cite[Cor.~34]{EM_descent} it also implies that the coarse co-assembly map is an isomorphism.
\end{proof}

\section{Cohomological dimension of the corona}
\label{sec45342345}

In this section we prove upper bounds for the cohomological dimension of the combing corona of coherently and expandingly combable spaces of finite asymptotic dimension. These notions of dimension are defined as follows.
\begin{defn}\label{def_dimensions}
\mbox{}
\begin{enumerate}
\item The \emph{cohomological dimension} of a compact Hausdorff space $Z$ with respect to an abelian group $M$ is
\[\cd_M(Z) \coloneqq  \sup\{n\mid \HA^n(Z,A;M) \not= 0 \text{ for some closed } A \subset Z\}.\]
\item\label{defn_asdim} The \emph{asymptotic dimension} of a metric space $X$ is the infimum of all those $n\in\N$ for which there exists an anti-\Cech system $(U_j)_{j \in \IN}$ for $X$ with $\dim(U_j) \le n$ for all $j$. (cf.~\cite[Section 1.E]{Gromov_asymptotic_invariants})
\end{enumerate}
Both of these dimensions are either natural numbers or $\infty$.
\end{defn}

Our upper bounds for the cohomological dimension will culminate in the following theorem, which unfortunately we were only able to prove for groups and not for more general spaces.
\begin{thm}\label{thm_cdleqasdim}
Let $G$ be a finitely generated group equipped with a coherent, expanding combing $H$ and $R$ be  either a unital ring whose underlying additive groups structure is finitely generated or a field.
Then 
\[\cd_R(\partial_H G) \leq \asdim(G)\,.\]
\end{thm}

A related result of Dranishnikov--Keesling--Uspenskij \cite{dku} is the inequality $\dim(\nu X) \le \asdim(X)$ between the asymptotic dimension of a proper metric space $X$ and the covering dimension of its Higson corona $\nu X$.

Our upper bounds for the cohomological dimension are inspired by the papers \cite{bestvina_mess,dranish_BM}, and is based on the fact that the combing compactification of the Rips complex has properties very similar to those of $Z$-structures (see the beginning of Section~\ref{secnm0987678bmvmvmv} for a thorough discussion of $Z$-structures).

Our main computational tool for this task is simplicial cohomology.  For a locally finite simplicial complex $K$ with closed subcomplex $L\subset K$ we denote the simplicial cochain complex of $(K,L)$ with finite supports and coefficients $M$ by $C\!\Delta^*_c(K,L;M)$ and its cohomology by $H\!\Delta^*_c(K,L;M)$. Furthermore, these groups only depend on the difference $K\setminus L$, so we define the simplicial cochain complex and cohomology of the open subcomplex $O\coloneqq K\setminus L$ by $C\!\Delta^*_c(O;M)\coloneqq C\!\Delta^*_c(K,L;M)$ and $H\!\Delta^*_c(O;M)\coloneqq H\!\Delta^*_c(K,L;M)$, respectively. It is well known that these cohomology groups are canonically isomorphic to the Alexander--Spanier cohomology groups with compact supports.

We can also introduce the notion of a $\sigma$-simplical complex and its simplicial cohomology: by a $\sigma$-simplical complex we mean a simplicial complex $\cK$ which is the union of an increasing sequence of locally finite, closed subcomplexes $K_0\subset K_1\subset K_2\subset\dots$ and we define its simplicial cohomology $H\!\Delta_c^*(\cK;M)$ to be the cohomology of the cochain complex 
\[C\!\Delta_c^*(\cK;M)\coloneqq \varprojlim_{n\in\N}C\!\Delta_c^*(K_n;M)\,.\]

The only  $\sigma$-simplical complex which will be relevant to us is the full Rips complex $\cP(X)$ of a proper discrete metric space $X$, which is the union of the increasing sequence of closed subcomplexes $P_n(X)$.
Note that in this case one can identify $(k+1)$-tuples of points in $X$ with $k$-simplices of $\cP(X)$ and this identification yields a canonical isomorphism between the coarse cochain complex $\CX^*(X;M)$ and the simplicial cochain complex $C\!\Delta_c^*(\cP(X);M)$.
Note that in this light, Lemma \ref{lem:charactermapisomorphism} can be seen as an isomorphism between Alexander--Spanier and simplicial cohomology of the $\sigma$-simplicial complex $\cP(X)$.
Because of this isomorphism we immediately obtain a $\varprojlim^1$-sequence
\[0\to {\varprojlim_{n\in\N}}^1H\!\Delta_c^{k-1}(P_n(X);M)\to \HX^k(X;M)\to \varprojlim_{n\in\N}H\!\Delta_c^k(P_n(X);M)\to 0\,.\]

After having related the coarse cohomology to the simplicial cohomology of the finite-scale Rips complexes, we want to go a step further and relate the latter to the simplicial cohomology of finite open subcomplexes. To this end, we make the following definition.

\begin{defn}
Let $X$ be a discrete, proper metric space and let $A\subset X$ be a subset. 
The \emph{simplicial neighborhood} of $A$ is defined to be the open subset
\[\cU(A)\coloneqq \cP(X)\setminus \cP(X\setminus A)\subset \cP(X)\,.\]
It is the union of the interiors of all those simplices which have at least one vertex in $A$.
Further, for $r\geq 0$ we define the \emph{simplicial $r$-neighborhood} as
\[\cU(A,r)\coloneqq \cU(\{x\in X\mid \dist(x,A)\leq r\})\]
and we denote the intersections of $\cU(A)$ and $\cU(A,r)$ with $P_n(X)$ by $U_n(A)$ and $U_n(A,r)$, respectively.
\end{defn}

Because of the canonical isomorphism between Alexander--Spanier and simplicial cohomology, continuity of Alexander--Spanier cohomology applied to the increasing sequence of subcomplexes $U_n(A,r)$ passes over to simplicial cohomology, yielding
\[\varinjlim_{r\in\N}H\!\Delta_c^k(U_n(A,r);M)\cong H\!\Delta_c^k(P_n(X);M)\]
whenever $A$ is non-empty.

The above definition also allows us to adapt the notion of uniform triviality of \cite[Page 6]{Geoghegan_Ontaneda} (see also \cite[Page 5]{dranish_BM}) to our setup, where we have to take the $\sigma$-structure of the Rips complex into account. Of course, the following definition can also be formulated using Alexander--Spanier cohomology.

\begin{defn}
Let $X$ be a discrete, proper metric space and $k\in\N$. We say that $\HX^k(X;M)$ is \emph{uniformly trivial} if for every $n\in\N$ there exists $N\geq n$ and a function $s\colon\R_{\geq 0}\to\R_{\geq 0}$ satisfying $\forall r\geq 0\colon s(r)\geq r$ and such that the canonical homomorphism
\begin{equation}\label{eq:def:unifloctriv}
H\!\Delta_c^k(\cU_N(\{x\},r);M)\to H\!\Delta_c^k(\cU_n(\{x\},s(r));M)
\end{equation}
vanishes for all $x\in X$ and all $r\geq0$.
\end{defn}

Our first upper bound for the cohomological dimension works not only for groups and with arbitrary unital rings as coefficients.
\begin{thm}\label{thm345zrtewqd}
Let $X$ be a discrete, proper metric space of finite asymptotic dimension equipped with a coherent and expanding combing $H$, and let $R$ be a unital ring. Then 
\[\operatorname{cd}_R(\partial_HX)+2\leq \min\{k\mid \HX^i(X;R)\text{ is uniformly trivial for all }i\geq k\}\,.\]
\end{thm}

The proof needs two lemmata.
\begin{lem}\label{lem23430987rew}
Let $X$ be a discrete, proper metric space equipped with a coherent and expanding combing $H$.

Then for every closed $A \subset \partial_H X$ there is a homotopy (of $\sigma$-compact spaces)
\[h \colon \combcompb{\cP(X)}{H} \times [0,1] \to \combcompb{\cP(X)}{H}\]
such that $h_0 \equiv \mathrm{identity}$, and for all times $t > 0$ we have $h_t|_A \equiv \mathrm{inclusion}$ and $h_t(\combcompb{\cP(X)}{H} \setminus A) \subset \cP(X)$.
\end{lem}

\begin{proof}
As $\partial_HX$ is metrizable by Lemma \ref{lem:metrizability}, we can choose any metric to define a continuous function 
\[\partial_HX\to[0,1]\,,\quad x\mapsto\min\{\dist(x,A),1\}\,.\]
Using Tietze's extension theorem, we can successively extend it to continuous functions $\combcompb{P_n(X)}{H}\to [0,1]$ for all $n\in\N$, and combining all of them we obtain a continuous function 
$\phi\colon \combcompb{\cP(X)}{H}\to [0,1]$ with the properties that $\phi(A)\subset \{0\}$ and $\phi(\partial_HX\setminus A)\subset (0,1]$.

Let $H^\infty\colon \combcompb{\cP(X)}{H}\times[0,1]\to \combcompb{\cP(X)}{H}$ be the contraction of the $\sigma$-compact space $\combcompb{\cP(X)}{H}$ which we constructed in the proof of Theorem \ref{thm:RipsCompactificationContractible}, but with contraction parameter reparametrized to $[0,1]$, i.e., such that we have $H^\infty_0\equiv p$ and $H^\infty_1=\id$. This contraction has the property that it maps the subspace $\combcompb{\cP(X)}{H} \times[0,1)\cup \cP(X)\times[0,1]$ into $\cP(X)$.

It is now easy to verify that the homotopy
\[h\colon (x,t)\mapsto H^\infty(x,1-t\cdot \phi(x))\]
has the desired properties.
\end{proof}

\begin{lem}\label{lem_unifloctrivimplieszero}
Let $X$ be a discrete, proper metric space of finite asymptotic dimension $d$ and $R$ be a unital ring. 
Suppose that there is $k\in\N$ such that $\HX^i(X;R)$ is uniformly trivial for all $i\geq k$. Then for every $n\in\N$ there exists an $r\geq 0$ such that for all subsets $A\subset X$ the canonical morphism $U_n(A,r)\supset U_n(A)\to \cU(A)$ induces the zero homomorphism
\[\HA_c^i(\cU(A);R)\to \HA_c^i(U_n(A,r);R)\]
for all $i\geq k$.
\end{lem}

\begin{proof}
Given any $k\in\N$ with the property that $\HX^i(X;R)$ is uniformly trivial for all $i\geq k$, and given any $n\in\N$, we will show that there are $N_{k,n}\geq n$, $r_{k,n}\geq 0$ and a cochain homotopy
\[h^{*}_{k,n}\colon C\!\Delta_c^{*}(P_{N_{k,n}}(X);R)\to C\!\Delta_c^{*-1}(P_n(X);R)\]
between the restriction map $\pi_{k,n}^*$ and some cochain map $f_{k,n}^*$ which vanishes in degrees $i\geq k$ such that $h^{*}_{k,n}$ restricts to 
\[C\!\Delta_c^{*}(U_{N_{k,n}}(A);R)\to C\!\Delta_c^{*-1}(U_n(A,r_{k,n});R)\]
for all $A\subset X$. Because of the isomorphism between simplicial and Alexander--Spanier cohomology of locally finite simplicial complexes, this set of data would obviously immediately imply that the restriction map
\[\HA_c^i(\cU(A);R)\to \HA_c^i(U_{N_{k,n}}(A);R)\to \HA_c^i(U_n(A,r_{k,n});R)\]
vanishes for all $i\geq k$ and all $A\subset X$.

We first prove the claim for $k\geq d+1$.
Choose an anti-Cech system $(\cU_l)_{l\in\N}$ for $X$ with $\dim(\cU_l)\leq d$ for all $l$.
Then for any $n\in\N$ there is $l\in\N$ such that $\cU_l$ has Lebesgue number bigger than $n$ and one obtains a simplicial map $P_n(X)\to \|\cU_l\|$ which maps each point $x\in X$ to the vertex corresponding to an open set $U\in\cU_l$ which contains $x$. Conversely, if the open sets in $\cU_l$ have diameter bounded by some $r_{k,n}\geq 0$, then one obtains a simplicial map $\|\cU_l\|\to P_{r_{k,n}}(X)$ which maps the vertex given by $U\in \cU_l$ to any point $x\in U\subset X$.
We choose $N_{k,n}\geq n+2r_{k,n}$. The composition of simplicial maps
\[f_{k,n}\colon P_n(X)\to \|\cU_l\|\to P_{r_{k,n}}(X)\to P_{N_{k,n}}(X)\]
now has the following properties:
\begin{itemize}
\item The induced map $f_{k,n}^*\colon C\!\Delta_c^*(P_{N_{k,n}}(X);R)\to C\!\Delta_c^*(P_n(X);R)$ is zero in degrees $i\geq d+1$, because $\|\cU_l\|$ has no $i$-simplices.
\item The map $f_{k,n}$ restricts to a map $X\to X$ which is $r_{k,n}$-close to the identity. Therefore, if $x_0,\dots,x_i$ span an $i$-simplex in $P_n(X)$, then any $i+2$ of the points $x_0,\dots,x_i,f_{k,n}(x_0),\dots,f_{k,n}(x_i)$ span an $(i+1)$-simplex in $P_{N_{k,n}}(X)$ and so the usual formula
\[(h^{i+1}_{k,n}\phi)(\langle x_0,\dots, x_i\rangle)\coloneqq \sum_{\nu=0}^i(-1)^\nu\phi(\langle x_0,\dots,x_\nu,f_{k,n}(x_\nu),\dots,f_{k,n}(x_i)\rangle)\]
gives a cochain homotopy $h_{k,n}^*\colon C\!\Delta_c^{*}(P_{N_{k,n}}(X);R)\to C\!\Delta_c^{*-1}(P_n(X);R)$ between $f_{k,n}^*$ and the restriction cochain map $\pi_{k,n}^*$. 
\item For a subset $A\subset X$ the cochain homotopy $h^*_{k,n}$ maps the subcomplex $C\!\Delta_c^{*}(U_{N_{k,n}}(A);R)$ to the subcomplex $C\!\Delta_c^{*-1}(U_{n}(A;r_{k,n});R)$, because if $\phi\in C\!\Delta_c^{i+1}(U_{N_{k,n}}(A);R)$ and $h^{i+1}_{k,n}\phi$ is non-zero on a simplex $\langle x_0,\dots, x_i\rangle$ of $P_n(X)$, then one of the points $x_0,\dots,x_i,f_{k,n}(x_0),\dots,f_{k,n}(x_i)$ is contained in $A$, and thus one of the points $x_0,\dots,x_i$ has distance at most $r_{k,n}$ of $A$, implying $\langle x_0,\dots, x_i\rangle\subset U_{N_{k,n}}(A,r_{k,n})$.
\end{itemize}
This finishes the proof of the claim in the case $k\geq d+1$.

For the other cases $k\leq d$ we now perform a downward induction on~$k$. We assume that $h_{k+1,n}$, $N_{k+1,n}$ and $r_{k+1,n}$ with the properties listed above have already been constructed for all $n\in\N$ and we assume that $\HX^k(X;R)$ is uniformly trivial.

For fixed $n\in\N$ let $N\geq n$ and $s\colon\R_{\geq 0}\to \R_{\geq 0}$ be the data given to us by the definition of uniform local triviality of $\HX^k(X;R)$. Using this we define $N_{k,n}\coloneqq N_{k+1,N}$ and $r_{k,n}\coloneqq s(r_{k+1,N})+N_{k,n}$. 
We now have
\begin{align*}
\pi^{k+1}_{k+1,N}&=\pi^{k+1}_{k+1,N}-f^{k+1}_{k+1,N}=\delta h^{k+1}_{k+1,N}+h^{k+2}_{k+1,N}\delta
\\\Rightarrow \delta\pi^k_{k+1,N}&=\pi^{k+1}_{k+1,N}\delta=\delta h^{k+1}_{k+1,N}\delta
\end{align*}
which tells us that the map 
\[\pi^k_{k+1,N}-h^{k+1}_{k+1,N}\delta\colon C\!\Delta_c^k(P_{N_{k,n}}(X);R)\to C\!\Delta_c^k(P_N(X);R)\]
maps any cochain to a cocycle. 
Furthermore, each subgroup $C\!\Delta_c^k(U_{N_{k,n}}(A);R)$ is mapped into the subgroup $C\!\Delta_c^k(U_N(A,r_{k+1,N});R)$.

For any $k$-simplex $\sigma=\langle x_0,\dots,x_k\rangle\subset P_{N_{k,n}}$ we define the cochain
\[\phi_{\sigma}\in C\!\Delta_c^k(U_{N_{k,n}}(\{x_0\});R)\subset C\!\Delta_c^k(P_{N_{k,n}}(X);R)\,,\qquad \phi_\sigma(\sigma')=\begin{cases}1&\sigma=\sigma'\\0&\text{else.}\end{cases}\]
Then, $(\pi^k_{k+1,N}-h^{k+1}_{k+1,N}\delta)(\phi_\sigma)$ defines a cocycle in $C\!\Delta_c^k(U_N(\{x_0\},r_{k+1,N});R)$ and the uniform local triviality of $\HX^k(X)$ implies that there is 
\[\psi_\sigma\in C\!\Delta_c^{k-1}(U_n(\{x_0\},s(r_{k+1,n}));R)\subset C\!\Delta_c^k(P_n(X);R)\]
such that 
\[\delta\psi_\sigma=\tau^k(\pi^k_{k+1,N}-h^{k+1}_{k+1,N}\delta)(\phi_\sigma)\,,\]
where $\tau^*\colon C\!\Delta_c^*(P_N(X);R)\to C\!\Delta_c^*(P_n(X);R)$ denotes the canonical restriction map.

As the $\phi_\sigma$ constitute a basis of $C\!\Delta_c^k(P_n(X);R)$, one can define the cochain homotopy $h^*_{k,n}$ by $h^k_{k,n}(\phi_\sigma)=\psi_\sigma$, $h^i_{k,n}=\tau^ih^i_{k+1,N}$ in degrees $i>k$ and $h^i_{k,n}=0$ for $i<k$.
For $i>k$ we clearly have 
\[\delta h^{i}_{k,n}+h^{i+1}_{k,n}\delta=\delta \tau^{i-1}h^i_{k+1,N}+\tau^{i}h^{i+1}_{k+1,N}\delta=\tau^i\pi^i_{k+1,N}=\pi^i_{k,n}\]
and in degree $k$ one calculates
\[(\delta h^{k}_{k,n}+h^{k+1}_{k,n}\delta)(\phi_\sigma)=\delta\psi_\sigma+\tau^kh^{k+1}_{k+1,N}(\phi_\sigma)=\tau^k\pi^k_{k+1,N}(\phi_\sigma)=\pi^k_{k,n}(\phi_\sigma)\,.\]
Therefore, $h^*_{k,n}$ is a cochain homotopy between $\pi^*_{k,n}$ and a cochain map $f^*_{k,n}$ which vanishes in degrees $i\geq k$.

It remains to check that $h^*_{k,n}$ maps the subcomplex $C\!\Delta_c^*(U_{N_{k,n}}(A);R)$ into $C\!\Delta_c^{*-1}(U_n(A,r_{k,n});R)$ for each subset $A\subset X$. In  degrees $i<k$ this is trivial and in degrees $i>k$ this follows from the fact that $r_{k,n}\geq s(r_{k+1,n})\geq r_{k+1,n}$ and the corresponding property of the maps $h^*_{k+1,N}$. 
In degree $k$, if we have $\phi_\sigma\in C\!\Delta_c^k(U_{N_{k,n}}(A);R)$, i.\,e.~one of the vertices $x_\nu$ of $\sigma$ is contained in $A$, then $\dist(x_0,A)\leq N_{k,n}$ and therefore $\psi_\sigma\in C\!\Delta_c^k(U_n(A,r_{k,n});R)$. 

This finishes the proof of the induction step and therefore the proof of the lemma.
\end{proof}

\begin{proof}[Proof of Theorem~\ref{thm345zrtewqd}]
We denote $l\coloneqq \operatorname{cd}_R(\partial_HX)$ and we let $A\subset \partial_HX$ be a closed subset such that $\HA^l(\partial_HX,A;R)\not=0$. We now choose a homotopy $h\colon\combcompb{\cP(X)}{H}\times [0,1]\to\combcompb{\cP(X)}{H}$ as in Lemma \ref{lem23430987rew}.
Using that $h$ is a continuous $\sigma$-map, we see that $Y_1\coloneqq h_1(\partial_HX)$ is a closed subset of some $\combcompb{P_m(X)}{H}$. Further, we choose $n$ so large that we have $h(\combcompb{P_m(X)}{H}\times[0,1])\subset \combcompb{P_n(X)}{H}$ and define the closed subset $Y_2\coloneqq h(Y_1\times [0,1])\subset \combcompb{P_n(X)}{H}$. Then $Y_1 \subset Y_2$ and $Y_1\cap \partial_HX=Y_2\cap\partial_HX=A$. The homotopy $h$ now restricts to a homotopy of continuous maps between pairs of spaces $(\combcompb{P_m(X)}{H},Y_1)\to (\combcompb{P_n(X)}{H},Y_2)$.
At the one end of the homotopy, $h_0$ is simply the inclusion, while on the other side $h_1$ factors through the pair $(Y_2,Y_2)$. Thus, these maps induce the zero-map on homology.

Using the naturality of the long exact sequence of Alexander--Spanier cohomology with respect to the inclusion of triples
\[(\combcompb{P_m(X)}{H},Y_1\cup\partial_HX,Y_1)\to (\combcompb{P_n(X)}{H},Y_2\cup\partial_HX,Y_2)\]
and observing that the identities $Y_1\cup\partial_HX\setminus Y_1=Y_2\cup\partial_HX\setminus Y_2=\partial_HX\setminus A$ allows us to apply strong excision, we obain the commutative diagram
\[\mathclap{\xymatrix{
\HA^l(\combcompb{P_n(X)}{H},Y_2;R)\ar[r]\ar[d]^{(h_t)_*=0}
&\HA^l(\partial_HX,A;R)\ar[r]^-{\delta}\ar@{=}[d]
&\HA^{l+1}(\combcompb{P_n(X)}{H},Y_2\cup\partial_HX;R)
\\\HA^l(\combcompb{P_m(X)}{H},Y_1;R)\ar[r]
&\HA^l(\partial_HX,A;R)&
}}\]
which shows that $\delta$ is injective. Again by strong excision, we can identify $\delta$ with the connecting homomorphism
$\HA_c^l(\partial_HX\setminus A;R)\to \HA_c^{l+1}(V;R)$ of the pair $(\combcompb{P_n(X)}{H}\setminus Y_2,\partial_HX\setminus A)$, where $V\coloneqq (\combcompb{P_n(X)}{H}\setminus Y_2)\setminus (\partial_HX\setminus A)=P_n(X)\setminus Y_2$.

Now, if all $\HX^i(X;R)$ with $i\geq l+1$ were uniformly trivial, Lemma \ref{lem_unifloctrivimplieszero} could be applied with $k=l+1$ and let $r\geq 0$ be the number provided by this lemma for exacly the $n$ which we consider here. Let $B\subset X$ be the largest subset such that $U_n(B,r)\subset V$. From the fact that the contraction corona $\partial_HX$ is Higson dominated, it is straightforward to see that the intersection of the closure of each $P_j(X)\setminus U_j(B)$ in $\combcompb{P_j(X)}{H}$ with the corona $\partial_HX$ is exactly $A$, and therefore $\cU(B)\cup (\partial_HX\setminus A)$ is open in $\combcompb{\cP(X)}{H}$. Naturality of the connecting homomorphisms of Alexander--Spanier cohomology now yields the commutative diagram
\[\xymatrix{
\HA_c^l(\partial_HX\setminus A;R)\ar@{=}[d]\ar[r]&\HA_c^{l+1}(\cU(B);R)\ar[d]^0
\\\HA_c^l(\partial_HX\setminus A;R)\ar@{^(->}[r]&\HA_c^{l+1}(V;R)
}\]
where the right vertical arrow vanishes by Lemma \ref{lem_unifloctrivimplieszero} and the lower horizontal arrow is injective as seen above. This is a contradiction, because we assumed at the beginning $\HA^l(\partial_HX,A;R)\not=0$.

Hence, $\HX^i(X;R)$ cannot be uniformly trivial for all $i\geq l+1$, proving the claim.
\end{proof}

We now turn to proving sufficient conditions for uniform triviality, because we want to obtain upper bounds for the cohomological dimension in terms of less mysterious quantities. The big issue is to treat the $\varprojlim^1$-term in the $\varprojlim^1$-sequence adequately. This is a rather technical algebraic exercise.

\begin{lem}[cf.~{Gray \cite{gray}}]\label{lem:invlimabvs}
Let 
\[\ldots \xrightarrow{f_5} M_4 \xrightarrow{f_4} M_3\xrightarrow{f_3} M_2\xrightarrow{f_2} M_1\xrightarrow{f_1} M_0\]
be an inverse system of either countable abelian groups or countably dimensional vector spaces over a fixed field and assume that $\varprojlim^1M_n$ vanishes. Then the Mittag-Leffler condition is satisfied, i.\,e.~for each $n\in\N$ there is $m\in\N$ such that $\im(M_{n+k}\to M_n)=\im(M_{n+m}\to M_n)$ for all $k\geq m$.
\end{lem}

\begin{proof}We make extensive use of the explicit formulas for $\varprojlim$ and $\varprojlim^1$ in this proof, namely that  $\varprojlim M_n$ is the kernel and $\varprojlim^1M_n$ the cokernel of the map
\[F_{\{M_n,f_n\}}\colon \prod_{n\in\N}M_n\to \prod_{n\in\N}M_n\,,\quad (x_n)_{n\in\N}\mapsto (x_n-f_{n+1}(x_{n+1}))_{n\in\N}\,.\]

For each $n\in\N$ consider the inverse system
\[\dots \subset A_4\subset A_3\subset A_2\subset A_1\subset A_0\]
consisting of the subgroups, resp.\ submodules $A_k\coloneqq \im(M_{n+k}\to M_n)$.  The canonical map of inverse systems $\{M_{n+k}\}\to \{A_k\}$ is surjective, and this easily implies that the induced map $\varprojlim^1M_{n+k}\to \varprojlim^1A_k$ must also be surjective. But $\varprojlim^1M_{n+k}=\varprojlim^1M_k=0$ and so $\varprojlim^1A_k=0$, which is equivalent to saying that $F_{\{A_k\}}$ is surjective.

We claim now that the above implies that the canonical homomorphism $A_0\to\varprojlim A_0/A_k$ is surjective: Let 
\[(x_k+A_k)_{k\in\N}\in\varprojlim A_0/A_k=\ker\left(F_{\{A_0/A_k\}}\right)\,,\]
i.\,e.~$x_k\in A_0$ and $x_{k+1}-x_k\in A_k$ for all $k\in\N$. Without loss of generality we may assume $x_0=0$, since $A_0/A_0=0$. Let $(a_k)_{k\in\N}\in\prod_{k\in\N}A_k$ be a preimage of $(x_{k+1}-x_k)_{k\in\N}\in\prod_{k\in\N}A_k$ under $F_{\{A_k\}}$, i.\,e.~$a_{k+1}-a_k=x_{k+1}-x_k$. So we obtain $x_k=x_k-x_0=a_k-a_0$ and thus $x_k+A_k=-a_0+A_k$. Therefore, $-a_0\in A_0$ is the preimage we were looking for.

Now, we see from 
\[M_n\supset A_0\twoheadrightarrow \varprojlim A_0/A_k\supset \prod_{k\in\N \setminus\{0\}}A_{k-1}/A_k\]
that if $M_n$ is countable, resp.\ countably dimensional, then the same is true for $\prod_{k\in\N \setminus\{0\}}A_{k-1}/A_k$. But the latter can only be true if $A_{k-1}/A_k$ are nontrivial for only finitely many $k$, proving the claim.
\end{proof}

\begin{lem}\label{lem:mlpluslimzero}
If an inverse system of abelian groups
\[\ldots \xrightarrow{f_5} M_4 \xrightarrow{f_4} M_3\xrightarrow{f_3} M_2\xrightarrow{f_2} M_1\xrightarrow{f_1} M_0\]
satisfies the Mittag-Leffler condition and $\varprojlim_mM_m=0$, then for each $n\in\N$ there is $N\geq n$ such that the homomorphism $M_N\to M_n$ vanishes.
\end{lem}
\begin{proof}
For each $m\in\N$ let $N_m\geq m$ be the natural number such that
\[M_m^s\coloneqq \im(M_{N_m}\to M_m)=\im(M_k\to M_m)\]
for all $k\geq N_m$. We may assume that the sequence $(N_m)_m$ is monotonously increasing.
Then each of the maps $M_{m+1}\to M_m$ maps $M^s_{m+1}$ surjectively onto $M^s_m$: if $x_m\in M^s_m$ is an arbitrary element, then it has a preimage $y_m$ under $M_{N_{m+1}}\to M_m$ and the map $M_{N_{m+1}}\to M^s_{m+1}$ maps $y_m$ to a preimage $x_{m+1}$ of $x_m$.

Therefore, starting with some $x_n\in M^s_n$, this procedure yields a sequence $(x_m)_{m\geq n}$ of preimages which constitute an element of $\varprojlim_mM_m$. But this limit is null by assumption, so in particular $x_n=0$.
This shows that
\[M_n^s=\im(M_{N_n}\to M_n)=0,\]
i.\,e.~the lemma holds with $N=N_n$.
\end{proof}

\begin{lem}\label{lemwfd4532345453}
Let $G$ be a finitely  generated group and $M$ a finitely generated abelian group or a finite-dimensional vector space. Then $\HX^k(G;M)=0$ and $\HX^{k+1}(G;M)=0$ together imply that $\HX^k(G;M)$ is uniformly trivial.
\end{lem}

\begin{proof}
As our space is a group, it obviously suffices to find for each $n\in\N$ an $N\geq n$ and construct the function $s$ such that \eqref{eq:def:unifloctriv} holds for the unit element $x=e\in G$. The other $x\in G\setminus\{e\}$ then follow by equivariance of the problem.

From our assumption on $M$, for each $m\in\N$ and $r\geq 0$, the simplicial cochain groups $C\!\Delta_c^k(U_m(\{e\},r);M)$ will be either finitely generated abelian groups or finite dimensional vector spaces, and hence the cohomology groups $H\!\Delta_c^k(U_m(\{e\},r);M)$ are also of this type. By continuity of the cohomology theory we have 
\[H\!\Delta_c^k(P_m(G);M)\cong \varinjlim_{r\in\N}H\!\Delta_c^k(U_m(\{e\},r);M)\,,\]
so these are countable abelian groups or countably dimensional vector spaces, respectively.

The assumptions $\HX^k(G;M) = \HX^{k+1}(G;M) = 0$ with the $\varprojlim^1$-sequence yield
$\varprojlim_m H\!\Delta_c^k(P_m(G);M)=0$ and $\varprojlim^1_m H\!\Delta_c^k(P_m(G);M)=0$.
Therefore we can apply Lemma \ref{lem:invlimabvs} and Lemma \ref{lem:mlpluslimzero} to obtain for each $n\in\N$ an $N\geq n$ such that the homomorphism $H\!\Delta_c^k(P_N(G);M)\to H\!\Delta_c^k(P_n(G);M)$ vanishes.

For $r\geq 0$ we choose a finite generating set $\{\phi_1,\dots,\phi_d\}$ for the cohomology $H\!\Delta_c^k(U_N(\{e\},r);M)$. From the diagram
\[\xymatrix{
H\!\Delta_c^k(U_N(\{e\},r);M)\ar[r]\ar[d]
&\varinjlim_{s\geq r}H\!\Delta_c^k(U_n(\{e\},s);M)\ar[d]^\cong
\\H\!\Delta_c^k(P_N(G);M)\ar[r]^{0}
&H\!\Delta_c^k(P_n(G);M)
}\]
we see that each $\phi_i$ is mapped to zero in some $H\!\Delta_c^k(U_n(\{e\},s_i);M)$, where $s_i\geq r$. Thus, the canonical map 
\[H\!\Delta_c^k(\cU_N(\{e\},r);M)\to H\!\Delta_c^k(\cU_n(\{e\},s(r));M)\]
vanishes for $s(r)\coloneqq \max\{s_1,\dots,s_d\}$.
\end{proof}

\begin{cor}\label{corrttrtrrt}
Let $G$ be a finitely generated group equipped with a coherent, expanding combing $H$ and let $G$ be of finite asymptotic dimension. Let $R$ be  either a unital ring whose underlying additive group structure is finitely generated or a field.
Then 
\[\cd_R(\partial_H G) + 1 = \max\{k\mid \HX^k(G;R) \not= 0\}.\]
\end{cor}

\begin{proof}
From Lemma~\ref{lemwfd4532345453} it follows that $\HX^i(G;R)$ is uniformly trivial for all $i > \max\{k\mid \HX^k(G;R) \not= 0\}$.  Therefore Theorem~\ref{thm345zrtewqd} then implies the inequality $\cd_R(\partial_H G) + 1 \le \max\{k\mid \HX^k(G;R) \not= 0\}$.

By Lemma~\ref{lem_transgressionisoforcontractible} the transgression map $T_H\colon \widetilde{\HA}^\ast(\partial_H G;R) \to \HX^{\ast+1}(G;R)$ is an isomorphism, which implies the reverse inequality.
\end{proof}

Theorem~\ref{thm_cdleqasdim}, which we stated at the beginning of this section, is now an easy consequence of the above corollary and the basic fact that for a space $Y$ we have $\HX^k(Y;R) = 0$ for all $k \ge \asdim(Y) + 2$.

\section{Constructing \texorpdfstring{$\boldsymbol{Z}$}{Z}-structures on groups}
\label{secnm0987678bmvmvmv}

Let a finitely generated, discrete $G$ admit a coherent and expanding combing. Our goal in this section is to show that we can put the combing corona on any suitable contractible space $X$ to get a compactification $\overline{X}$, which is a so-called $Z$-structure for $G$ (see Definition \ref{defn_Zstructures}). The reason why we want to do this is the following: up to now we have $\sigma$-compactified the Rips complex $\cP(G)$ of $G$ in order to get a contractible $\sigma$-compactification. But the space $\cP(G)$ is not finite-dimensional and not even metrizable. Putting the combing corona on a suitable finite-dimensional and metrizable space $X$ in order to get a contractible compactification has implications for the structure of the corona $\partial_H G$ (see, e.g., Corollary \ref{cor_coronadimension}).

\begin{defn}
A closed subset $Z$ of a topological space $\overline{X}$ is called a $Z$-set, if there is a homotopy $H\colon \overline{X}\times[0,1]\to \overline{X}$ such that $H_0$ is the identity and $H_t(\overline{X})\subset X\coloneqq \overline{X}\setminus Z$ for all $t>0$.
\end{defn}

In the proof of Theorem \ref{thm:RipsCompactificationContractible} we have actually shown that $\partial_HX$ is a $Z$-set in $\combcompb{\cP(Z)}{H}$ if $H$ is a coherent and expanding combing on the metric space $X$ and $Z\subset X$ is a discretization. 
For this particular example it might seem wise to demand that the homotopy in the definition of $Z$-sets is even a $\sigma$-map. However, we are only interested in using this notion for non-$\sigma$-spaces, so the above definition suffices.

Let us now introduce the necessary technical background for this section.

\begin{defn}
Let $G$ be a finitely generated, discrete group acting properly and cocompactly on a locally compact Hausdorff space $X$ from the left. Then the \emph{action coarse structure} on $X$ is the coarse structure generated by all the entourages $G\cdot (K\times K)$ with $K\subset X$ compact. 
\end{defn}
In particular, on can consider the action coarse structure on $G$ coming from the left action of $X$ onto itself. This coarse structure agrees with the coarse structure induced by the word metric associated to any choice of finite, symmetric generating set.

The proofs of the following two lemmas are elementary and will therefore be omitted.

\begin{lem}
The action coarse structure on $X$ turns $X$ into a proper topological coarse space.\qed
\end{lem}

\begin{lem}\label{lem_coarse_equiv_action}
For any $x_0\in X$, the map $\alpha\colon G\to X, g\mapsto g\cdot x_0$ is a coarse equivalence with respect to the action coarse structures on $G$ and $X$.

Any map $\beta\colon X\to G$ obtained in the following way is a coarse inverse to~$\alpha$: choose a bounded subset $B\subset X$ with $G\cdot B=X$ and set $\beta(x)\coloneqq g$ for a $g\in G$ with $x\in g\cdot B$. Any two such choices for $\beta$ are close to each other.\qed
\end{lem}

As a consequence of the previous lemma, if $H$ is a proper combing on $G$, then $H$ can be pulled back via $\alpha$ to a proper combing on $X$ and we obtain the combing compactification $\combcompb{X}{H}$ of $X$ with corona $\partial_HG$. It has the property that $\alpha$ extends to a continuous map $\overline{\alpha}\colon\combcompb{G}{H} \to\combcompb{X}{H}$ which is the identity on the corona $\partial_HG$.

\begin{lem}\label{lem_boundaryswapping}
Let $G$ be a finitely generated, discrete group acting properly and cocompactly on a locally compact Hausdorff space $X$ and let $H$ be a coherent and expanding combing on $G$.

If there exists a $G$-equivariant homotopy equivalence $\phi\colon X\to \cP(G)$, then there exists a continuous contraction $\tilde H\colon \combcompb{X}{H}\times[0,1]\to\combcompb{X}{H}$ with $\tilde H_0=\id$ and $\tilde H_t(\combcompb{X}{H})\subset X$ for all $t\in(0,1]$. Especially, $\partial_H X$ is a $Z$-set in $\combcompb{X}{H}$.
\end{lem}

Before proving the above lemma, let us remark two things: First, by a \emph{$G$-equivariant homotopy equivalence} we mean that also the homotopy inverse and the two homotopies to the identities are $G$-equivariant. 

Second, we note that it is irrelavant for the lemma whether we consider $\cP(G)$ as a $\sigma$-locally compact space or simply as a topological space, because the $G$-equivariance implies that all maps involved are automatically $\sigma$-maps. More precisely, if $K\subset X$ is compact and such that $G\cdot K=X$, then $\phi(K)$ is compact and therefore contained in a finite subcomplex of $\cP(G)$, implying that $\phi(X)=G\cdot \phi(K)$ is contained in some $P_n(G)$. Conversely, any $G$-equivariant homotopy inverse $\psi\colon\cP(G)\to X$ is trivially a $\sigma$-map, as is also any homotopy between $\psi\circ\phi$ and the identity. Last, but not least, a $G$-equivariant homotopy equivalence $\cP(G)\times[0,1]\to\cP(G)$ between $\phi\circ\psi$ and the respective identity maps each $P_m(G)\times[0,1]$ into some $P_{m'}(G)$, since $G$ acts cocompactly on $P_m(G)\times [0,1]$ and we argue now analogously as for $\phi$.

\begin{proof}
We choose a $G$-equivariant homotopy inverse $\psi\colon\cP(G)\to X$ and let $H'\colon X\times[0,1]\to X$ be any $G$-equivariant homotopy between $H'_0=\id$ and $H'_1=\psi\circ\phi$. Cocompactness of the $G$-actions together with $G$-equivariance of the maps implies that for all $n$ sufficiently large the diagrams of coarse maps, where $\alpha$ is the coarse equivalence from Lemma~\ref{lem_coarse_equiv_action},
\[\xymatrix{
G\ar[r]^-{\alpha}\ar[dr]_-{\mathrm{incl.}}&X\ar[d]^-{\phi}
&G\ar[dr]_-{\alpha}\ar[r]^-{\mathrm{incl.}}&P_n(G)\ar[d]^-{\psi}
&G\ar[rr]^-{\alpha\times\{t\}}\ar[drr]_-{\alpha}&&X\times[0,1]\ar[d]^-{H'}
\\&P_n(G)&&X&&&X
}\]
commute up to closeness. Therefore, the maps $\phi,\psi$ and $H'$ can be extended to continuous ($\sigma$-)maps
\[\overline{\phi}\colon\combcompb{X}{H}\to\combcompb{\cP(G)}{H}\,,\quad \overline{\psi}\colon\combcompb{\cP(G)}{H}\to\combcompb{X}{H}\,,\quad \overline{H'}\colon\combcompb{X}{H}\times[0,1]\to \combcompb{X}{H}\]
such that $\overline{\phi}$, $\overline{\psi}$ and $\overline{H'}_t$ for all $t\in[0,1]$ restrict to the identity on $\partial_HG$.

Let $H^\infty\colon \combcompb{\cP(G)}{H}\times[0,\infty]\to \combcompb{\cP(G)}{H}$ be the contraction of the $\sigma$-compact space $\combcompb{\cP(G)}{H}$ which we constructed in the proof of Theorem \ref{thm:RipsCompactificationContractible}, but with contraction parameter reparametrized to $[0,1]$, i.e., such that $H^\infty_0\equiv p$ and $H^\infty_1=\id$. Recall that this contraction maps $\combcompb{\cP(G)}{H} \times[0,1)$ into $\cP(G)$.

Finally, we choose a continous function $\rho\colon \combcompb{X}{H}\to [0,1)$ which vanishes on $\partial_H G$ and is strictly positive on $X$. Then we can define the desired contraction $\tilde H$ by the formula
\[\tilde H(x,t)\coloneqq \begin{cases}\overline{H'}(x,t/\rho(x))&t\leq \rho(x)\\\overline{\psi}\left(H^\infty\left(\overline{\phi}(x),\frac{1-t}{1-\rho(x)}\right)\right)&t\geq \rho(x)\,.\end{cases}\]
Note that despite the fraction $t/\rho(x)$ the homotopy is well defined, because $\overline{H'}(x,-)$ is constant for each $x\in\partial_HX$, and thanks to the fraction it has the property that $\tilde H_t(\combcompb{X}{H})\subset X$ for all $t\in(0,1]$. Furthermore, $\tilde H_1 \equiv \psi(p)$.
\end{proof}

The investigation of the following notion was started by Borsuk \cite{borsuk}. A thorough treatment of it can be found in, e.g., Hu \cite{hu}.

\begin{defn}
A separable metrizable space $X$ is called an \emph{absolute retract} (AR) if for all separable metrizable spaces $Y$ with closed subset $A\subset Y$ any continuous map $A\to X$ can be extended to a continuous map $Y\to X$.
\end{defn}

Note that one can define the notion of absolute retracts for a different class of spaces than separable, metrizable ones. We have chosen the class of separable, metrizable spaces since the next lemma needs results of Hanner \cite{hanner}, who works with this class of spaces for ARs.

Furthermore, with the choice of the class of separable, metrizable spaces for the definition of ARs, a compact space is a Euclidean retract if and only if it is a finite-dimensional absolute retract (use Hurewicz--Wallman \cite[Theorem V.2]{hurewicz_wallman}). Here a space $X$ is being called a \emph{Euclidean retract} (ER) if it is homeomorphic to a retract of Euclidean space $\IR^n$ for some $n \in \IN$.

\begin{lem}[Guilbault--Moran {\cite[Lemma 3.3]{guilbault_moran_Z}}]
\label{lem_compactificationAR}
If $Z\subset \overline{X}$ is a $Z$-set in a compact, metrizable space $\overline{X}$ and $X\coloneqq \overline{X}\setminus Z$ is an AR, then $\overline{X}$ is also an AR.
\end{lem}

The following definition is essentially due to Bestvina \cite{bestvina}:

\begin{defn}\label{defn_Zstructures}
Let $G$ be a finitely generated, discrete group.

A pair of spaces $(\overline{X},Z)$ such that $\overline{X}$ is a compact, metrizable space is called a \emph{$Z_\ER^\free$-structure} for $G$ if
\begin{enumerate}
\item $\overline{X}$ is a Euclidean retract,
\item $Z$ is a $Z$-set in $\overline{X}$,
\item $G$ acts properly, cocompactly and freely on $X \coloneqq  \overline{X} \setminus Z$, and
\item\label{wdfe54} $\overline{X}$ is a Higson dominated compactification of $X$, where we equip $X$ with the action coarse structure.
\end{enumerate}
We do not demand that the action of $G$ extends continuously to $\overline{X}$.

Dranishnikov \cite{dranish_BM} considered a variation: a \emph{$Z_\AR$-structure} for $G$ is a pair $(\overline{X},Z)$ as above but $\overline{X}$ is just demanded to be an absolute retract and $G$ acts only properly and cocompactly on $X$.

Analogously we define the variations \emph{$Z_\ER$-structures} and \emph{$Z_\AR^\free$-structures}.
These variations were also considered by Moran \cite[Page 6]{moran}.
\end{defn}

\begin{rem}\label{rem2345ztrwe}
Since $G$ acts cocompactly and properly on $X$, we may equivalently replace Condition~\ref{wdfe54} of the above definition by the so-called \emph{nullity condition}: for every open cover $\cU$ of $\overline{X}$ and every compact subset $K \subset X$ all but finitely many $G$-translates of $K$ are $\cU$-small.
\end{rem}

In the following, in order not to burden the text with complicated phrases, if we just say ``$Z$-structure'' then we mean one of the versions defined above. Which one should be clear from either the context or the cited source.

For $Z$-structures on finitely generated groups Bestvina \cite[Lemma~1.4]{bestvina} proved the following boundary swapping result: if the classifying space $BG$ is finite, then we can ``$Z$-compactify'' $EG$ by the $Z$-boundary of a $Z$-structure for $G$. This was generalized by Moran \cite[Appendix~A]{moran} to the infinite-dimensional case, and further generalized by Guilbault--Moran \cite{guilbault_moran_Z}.

\begin{thm}\label{thm13243wer}
Let $G$ be a finitely generated discrete group acting properly and cocompactly on an AR (resp.\ ER) $X$ and let $H$ be a coherent and expanding combing on $G$. Suppose that we have a $G$-equivariant homotopy equivalence $f\colon X \to \cP(G)$.

Then $(\combcompb{X}{H},\partial_H G)$ is a $Z_{\AR}$-structure (resp.\ $Z_{\ER}$-structure) for $G$.
\end{thm}

\begin{proof}
We apply Lemma \ref{lem_boundaryswapping} to conclude that the combing compactification $\combcompb{X}{H}$ of $X$ is contractible and such that $\partial_H G$ is a $Z$-set in it. It remains to show that $\combcompb{X}{H}$ is an AR (resp.\ an ER).

That $\combcompb{X}{H}$ will be an AR, if $X$ is an AR, is Lemma \ref{lem_compactificationAR}. Let $X$ be an ER and we have to show that $\combcompb{X}{H}$ is an ER, too. We know already that $\combcompb{X}{H}$ is an AR, and therefore we have to show that $\combcompb{X}{H}$ is finite-dimensional. A short proof of the fact that the Lebesgue dimension of $\combcompb{X}{H}$ is equal to the Lebesgue dimension of $X$ was given by Ivanov \cite{ivanov_MO} on MathOverflow. However, since the latter is not a peer reviewed journal, we decided to recite the relevant part of Ivanov's posting here.

Let $\{U_i\}$ be an open covering of $\combcompb{X}{H}$ and denote by $\rho$ its Lebesgue number with respect to some metric $\bar d$ on $\combcompb{X}{H}$. As  $\combcompb{X}{H}$ is compact, the homotopy $\tilde H$ from Lemma \ref{lem_boundaryswapping} is uniformly continuous with respect to $\bar d$ and in particular there exists $t>0$ such that $\bar d(x,\tilde H_t(x))<\frac{\rho}{3}$ for all $x\in\combcompb{X}{H}$. Hence, if $\{V_j\}$ is an open covering of $X$ by sets $V_j$ of $\bar d$-diameter less than $\frac{\rho}3$, then each of the sets $\tilde H_t^{-1}(V_j)$ has diameter less than $\rho$ and is thus contained in one of the sets $U_i$. Assuming that $X$ has Lebesgue dimension $N$, there exists a refinement $\{W_k\}$ of $\{V_j\}$ of multiplicity at most $N+1$. It is now easy to see that $\{\tilde H_t^{-1}(W_k)\}$ is a refinement of the covering $\{U_i\}$ of $\combcompb{X}{H}$ which also has multiplicity at most $N+1$.
\end{proof}

\begin{rem}
In order to prove the previous theorem, we could have also followed the arguments given by Moran in her proof of a related boundary swapping result \cite[Theorem~A.1]{moran}. In this case we would have defined the topology on $X \cup \partial_H G$ ad hoc in the following way:

We define
\[\bar f \coloneqq  f \cup \id_Z \colon X \cup \partial_H G \to \combcompb{\cP(G)}{H}\]
to be the identity on $\partial_G H$ and to be $f$ on $X$. Then we define the topology of $X \cup \partial_H G$ to be the one generated by the open subsets of $X$ and all the sets of the form ${\bar f}^{-1}(U)$ for open subsets $U \subset \combcompb{\cP(G)}{H}$.

One can check that this topology on $\combcompb{X}{H}$ coincides with the one resulting from our Lemma \ref{lem_boundaryswapping}.
\end{rem}

As a first corollary from Theorem~\ref{thm13243wer} we get the following construction theorem for $Z$-structures for $G$. Recall that $\EG$ denotes the classifying space for proper actions, and that $\EG = EG$ if $G$ is torsion-free.

\begin{cor}\label{cor243terwe}
Let $G$ be equipped with an expanding and coherent combing~$H$ and let $G$ admit a $G$-finite model for $\EG$.

\begin{enumerate}
\item Then $\big(\combcompb{\,\EG\;}{H},\partial_H G \big)$ is a $Z_\ER$-structure for $G$.
\item If $G$ is torsion-free, then $(\combcompb{EG}{H},\partial_H G)$ is a $Z_\ER^\free$-structure for $G$.
\end{enumerate}
\end{cor}

\begin{proof}
By assumption, $\EG$ is a Euclidean retract on which $G$ acts properly and cocompactly. Its Rips complex $\cP(G)$ is a model for $\EG$ and hence we get a $G$-equivariant homotopy equivalence $\EG \to \cP(G)$. Theorem~\ref{thm13243wer} now implies that $\big(\combcompb{\,\EG\;}{H},\partial_H G \big)$ is a $Z_\ER$-structure for $G$. If $G$ is torsion-free, then we just use the fact that in this case $\EG = EG$ and $G$ acts freely on it.
\end{proof}

Bestvina--Mess \cite[Proposition 2.6]{bestvina_mess} proved that if the pair $(\overline{X},Z)$ is a $Z_\ER^\free$-structure for $G$, then $\dim(Z) < \dim(X)$ and therefore the boundary $Z$ is finite-dimensional (see also Guilbault--Tirel \cite[Theorem 1.1]{guilbault_tirel}). Due to the boundary swapping result \cite[Lemma 1.4]{bestvina}, Bestvina \cite[Theorem 1.7]{bestvina} then proved that if $(\overline{X},Z)$ is a $Z_\ER^\free$-structure for $G$, then $\dim(Z) + 1 = \cd(G)$, where $\cd(G)$ is the cohomological dimension of~$G$.

Moran \cite[Corollary 3.2]{moran_published} generalized the finite-dimensionality statement to $Z_\AR$-structures, i.e., if $(\overline{X},Z)$ is a $Z_\AR$-structure, then $\dim(Z) < \infty$. She also noticed \cite[Remark~1 on Page~8]{moran} that the inequality $\dim(Z) < \dim(X)$ holds whenever $Z$ is a $Z$-set in a compact metric space $\overline{X}$.

Combining the above with Corollary~\ref{cor243terwe} we therefore get the following:

\begin{cor}\label{cor_coronadimension}
Let $H$ be an expanding and coherent combing on a finitely generated, discrete group $G$.
\begin{enumerate}
\item If $G$ admits a finite model for its classifying space $BG$, then
\[\dim(\partial_H G) + 1 = \cd(G).\]
\item If $G$ admits a $G$-finite model for $\EG$, then
\[\dim(\partial_H G) < \underline{\smash{\mathrm{gd}}}(G),\]
where the $\underline{\smash{\mathrm{gd}}}(G)$ denotes the least possible dimension of a $G$-finite model for the classifying space for proper actions $\EG$.
\end{enumerate}
\end{cor}

As another corollary from Theorem~\ref{thm13243wer} we get the following boundary swapping result:

\begin{cor}\label{cor_boundaryswap}
let $G$ be equipped with an expanding and coherent combing~$H$, let $G$ be torsion-free, and let $(\overline{X},Z)$ be a $Z_\AR^\free$-, resp.\ $Z_\ER^\free$-structure for $G$.

Then $(\combcompb{X}{H},\partial_H G)$ is also a $Z_\AR^\free$-, resp.\ $Z_\ER^\free$-structure for $G$.
\end{cor}

\begin{proof}
Follows directly from Theorem~\ref{thm13243wer} by noting that under the stated assumptions both $X$ and $\cP(G)$ are models for $EG$ which provides the needed $G$-equivariant homotopy equivalence $X \to \cP(G)$.
\end{proof}

\begin{rem}\label{rem24356545435459}
Assume that additionally to the assumptions in the above statements we require that the combing $H$ on $G$ is coarsely equivariant. Then the $Z$-structures we construct are, in fact, even $\EZ$-structures, meaning that the $G$-action extends continuously to the corona. The reason for this is that in Lemma \ref{lem_boundaryswapping} we pull-back the combing from $G$ to $X$ by the action of $G$ on $X$ and hence the pulled-back combing is still coarsely equivariant.
\end{rem}

The above boundary swapping results enable us to prove a certain uniqueness result for the combing corona.

\begin{lem}
Let $G$ be a finitely generated and discrete group admitting a $G$-finite model for $\EG$. Then the combing coronas of any two expanding and coherent combings on $G$ are shape equivalent.

If additionally $G$ is torsion-free, then the combing coronas are also shape equivalent to the $Z$-boundary of any $Z_\ER^\free$-structure for $G$.
\end{lem}

\begin{proof}
The first statement follows from combining Corollary \ref{cor243terwe} with a result of Guilbault \cite[Corollary 3.8.15]{guilbault}. The addendum follows from Corollary \ref{cor_boundaryswap} in conjunction with the corresponding shape equivalence result of Bestvina \cite[Proposition 1.6]{bestvina}.
\end{proof}

\section{Final remarks and open questions}

In this final section we compile some questions that arise out of the present paper, and some of these questions are combined with some thoughts of us.

\subsection*{Examples of expandingly combable spaces and groups}

We have seen in Section~\ref{sec_examples} that hyperbolic and $\CAT(0)$ spaces admit expanding and coherent combings. But there is another class of spaces that one might count as a basic class of non-positively curved spaces: the hierarchically hyperbolic spaces introduced by Behrstock--Hagen--Sisto \cite{HHS_1,HHS_2}.

\begin{question}\label{Q_HHS}
Do hierarchically hyperbolic spaces admit coherent and expanding combings?
\end{question}

Note that hierarchically hyperbolic groups admit nice boundaries \cite{HHS_boundaries} which generalize the Gromov boundary of hyperbolic groups, giving evidence for a positive answer to the above question.\footnote{After we uploaded our preprint to the arXiv Haettel--Hoda--Petyt \cite{HHP} solved Question \ref{Q_HHS} affirmatively: a hierarchically hyperbolic spaces (if it is proper and finite-dimensional) is coarsely convex and hence admits a coherent and expanding combing. But it is open whether the boundary constructed by Durham--Hagen--Sisto \cite{HHS_boundaries} is homeomorphic to the resulting combing corona.}

We have argued in Section~\ref{secjkjbjnbbbm} that the automatic structure on central extensions of hyperbolic groups, which was constructed by Neumann--Reeves \cite{neumann_reeves_2,neumann_reeves}, is not expanding. But maybe one can construct a different combing on central extensions of hyperbolic groups, which will be expanding:

\begin{question}\label{question_central_extensions}
Do central extensions of hyperbolic groups admit coherent and expanding combings?
\end{question}

Note that the above question is non-trivial by the following related example: the $3$-dimensional integral Heisenberg group is a central extension of $\IZ \times \IZ$ by~$\IZ$. But it is known that it has a cubic Dehn function (\cite[Sec.~8.1]{epstein_et_al} or \cite[Rem.~5.9]{gersten_heisenberg_group}) and therefore can not admit any quasi-geodesic combing. Hence a positive answer to the above Question~\ref{question_central_extensions} must either strongly exploit the hyperbolicity or it must be such that the constructed combings are not quasi-geodesic (even though the ones we start with are quasi-geodesic).\footnote{Elia Fioravanti solved this question affirmatively. His argument is as follows. Central extensions of a group $G$ by an abelian group $A$ are classified by elements of $H^2(G;A)$ \cite[Ch.~4]{brown}. Now if $G$ is hyperbolic, then any cohomology class of $G$ may be represented by a bounded cocycle \cite{mineyev}. But if an extension $A \to E \to G$ is classified by a bounded cocycle, then $E$ is quasi-isometric to the product $A \times G$ \cite[Thm.\ 3.1]{gersten_bounded_coho}. Therefore, if $A$ and $G$ admit coherent and expanding combings, then so does $E$.}

Bestvina \cite[Examples 1.2(iv)]{bestvina} shows that the Baumslag--Solitar group $BS(1,2)$ admits a $Z$-structure and argues how one can get it as the visual boundary of $\CAT(0)$-spaces. For the other Baumslag--Solitar groups $BS(m,n)$ a construction of $Z$-boundaries was recently provided by Guilbault--Moran--Tirel \cite{boundaries_BS_groups}.
It is a natural question if these $Z$-boundaries come from  expanding and coherent combings on the Baumslag--Solitar groups.
\begin{question}
Can one adjust the construction of Guilbault--Moran--Tirel \cite{boundaries_BS_groups} to produce coherent and expanding combings on Baumslag--Solitar groups?
\end{question}

Note that since Baumslag--Solitar groups have exponential Dehn functions, any potential combing on them must have exponentially long combing paths.

\subsection*{Cones over the combing corona}

Fukaya--Oguni \cite{fukaya_oguni} showed that a coarsely convex space is coarsely homotopy equivalent to the Euclidean cone over its boundary. This implies the coarse Baum--Connes conjecture for the space, which is of course a stronger statement than our injectivity result (Theorem \ref{thm243reds}) for expandingly and coherently combable spaces. The natural question is now if one can generalize the result of Fukaya--Oguni to our situation.

\begin{question}
Under which additional conditions (preferably more general than coarse convexity and in the spirit of the present paper) on a coherent and expanding combing $H$ on $X$ can one prove that $X$ is coarsely homotopy equivalent to the Euclidean cone over the combing corona $\partial_H X$?
\end{question}

\subsection*{Proper and coherent combings}

In this paper we introduced the properness condition for combings allowing us to construct the combing compactification. But currently we do not know much about the difference between combable spaces and properly combable spaces. In Example \ref{ex344wret} we constructed a combable group, where the combing is not proper. But the group we used (i.e., $\IZ$) does of course admit a proper combing. So the question here is the following one:

\begin{question}\label{ques3434erteew}
Is there a combable space not admitting a proper combing?
\end{question}

We have seen in Lemma \ref{lem_coherent_is_proper} that a coherent combing is always proper, and in Example~\ref{example_non_coherent} we have given an example of a proper combing which is not coherent. But a reparametrization of the combing paths in this example makes the combing coherent. We actually do not know how to obstruct the existence of coherent combings if we want to allow proper ones to exist.

\begin{question}
Does there exist a space admitting a proper combing, but not a coherent one?
\end{question}

\subsection*{Automatic groups}

An interesting and large class of groups with coherent combings are automatic groups. But many of these automatic structures are not expanding. The question is now if this is a general feature of automatic structures (i.e., that they are usually not exanding) or if it just happened that many of the currently known automatic structures are not expanding but that there are actually also many expanding automatic structures out there. If we do find many more automatic structures that are expanding, it would be nice to have a criterion for expandingness for automatic structures that is akin to the combinatorial/algorithmic nature of automaticity:

\begin{question}\label{ques354erweer}
Is there a combinatorial or algorithmic condition that one can impose on automatic structures such that they will be expanding?
\end{question}

Of course, there is also the question if there is an automatic group that can not admit any expanding combing at all:

\begin{question}\label{ques354erweer2}
Does there exists an automatic group which does not admit an expanding (and coherent) combing?
\end{question}

In Section~\ref{secjkjbjnbbbm} we argued that many of the examples of automatic structures that we know induce in general non-expanding combings. From the list of automatic groups we have given, we have not discussed Artin groups of finite type, not mapping class groups, and not groups acting nicely on buildings of certain types. So the following questions are immediate:

\begin{question}\label{quesnrt23er}
Are the automatic structures on
\begin{enumerate}
\item Artin groups of finite type,
\item mapping class groups, or
\item groups acting geometrically and in an order preserving way on Euclidean buildings of type $\tilde A_n$, $\tilde B_n$ or $\tilde C_n$
\end{enumerate}
expanding? If yes, how do the boundaries look like?
\end{question}

\subsection*{Existence of $Z$-structures}

Bestvina \cite[Section 3.1]{bestvina} asked whether every group admitting a finite classifying space admits a $Z$-structure. This question is still open (see also Guilbault \cite{guilbault_weak_Z} for some recent results about weak $Z$-structures). In the Corollary~\ref{cor243terwe} we have seen that a coherent and expanding combing on the group can be used to construct a $Z$-structure and almost all of the  currently known examples of $Z$-structures arise in this way.

\begin{question}
Do there exist groups that
\begin{enumerate}
\item admit $Z$-structures which are not coronas of proper combings, or
\item admit $Z$-structures but do not admit any proper combings?
\end{enumerate}
\end{question}

A possible way to get groups with $Z$-structures but not proper combings could be to exploit the fact that the existence of a combing on a group implies that its Dehn function is at most exponential. Now groups with super-exponentially growing Dehn functions are known and hence such groups are not combable. The question is whether such groups admit $Z$-structures.

\subsection*{Dimension of the combing corona}

For groups with $Z$-structures, Bestvina--Mess \cite{bestvina_mess} proved Corollary~\ref{corrttrtrrt} for the $Z$-boundary without assuming that $G$ has finite asymptotic dimension. Moreover, Guilbault--Moran \cite{guilbault_moran} proved the estimate of Theorem~\ref{thm_cdleqasdim} for the covering dimension of the $Z$-boundary, which is a much stronger statement (because P.~S.~Alexandrov proved that if the covering dimension of a compact space is finite, then it coincides with the cohomological dimension). Hence our main results from Section~\ref{sec45342345} are weakenings of the results of Bestvina--Mess and of Guilbault--Moran, but they apply to the combing corona and even if the group does not admit any $Z$-structure.

The arguments of Bestvina--Mess need that the contractible space which is $Z$-compactified has finite dimension, whereas the arguments of Guilbault--Moran use a metric on the space. The Rips complex $\cP(G)$ that we have to use in our arguments lacks both these properties, which explains why we only got weaker results. But if $G$ admits a $G$-finite model for $\EG$, then the boundary swapping results from Section~\ref{secnm0987678bmvmvmv} show that the results of Bestvina--Mess and Guilbault--Moran are applicable to the combing compactification in this case. So the question is now of course whether the assumption of having a $G$-finite model for $\EG$ is really necessary:

\begin{question}\label{quesdght453}
Let $G$ be equipped with an expanding and coherent combing.
\begin{enumerate}
\item Is Corollary~\ref{corrttrtrrt} true without assuming that $G$ has finite asymptotic dimension?
\item Can we strengthen the estimate of Theorem~\ref{thm_cdleqasdim} to the covering dimension of the combing corona?
\end{enumerate}
\end{question}

We can of course ask the same questions in the general case of spaces instead of just groups. But here we do not even have a result like Corollary~\ref{corrttrtrrt} available. Although Theorem~\ref{thm345zrtewqd} holds for spaces (and not only for groups), it is Lemma~\ref{lemwfd4532345453} which we could prove only for groups. To prove a corresponding statement for spaces one would probably need additional assumptions. Maybe being coarsely convex in the sense of Fukaya--Oguni would suffice. Evidence for this comes from the fact that combings can be used to give quantitative bounds on Dehn functions and therefore on the vanishing of homology classes; see, e.g., \cite{ji_ramsey,meyer,ogle_pol,engel_BSNC}. But one should keep in mind that there are examples due to Swenson \cite{swenson} of hyperbolic spaces which have an infinite-dimensional boundary, i.e., one should be cautious when trying to answer the following question:

\begin{question}
Under which conditions on a metric space $X$ can we prove an analogue of Lemma~\ref{lemwfd4532345453} for $X$ and hence get Corollary~\ref{corrttrtrrt} for $X$?
\end{question}

\subsection*{Dynamic properties of the combing corona}

Assume we are given a space $X$ which comes equipped with a proper combing $H$ and a group $G$ acting on $X$ such that $H$ is coarsely equivariant. The latter means that Lemma~\ref{lem24354terwr} is applicable and therefore the action of $G$ extends to a continuous action by homeomorphisms on the combing corona $\partial_H X$.

This enables us now to study the dynamics of this action on the combing corona, and provides us therefore with a plethora of new questions.

A possible question is whether such boundary actions are amenable. A positive answer would mean that the groups are exact since they act amenably on compact Hausdorff spaces. But in this generality the answer is negative: the set $P(n,\IR)$ of all positive-definite, symmetric $(n\times n)$-matrices with real coefficients is a non-positively curved Riemannian manifold, i.e., a $\CAT(0)$-space. The group $\mathit{GL}(n,\IR)$ acts transitively by isometries on it and $P(n,\IR)$ can be identified as the homogeneous space $\mathit{GL}(n,\IR) / O(n,\IR)$. Now there exist points at infinity of $P(n,\IR)$ which are stabilized by the whole group $\mathit{GL}(n,\IR)$, see \cite[Ex.~II.10.6.3]{bridson_haefliger}. Hence the action is not amenable since it has non-amenable point-stabilizers.

So the question is whether there are any sufficient conditions such that we get an amenable action:

\begin{question}
Under which conditions on $X$, on a combing $H$ on it, and on the group $G$ can we prove that the action of $G$ on $\partial_H X$ is amenable?
\end{question}

Amenable actions arising as above have nice properties, like the following one: suppose that $G$ comes with an expanding, coherent and coarsely equivariant combing $H$ on it. Since the corona $\partial_H G$ is small at infinity and since the left-action of $G$ on itself extends to it, one can prove that the right-action of $G$ on itself extends trivially to $\partial_H G$. Assuming now that the left-action of $G$ on $\partial_H G$ is amenable, we get that the group is bi-exact \cite[Sec.~3]{houdayer_isono} and hence its group von Neumann algebra is solid \cite{ozawa_solid,ozawa_kurosh}.

By \cite[Thm.\ 5.8]{anantharaman_delaroche} we know that $G$ acts amenably on $\partial_H X$ if and only if the $C^\ast$-algebra $C(\partial_H G) \rtimes_r G$ is nuclear (note that the 'Property (W)' occuring in \cite[Thm.\ 5.8]{anantharaman_delaroche} is always satisfied for discrete groups $G$ by \cite[Ex.\ 4.4]{anantharaman_delaroche}).
In the case of a hyperbolic group $G$ acting on its Gromov boundary $\partial G$, the $C^\ast$-algebra $C(\partial_H G) \rtimes_r G$ is already known to be nuclear: by Anantharaman-Delaroche \cite{anantharaman_delaroche} and Laca--Spielberg \cite{laca_spielberg} we know that $C(\partial_H G) \rtimes_r G$ is a Kirchberg algebra, and such algebras are known to have finite nuclear dimension (actually, they have nuclear dimension exactly $1$; Ruis--Sims--S{\o}rensen \cite{ruis_sims_sorensen} and Bosa et al.\ \cite{BBSTWW}) which implies nuclearity.

\newcommand{\etalchar}[1]{$^{#1}$}
\providecommand{\bysame}{\leavevmode\hbox to3em{\hrulefill}\thinspace}
\providecommand{\MR}{\relax\ifhmode\unskip\space\fi MR }
\providecommand{\MRhref}[2]{\href{http://www.ams.org/mathscinet-getitem?mr=#1}{#2}
}
\providecommand{\href}[2]{#2}

\bibliographystyle{amsalpha}

\end{document}